\documentclass[11pt, DIV10,a4paper]{article}

\usepackage{natbib}
\newcommand {\ctn}{\citet} 
\usepackage{graphicx,subfigure,amsmath,latexsym,amssymb}
\usepackage{float,epsfig,multirow,rotating,times}
\usepackage{upgreek,wrapfig}
\usepackage[colorlinks,breaklinks,bookmarksopen,bookmarksnumbered]{hyperref}
\usepackage{comment}
\usepackage{url}
\usepackage{gensymb}

\newtheorem{theorem}{Theorem}

\newtheorem{remark}[theorem]{Remark}

\newenvironment{proof}[1][Proof]{\textbf{#1.} }{\ \rule{0.5em}{0.5em}}

\numberwithin{equation}{section}
\numberwithin{algo}{section}
\numberwithin{table}{section}
\numberwithin{figure}{section}

\usepackage{setspace}
\usepackage[mathscr]{euscript}
\usepackage[margin=1in]{geometry}
\begin{document}


\title{\vspace{-0.8in}
{\bf Bayesian Appraisal of Random Series Convergence with Application to Climate Change}}
\author{Sucharita Roy and Sourabh Bhattacharya\thanks{
Sucharita Roy is an Assistant Professor and Head of the Department of Mathematics in St. Xavier's College, Kolkata, pursuing PhD 
in Interdisciplinary Statistical Research Unit, Indian Statistical Institute, 203, B. T. Road, Kolkata 700108.
Sourabh Bhattacharya is an Associate Professor in Interdisciplinary Statistical Research Unit, Indian Statistical
Institute, 203, B. T. Road, Kolkata 700108.
Corresponding e-mail: sourabh@isical.ac.in.}}
\date{\vspace{-0.5in}}
\maketitle%

\begin{abstract}

        Convergence analysis of infinite series constitutes a very long tradition in classical mathematics. Even then, for most infinite series, none of the existing
        methods of convergence analysis succeeds in providing conclusive answers. \ctn{Roy17} attempt to rise to this challenge by providing Bayesian characterization
        of infinite series with respect to their convergence properties and demonstrate quite successful applications in a variety of deterministic infinite series
        where the convergence properties are either known or unknown. Their most important application, namely, to the Dirichlet series characterizing the (in)famous
        Riemann Hypothesis, revealed insights that are not in support of the most celebrated conjecture for over $150$ years.

	In contrast with deterministic series considered by \ctn{Roy17}, in this article we take up random infinite series for our investigation. 
	Remarkably, our method does not require any simplifying assumption, such as independence or restrictive dependence among the random variables.
	Albeit the Bayesian characterization theory for random series is no different from that for the deterministic setup,  
	construction of effective upper bounds for partial sums, required for implementation, turns out to be a challenging undertaking in the random setup. 
	The difficulty steps in as the consequence of non-availability of the functional forms of the random summands of the series, and the problem
	persists even if the distributions of the summands are assumed to be known.

	In this article, we first construct parametric upper bound forms assuming parametric densities of the random summands. 
	But despite their mathematical validity for non-negative summands and good performance in such setups, they are not generally applicable, which leads
	us to propose a flexible bound for general setups. But even for series driven by normal distributions, the general bound exhibits   
	correct but very inefficient and less persuasive convergence analysis. 
	Moreover, application to random Dirichlet series yields wrong answers in many cases.
	Hence, we propose a general nonparametric bound structure, borrowing ideas from \ctn{Roy20}. Simulation studies demonstrate
	high accuracy and efficiency of the nonparametric bound in all the setups that we consider. 

	Finally, exploiting the property that the summands tend to zero in the case of series convergence, we consider application of our nonparametric bound driven Bayesian
	method to global climate change analysis. 
	Specifically, analyzing the global average temperature record over the years $1850-2016$ and Holocene global average temperature 
	reconstruction data $12,000$ years before present, we conclude, in spite of the current global warming situation, that global climate dynamics 
        is subject to temporary variability only, the current global warming being an instance, and long term global warming or cooling either in the past or in the future,
	are highly unlikely.
\\[2mm]
\end{abstract}

\section{Introduction}
\label{sec:intro}

Convergence assessment of deterministic infinite series is a part of basic mathematical analysis and is included in the curriculum of almost all schools and colleges. 
Yet, for most infinite series there still does not exist any test of convergence that can provide conclusive answers, an issue that has concerned among many, 
the first author of this article, the head of the department of Mathematics in St. Xavier's College, Kolkata.
In response to her informal question if the Bayesian paradigm is powerful enough to even attempt answering such questions of convergence, 
\ctn{Roy17} indeed came up with a novel Bayesian procedure
to address questions of series convergence. Their key idea is to embed the underlying infinite series, even if deterministic, in a random, stochastic process framework,
and then to build a recursive Bayesian algorithm for inference regarding the probability of convergence. They proved that  
the Bayesian algorithm converges to $1$ 
if and only if the underlying series converges and to $0$ if and only if the series diverges. Oscillatory series with multiple limit points, including
infinite number of limit points, are also treated under similar Bayesian recursive frameworks by the authors, with proper Bayesian characterizations of their properties.
Applications of their Bayesian method to a variety of infinite series yielded very encouraging results, and answers were obtained even where all existing methods
of convergence assessment failed.

Although convergence assessment of infinite series constitutes a part of elementary mathematical analysis, it also holds the key to the solution of the most notorious
unsolved problem of mathematics, namely, the Riemann Hypothesis. Establishment of convergence of the Dirichlet series
for the M\"{o}bius function, for the real part of a complex-valued parameter of the series exceeding $1/2$, would establish truth of Riemann Hypothesis. On the other hand,
divergence of the series for even any particular value of the real part exceeding $1/2$ would negate the famous conjecture. On careful application of their Bayesian method
to the Dirichlet series, \ctn{Roy17}, to their utter surprise, found that the truth of Riemann Hypothesis is not supported by their Bayesian procedure. 

In this article, we shall concern ourselves with random series of the form $\sum_{i=1}^{\infty}X_i$, where $X_i$ are random, not deterministic quantities as in the
examples in \ctn{Roy17}.
Now recall that the Bayesian procedure of \ctn{Roy17} treats even the deterministic elements of the series as realizations of some stochastic process. Hence, when the elements
of the infinite series are random themselves, then there is certainly no need for any new theory for studying random series convergence. But although
no new general theory is required, there are important details to pay attention to. The main issue is that, in the case of deterministic infinite series, the functional forms
of the series elements are known, which \ctn{Roy17} usefully exploited to construct bounds for the partial sums associated with the series. However, in the case of random
series elements, the functional forms are unavailable. In fact, even the distributional forms of the series elements are not available in reality, and if they are assumed
to be available for the sake of theoretical development, construction of bounds for the partial sums in general, is still highly non-trivial. 

Our main contribution in this article is to create appropriate bounds for the partial sums in the context of random infinite series. 
We begin with creation of upper bounds in parametric setups, whose mathematical validity is ensured for summands with non-negative supports. Simulation
experiments under several such setups corroborate much accuracy and efficiency of such upper bounds when employed in our Bayesian procedure.
However, since these bounds are not generally applicable, we propose a flexible parametric upper bound structure, although its mathematical validity in general situations
can not be guaranteed. Although the general bound works well in several setups with non-negatively supported summands, its performance in random series driven by 
hierarchical normal distributions has been very inefficient and less persuasive, in spite of correct indications of convergence and divergence. 
Furthermore, in the case of random Dirichlet series, the general parametric bound yields wrong answers in many cases.
Hence, borrowing
ideas from \ctn{Roy20}, we propose a nonparametric upper bound for the partial sums. 
The bound does not require any distributional assumption or non-negativity and improves itself adaptively with the iterations of the recursive Bayesian procedure. Simulation 
experiments demonstrate that not only is this bound far more accurate and efficient than the general parametric bound, but is also very much comparable in performance
with the mathematically valid parametric bounds in the relevant non-negative setups.

Now, investigation of general series convergence, either deterministic or random, may be mathematically or probabilistically extremely challenging and hence
makes for commendable undertaking, but such efforts would be more fruitful if determination of series convergence properties can be related to
solutions of scientific problems of much broader interest and importance. In this regard, the efforts of \ctn{Roy17} did not seem to go in vain, as their 
novel Bayesian procedure for general deterministic series convergence assessment led to surprisingly important insights regarding the most challenging but
influential unsolved problem of mathematics, the Riemann Hypothesis. Random infinite series seems to be more abstruse compared to deterministic ones as 
it is not immediately clear if they can be related to scientific problems of broad importance. In this article, we attempt to relate investigation of
convergence properties of random infinite series to important scientific questions on climate change. Specifically, we attempt to address if global warming 
will continue or if global temperature will stabilize in the future. We also attempt to learn if global temperature was stable in the past
or if there were instances of long periods of global warming or cooling. 
Based on records of current global temperature data and palaeoclimate reconstruction data, we infer with our Bayesian recursive procedure in conjunction with 
the nonparametric bound for the partial sums that we propose, that climate dynamics is subject to temporary variations, and long-term
global warming or cooling is unlikely in the past as well as in the future.

The rest of our article is structured as follows.
First, in Section \ref{sec:recursive1}, we provide an overview of the recursive Bayesian procedure introduced by \ctn{Roy17} for characterizing convergence properties
of general deterministic or random infinite series. Then, in Section \ref{sec:parametric}, we put in our efforts towards building parametric 
upper bounds for partial sums of random series and in Section \ref{sec:simstudy1}
assess the performance of such parametric bound structure with simulation experiments.
We propose the nonparametric bound structure in Section \ref{sec:nonpara} and evaluate its performance with simulation studies in the same section.
Using the proposed nonparametric bound structure we analyze past and future global climate change in Section \ref{sec:gw}. Finally, in Section \ref{sec:summary}
we summarize our contributions and provide relevant discussions.

\section{Overview of the recursive Bayesian procedure for infinite series}
\label{sec:recursive1}

\subsection{Stage-wise likelihoods}
\label{subsec:Bayesian_method}
Letting $\{X_i\}_{i=1}^{\infty}$ denote some stochastic process, for $j=1,2,3,\ldots$, let
\begin{equation*}
S_{j,n_j}=\sum_{i=\sum_{k=0}^{j-1}n_k+1}^{\sum_{k=0}^jn_k}X_i,
\end{equation*}
where $n_0=0$ and $n_j\geq 1$ for all $j\geq 1$.
Also let $\{c_j\}_{j=1}^{\infty}$ be a non-negative decreasing sequence and
\begin{equation*}
Y_{j,n_j}=\mathbb I_{\left\{\left|S_{j,n_j}\right|\leq c_j\right\}}.
\end{equation*}
Let, for $j\geq 1$, the probability associated with $Y_{j,n_j}$ be given by
\begin{equation*}
P\left(Y_{j,n_j}=1\right)=p_{j,n_j}.
\end{equation*}
Hence, the likelihood of $p_{j,n_j}$, given $y_{j,n_j}$, is of the form
\begin{equation}
L\left(p_{j,n_j}\right)=p^{y_{j,n_j}}_{j,n_j}\left(1-p\right)^{1-y_{j,n_j}}.
\label{eq:likelihood}
\end{equation}
In the above, $p_{j,n_j}$ can be interpreted as the probability that the
series $S_{1,\infty}=\sum_{i=1}^{\infty}X_i$ is convergent when the data observed is $S_{j,n_j}$. 

\subsection{Recursive Bayesian posteriors}
\label{subsec:recursive_posteriors}



Consider the sequences $\left\{\alpha_j\right\}_{j=1}^{\infty}$ and $\left\{\beta_j\right\}_{j=1}^{\infty}$,
where $\alpha_j=\beta_j=1/j^2$ for $j=1,2,\ldots$.
At the first stage of our recursive Bayesian algorithm, that is, when $j=1$, 
let us assume that the prior is given by
\begin{equation*}
\pi(p_{1,n_1})\equiv Beta(\alpha_1,\beta_1),
\end{equation*}
where, for $a>0$ and $b>0$, $Beta(a,b)$ denotes the Beta distribution with mean $a/(a+b)$
and variance $(ab)/\left\{(a+b)^2(a+b+1)\right\}$.
Combining this prior with the
likelihood (\ref{eq:likelihood}) (with $j=1$), we obtain the following posterior of $p_{1,n_1}$ given $y_{1,n_1}$:
\begin{equation*}
\pi(p_{1,n_1}|y_{1,n_1})\equiv Beta\left(\alpha_1+y_{1,n_1},\beta_1+1-y_{1,n_1}\right).
\end{equation*}
At the second stage (that is, for $j=2$), for the prior of $p_{2,n_2}$ we consider the posterior
of $p_{1,n_1}$ given $y_{1,n_1}$ associated with the $Beta(\alpha_1+\alpha_2,\beta_1+\beta_2)$ prior.
That is, our prior on $p_{2,n_2}$ is given by:
\begin{equation}
\pi(p_{2,n_2})\equiv Beta\left(\alpha_1+\alpha_2+y_{1,n_1},\beta_1+\beta_2+1-y_{1,n_1}\right).
\label{eq:prior_stage_2}
\end{equation}
%
The posterior of $p_{2,n_2}$ given $y_{2,n_2}$ is then obtained by combining the second stage prior
(\ref{eq:prior_stage_2}) with (\ref{eq:likelihood}) (with $j=2$). The form of the posterior
at the second stage is thus given by
\begin{equation*}
\pi(p_{2,n_2}|y_{2,n_2})\equiv Beta\left(\alpha_1+\alpha_2+y_{1,n_1}+y_{2,n_2},\beta_1+\beta_2+2-y_{1,n_1}-y_{2,n_2}\right).
\end{equation*}
Continuing this way, at the $k$-th stage, where $k>1$, we obtain the following posterior of $p_{k,n_k}$:
\begin{equation}
\pi(p_{k,n_k}|y_{k,n_k})\equiv Beta\left(\sum_{j=1}^k\alpha_j+\sum_{j=1}^ky_{j,n_j},
k+\sum_{j=1}^k\beta_j-\sum_{j=1}^ky_{j,n_j}\right).
\label{eq:posterior_stage_k}
\end{equation}
It follows from (\ref{eq:posterior_stage_k}) that
\begin{align}
E\left(p_{k,n_k}|y_{k,n_k}\right)&=\frac{\sum_{j=1}^k\alpha_j
+\sum_{j=1}^ky_{j,n_j}}{k+\sum_{j=1}^k\alpha_j+\sum_{j=1}^k\beta_j};
\label{eq:postmean_p_k}\\
Var\left(p_{k,n_k}|y_{k,n_k}\right)&=
\frac{(\sum_{j=1}^k\alpha_j+\sum_{j=1}^ky_{j,n_j})(k+\sum_{j=1}^k\beta_j-\sum_{j=1}^ky_{j,n_j})}
{(k+\sum_{j=1}^k\alpha_j+\sum_{j=1}^k\beta_j)^2(1+k+\sum_{j=1}^k\alpha_j+\sum_{j=1}^k\beta_j)}.
\label{eq:postvar_p_k}
\end{align}
Since $\sum_{j=1}^k\alpha_j=\sum_{j=1}^k\beta_j=\sum_{j=1}^k\frac{1}{j^2}$, (\ref{eq:postmean_p_k})
and (\ref{eq:postvar_p_k}) admit the following simplifications:
\begin{align}
E\left(p_{k,n_k}|y_{k,n_k}\right)&=\frac{\sum_{j=1}^k\frac{1}{j^2}+\sum_{j=1}^ky_{j,n_j}}
{k+2\sum_{j=1}^k\frac{1}{j^2}};
\notag\\
Var\left(p_{k,n_k}|y_{k,n_k}\right)&=
\frac{(\sum_{j=1}^k\frac{1}{j^2}+\sum_{j=1}^ky_{j,n_j})(k+\sum_{j=1}^k\frac{1}{j^2}-\sum_{j=1}^ky_{j,n_j})}
{(k+2\sum_{j=1}^k\frac{1}{j^2})^2(1+k+2\sum_{j=1}^k\frac{1}{j^2})}.
\notag
\end{align}

\subsection{Characterization of convergence properties of the underlying infinite series}
\label{subsec:characterization}

Note that (see, for example, \ctn{Oksendal00}) it is possible to represent any stochastic process $X=\{X_i:i\in \mathfrak I\}$, for fixed
$i$, as a random variable $\omega\mapsto X_i(\omega)$, where $\omega\in\mathfrak S$;
$\mathfrak S$ being the set of all functions from $\mathfrak I$ into $\mathbb R$. 
Also, fixing $\omega\in\mathfrak S$, the function $i\mapsto X_i(\omega);~i\in \mathfrak I$,
represents a path of $X_i;~i\in\mathfrak I$. Indeed, we can identify $\omega$ with the function
$i\mapsto X_i(\omega)$ from $\mathfrak I$ to $\mathbb R$.
%

Now observe that the sample space of $S_{1,\infty}$ is also given by $\mathfrak S$.
We also assume, for the sake of generality, that for any $\omega\in\mathfrak S\cap\mathfrak N^c$, where
$\mathfrak N~(\subset\mathfrak S)$ has zero probability measure, the non-negative monotonically 
decreasing sequence $\{c_j\}_{j=1}^{\infty}$
depends upon $\omega$, so that we shall denote the sequence by $\{c_j(\omega)\}_{j=1}^{\infty}$.
In other words, we allow $\left\{c_j(\omega)\right\}_{j=1}^{\infty}$ to depend upon the corresponding series. 

With the above notions, the following two theorems provide Bayesian characterizations of convergence and divergence, respectively, of the underlying series $S_{1,\infty}$.
\begin{theorem}[\ctn{Roy17}]
\label{theorem:convergence}
For any $\omega\in\mathfrak S\cap\mathfrak N^c$, where $\mathfrak N$ is some null set having probability measure zero,
$S_{1,\infty}(\omega)$ is convergent 
if and only if there exists a non-negative monotonically decreasing sequence $\left\{c_j(\omega)\right\}_{j=1}^{\infty}$
such that for any choice of the sequence $\{n_j\}_{j=1}^{\infty}$,
\begin{equation*}
\pi\left(\mathcal N_1|y_{k,n_k}(\omega)\right)\rightarrow 1,
\end{equation*}
as $k\rightarrow\infty$, 
where $\mathcal N_1$ is any neighborhood of 1 (one).
\end{theorem}

\begin{theorem}[\ctn{Roy17}]
\label{theorem:divergence}
For any $\omega\in\mathfrak S\cap\mathfrak N^c$, where $\mathfrak N$ is some null set having probability measure zero,
$S_{1,\infty}(\omega)$ is divergent if and only if there exists a sequence $\{n_j(\omega)\}_{j=1}^{\infty}$ such that
\begin{equation*}
\pi\left(\mathcal N_0|y_{k,n_k(\omega)}(\omega)\right)\rightarrow 1,
\end{equation*}
as $k\rightarrow\infty$, 
where $\mathcal N_0$ is any neighborhood of 0 (zero).
\end{theorem}
\ctn{Roy17} point out that Theorem \ref{theorem:divergence} encompasses even oscillatory series. 

\begin{remark}
\label{remark:remark1}
Although $c_j(\omega)$ has so far been referred to as a non-negative monotonically decreasing sequence (see also \ctn{Roy17}), it is sufficient 
for $c_j(\omega)$ to be a non-negative sequence that converges to zero. All the results of \ctn{Roy17}, including Theorems \ref{theorem:convergence}
and \ref{theorem:divergence} continue to hold with this more flexible condition. This extra flexibility is valuable in our random series context where
$c_j(\omega)$ are non-negative and converge to zero but can not be guaranteed to be monotonically decreasing. 
\end{remark}

\section{Random infinite series and parametric upper bound for the partial sums}
\label{sec:parametric}


Let us assume that $\left\{X_{i}(\omega)\right\}_{i=1}^{\infty}$, for $\omega\in\mathfrak S\cap\mathfrak N^c$ 
is a given sequence of random variables (not necessarily independent) such that the marginal distribution of $X_{i}$ is
$f_{\theta_i}(\cdot)$, and that
we wish to learn if $S_{1,\infty}(\omega)=\sum_{i=1}^{\infty}X_{i}(\omega)$ converges for 
$\omega\in\mathfrak S\cap\mathfrak N^c$. 
In this regard, we assume that 
the form of the density $f_{\theta_i}$ is known. We shall consider both known and unknown $\theta_i$. 

In fact, for our Bayesian theory for characterizing infinite series, it is not strictly necessary to assume that the form of $f_{\theta_i}$ is known.
However, we need to be able to obtain appropriate $c_j(\omega)$ such that $|S_{j,n_j}(\omega)|\leq c_j(\omega)$ for $j\geq j_0(\omega)$ whenever $S_{1,\infty}(\omega)<\infty$.
In the case of deterministic series, the functional forms of the series elements are known. Embedding the series in question in a class of series most of whose
convergence properties are related to the values of some (set of) parameter(s) $a$, \ctn{Roy17} could obtain suitable $c_j(\omega)$ for the series of interest by exploiting
the convergence properties of the parameterized class of series. For the current random series scenario, availability of information regarding some suitable
class of series in which we can embed our given random series of interest will be useful for our purpose. 
In this regard, assuming a known form of the density $f_{\theta_i}$ will be useful for constructing
parametric upper bounds for the partial sums. However, we shall also construct a general and effective nonparametric upper bound form that does not require any such information
but improves itself adaptively with the recursive Bayesian steps.

\subsection{Construction of parametric upper bound for the partial sums}
\label{subsec:parametric_upper_bound}
It will be convenient for our purpose to build the theory with unknown $\theta_i$ and to view known $\theta_i$ situations as special cases. 
\subsubsection{Unknown $\theta_i$}
\label{subsubsec:theta_unknown}
Let us begin with the assumption that $\left\{\theta_i\right\}_{i=1}^{\infty}$ is a stochastic process (again, not necessarily independent) with marginal density
$g_{\psi_i}$ where the density form as well as $\psi_i$ will be assumed to be known in this parametric bound construction setup.

For $i\geq 1$, let us introduce spaces for convergence and divergence, which we denote by $\Psi^{(c)}_i$ and $\Psi^{(d)}_i$, respectively, such that
$\sum_{i=1}^{\infty}\varphi_i$ is convergent and divergent, respectively, for $\varphi_i\in\Psi^{(c)}_i$ and $\varphi_i\in\Psi^{(d)}_i$, for $i\geq 1$.
In the above infinite sum, we assume that $\varphi_i$ varies only with respect to $i$ and is constant with respect to all other possible parameters.

To illustrate, let for any $\epsilon>0$, $\Psi^{(c)}_i=\left\{i^{-p}:~p\in[1+\epsilon,\infty)\right\}$ and $\Psi^{(d)}_i=\left\{i^{-p}:~p\in(-\infty,1]\right\}$, or
$\Psi^{(c)}_i=\left\{q^{-i}:q\in[1+\epsilon,\infty)\right\}$ and $\Psi^{(d)}_i=\left\{q^{-i}:q\in[0,1]\right\}$. 
Thus, a typical element of $\Psi^{(c)}_i=\left\{i^{-p}:~p\in[1+\epsilon,\infty)\right\}$ is $\varphi_i=i^{-p}$, where $p\in[1+\epsilon,\infty)$. 
Hence, if $p\in[1+\epsilon,\infty)$ is held fixed, then
$\varphi_i$ changes only with respect to $i$. Hence, $\sum_{i=1}^{\infty}\varphi_i<\infty$ for any fixed $p\in[1+\epsilon,\infty)$. On the other hand,
$\sum_{i=1}^{\infty}\varphi_i=\infty$ for $\varphi_i=i^{-p}\in\Psi^{(d)}_i=\left\{i^{-p}:~p\in(-\infty,1]\right\}$, with $p$ held fixed.

However, the provision of allowing $\varphi_i$ to vary only with respect to $i\geq 1$, will be restricted to infinite sums only, not elsewhere.

To proceed, we assume that $E(|\theta_i|)=h_i(\psi_i)$, where $h_i:\Psi^{(c)}_i\cup\Psi^{(d)}_i\mapsto\mathbb R^+$ (where $\mathbb R^+=[0,\infty)$) is 
such that $\sum_{i=1}^{\infty}h_i(\varphi_i)<\infty$
for $\varphi_i\in\Psi^{(c)}_i$; $i\geq 1$ and $\sum_{i=1}^{\infty}h_i(\varphi_i)=\infty$ for $\varphi_i\in\Psi^{(d)}_i$; $i\geq 1$.



For any $\varphi_i\in \Psi^{(c)}_i\cup\Psi^{(d)}_i$, let $G_{\varphi_i}$ denote the cumulative distribution function (cdf) of $g_{\varphi_i}$.
Now let, for each $x\in\mathbb R$, $G_i(x)=\underset{\varphi_i\in\Psi^{(c)}_i}{\inf}~G_{\varphi_i}(x)$. 
Assume that $G_i(\cdot)$ is continuous for $i\geq 1$. Then it follows that 
$\underset{x\rightarrow -\infty}{\lim}~G_i(x)=0$, $\underset{x\rightarrow \infty}{\lim}~G_i(x)=1$.
Also, if $x_1<x_2$, 
$G_i(x_1)\leq G_{\psi_i}(x_1)\leq G_{\psi_i}(x_2)$ for all $\psi_i\in\Psi^{(c)}_i$, so that 
$G_i(x_1)\leq G_i(x_2)$, satisfying the monotonicity property. Hence, $G_i(\cdot)$ is a continuous distribution function
for $i\geq 1$. Let $g_i$ denote the corresponding density function.

Let $\tilde\theta_i\sim g_i$. Then $\sum_{i=1}^{\infty}E\left(\left|\tilde\theta_i\right|\right)<\infty$. By Theorem 1 of \ctn{Kawata72} (see also \ctn{Pakes04}) it follows
that the series $\sum_{i=1}^{\infty}\tilde\theta_i$ is absolutely convergent almost surely, irrespective of any dependence structure among the $\tilde\theta_i$'s.

Hence, it follows that if $G_i(\cdot)$ is continuous for $i\geq 1$, then it is a distribution function satisfying
$G_i(x)\leq G_{\psi_i}(x)$ for all $x$ and $\psi_i\in \Psi^{(c)}_i$.
Consequently, for any {\it fixed} random number $U_i$, where $U_i\sim U(0,1)$, the uniform distribution on $(0,1)$ (this means
that we first draw $U_i\sim U(0,1)$ and then fix this $U_i$ to invert the
distribution functions $G_{\psi_i}$ and $G_i$, as below),
it holds that {\it for all} $\psi_i\in\Psi^{(c)}_i$,
\begin{equation}
G^{-}_{\psi_i}(U_i)\leq G^{-}_i(U_i),
\label{eq:sand2}
\end{equation}
where, for any distribution function $G$, $G^{-}(x)=\inf\{y:G(y)\geq x\}$, is the inverse
of $G$.

The inversions in (\ref{eq:sand2}) are nothing but simulations from the distributions 
corresponding to $G_{\psi_i}$ and $G_i$, respectively.
We thus set $\theta_{\psi_i}=G^{-}_{\psi_i}(U_i)$ and $\tilde \theta_i=G^{-}_i(U_i)$. 

Since inequality (\ref{eq:sand2}) holds for all $\psi_i\in\Psi^{(c)}_i$, this implies that for 
{\it fixed} $U_i$, whatever value of $\theta_{\psi_i}$ is simulated
using the relation $\theta_{\psi_i}=G^{-}_{\psi_i}(U_i)$, whatever may be the values of $\psi_i\in\Psi^{(c)}_i$,
it must always hold that 
\begin{equation}
\theta_{\psi_i}\leq \tilde \theta_i. 
\label{eq:sand3}
\end{equation}
Now suppose that $X_i$ are non-negative and admits the form $X_i=F^{-}_{\theta_i}(U_i)$, where $F_{\theta_i}$ is the distribution function of $X_i$ conditional on $\theta_i$, 
and assume that (\ref{eq:sand3}) ensures the inequality $X_{\theta_{\psi_i}}=F^{-}_{\theta_{\psi_i}}(U_i)\leq F^{-}_{\tilde\theta_i}(U_i)=X_{\tilde\theta_i}$.
Then, setting $X_{\theta_{\psi_i}}=X_i$ so that $F^{-}_{\theta_{\psi_i}}(U_i)=X_i$, would enable us to obtain $U_i$ in terms of $X_i$ and
$\theta_{\psi_i}$, for given $\theta_{\psi_i}$. This $U_i$ will
then be used in $F^{-}_{\tilde\theta_i}(U_i)$ to form $X_{\tilde\theta_i}=F^{-}_{\tilde\theta_i}(U_i)$, for given $\tilde\theta_i$. 
The partial sums associated with $\{X_{\tilde\theta_i}\}_{i=1}^{\infty}$ will then
constitute valid upper bounds for the partial sums corresponding to the underlying random series summands $\{X_{i}\}_{i=1}^{\infty}$.

Note that the above assumption of non-negative support of $X_i$ is crucial, since for general supports, upper bounds for the partial sums can not ensure
that the absolute values of the partial sums are bounded above by the absolute values of the corresponding upper bounds.

All the above results and discussions continue to hold if $X_i$ are discrete random variables with finite support. The proof that $G_i$ are valid distribution
functions in such cases is the same as that presented in Section S-1 of \ctn{Sabya12}.
Indeed, the principle of constructing upper bounds in the method described so far has some parallel in \ctn{Sabya12}, although in a very different, perfect sampling
context. 

\subsubsection{Known $\theta_i$}
\label{subsubsec:theta_known}
Now, if $\theta_i$ are known, then we can apply the same procedure to $f_{\theta_i}$ instead of $g_{\psi_i}$.
In that case, letting $F_{\theta_i}$ denote the distribution function associated with $f_{\theta_i}$ and 
$F_i(x)=\underset{\varphi_i\in\Psi^{(c)}_i}{\inf}~F_{\varphi_i}(x)$ for $x\in\mathbb R$, we shall then have
\begin{equation}
X_i=F^{-}_{\theta_i}(U_i)\leq F^{-}_i(U_i)=\tilde X_i, 
\label{eq:theta_known}
\end{equation}
which ensures $S_{j,n_j}\leq \tilde S_{j,n_j}$, where $\tilde S_{j,n_j}$ are the partial sums
associated with $\{\tilde X_i\}_{i=1}^{\infty}$. This would enable us to set $c_j=\tilde S_{j,n_j}$ as the upper bound for the partial sums
of $\{X_i\}_{i=1}^{\infty}$.
For known $\theta_i$, given $X_i$, $U_i$ is available from the first equality of (\ref{eq:theta_known}), which can be used in the second equality 
of (\ref{eq:theta_known}) to form $\tilde X_i$. 

\subsection{Upper bound for partial sums for hierarchical scale families on non-negative supports}
\label{sec:upper_bound}
To see the utility of (\ref{eq:sand3}), 
let us assume that the distribution of $X_i$ given $\theta_i$ is a scale family on $[0,\infty)$, that is, 
\begin{equation}
f_{\theta_i}(x_i)=\frac{1}{\theta_i}f\left(\frac{x_i}{\theta_i}\right)\mathbb I_{\left\{x_i>0\right\}},
\label{eq:scale1}
\end{equation}
where $\theta_i>0$, and $f(\cdot)$ is a density function supported on $[0,\infty)$. Let us assume that $\theta_i$ are random and have densities $g_{\psi_i}$
with the same details as in Section \ref{subsubsec:theta_unknown}. Since $\theta_i$ are also random variables, the model pertains to a hierarchical
scale family.

The distribution function corresponding to (\ref{eq:scale1}) is of the form $F\left(\frac{x_i}{\theta_i}\right)$, where $F$ is the cdf corresponding to the
density function $f$. Hence, $X_i=\theta_iF^{-}(U_i)$. Let $X_{\theta_{\psi_i}}=\theta_{\psi_i}F^{-}(U_i)$ and 
$X_{\tilde\theta_i}=\tilde\theta_iF^{-}(U_i)$. 
Here $U_i$ are $iid$ $U(0,1)$ random variables assumed to be independent of the uniform random variables used to draw $\theta_i$ and $\theta_{\psi_i}$.
Since $F^{-}(U_i)>0$, 
(\ref{eq:sand3}) ensures 
\begin{equation}
X_{\theta_{\psi_i}}\leq X_{\tilde\theta_i}.
\label{eq:sand4}
\end{equation}
It follows from (\ref{eq:sand4}) that 
\begin{equation}
S^{\theta_{\psi}}_{j,n_j}\leq S^{\tilde\theta}_{j,n_j}, 
\label{eq:S1}
\end{equation}
where $S^{\theta_{\psi}}_{j,n_j}$ and $S^{\tilde\theta}_{j,n_j}$ are the partial sums associated with 
the series $\left\{X_{\theta_{\psi_i}}\right\}_{i=1}^{\infty}$ and $\left\{X_{\tilde\theta_i}\right\}_{i=1}^{\infty}$, respectively.
The relation (\ref{eq:S1}) enables us to set $c_j=S^{\tilde\theta}_{j,n_j}$. 
Note that since $X_{\theta_{\psi_i}}$ and $\theta_{\psi_i}$ are known in the relation $X_{\theta_{\psi_i}}=\theta_{\psi_i}F^{-}(U_i)$, 
$U_i$ can be obtained from this equality, and can be used to form $X_{\tilde\theta_i}=\tilde\theta_iF^{-}(U_i)$.
In fact, for given $X_i$ and $\theta_{\psi_i}$, we set $X_{\theta_{\psi_i}}=\theta_{\psi_i}F^{-}(U_i)=X_i$, 
and solve for $U_i$ from the last equality, which we then use for construct
$X_{\tilde\theta_i}=\tilde\theta_iF^{-}(U_i)$.


\subsubsection{Illustration with hierarchical exponential distribution}
\label{subsec:exponential}

Let $f_{\theta_i}(x)=\frac{1}{\theta_i}\exp\left(-\frac{x}{\theta_i}\right)$; $x>0$, $\theta_i>0$. Also, let
$g_{\psi_i}(\theta)=\frac{1}{\psi_i}\exp\left(-\frac{\theta}{\psi_i}\right)$; $\theta>0$, $\psi_i>0$.
Here $G_{\psi_i}(\theta)=1-\exp\left(-\frac{\theta}{\psi_i}\right)$. Let $r_i(\epsilon)=\min\left\{i^{(1+\epsilon)},(1+\epsilon)^i\right\}$. Then
$G_{\tilde\theta_i}(\theta)=1-\exp\left(-\theta r_i(\epsilon)\right)$.

The upper bounds for the partial sums in this case can be constructed in the following manner. 
Note that here $X_{\theta_{\psi_i}}=\theta_{\psi_i}F^{-}(U_i)=-\theta_{\psi_i}\log U_i$
and $X_{\tilde\theta_i}=\tilde\theta_iF^{-}(U_i)=-\tilde\theta_i\log U_i$, where, for $i\geq 1$, $U_i\stackrel{iid}{\sim}U(0,1)$.
Also, $\theta_{\psi_i}=-\psi_i\log U^*_i$ and $\tilde\theta_i=-r^{-1}_i(\epsilon)\log U^*_i$, where $U^*_i\stackrel{iid}{\sim}U(0,1)$ and are independent of $U_i$,
for $i\geq 1$. For theoretically sound bound construction in practice, we shall first simulate $\theta_{\psi_i}$ and $\tilde\theta_i$ using the same $U^*_i$. 
Then, we shall obtain $U_i$ from the equality $X_{\theta_{\psi_i}}=-\theta_{\psi_i}\log U_i=X_i$, which we shall use to 
construct $X_{\tilde\theta_i}=-\tilde\theta_i\log U_i$. These, in turn, lead to (\ref{eq:sand4}) and (\ref{eq:S1}).

To obtain the relevant result regarding upper bounds for the partial sums we begin with the following theorem.
\begin{theorem}
\label{lemma:lemma_exp}
Let $\theta_i$ be independent. Then $\sum_{i=1}^{\infty}\theta_i<\infty$ almost surely if and only if $\sum_{i=1}^{\infty}\psi_i<\infty$.
\end{theorem}
\begin{proof}
By Kolmogorov's three series theorem (see, for example, \ctn{Resnick14}), it is easy to see that 
$\sum_{i=1}^{\infty}\psi_i<\infty$ implies $\sum_{i=1}^{\infty}\theta_i<\infty$ almost surely. 
We now show that for any $R>0$, $\sum_{i=1}^{\infty}E\left(\theta_i\mathbb I_{\left\{\theta_i<R\right\}}\right)=\infty$ if $\sum_{i=1}^{\infty}\psi_i=\infty$.
This would then ensure, by Kolmogorov's three series theorem, that $\sum_{i=1}^{\infty}\theta_i=\infty$ almost surely.

Note that 
\begin{equation}
E\left(\theta_i\mathbb I_{\left\{\theta_i<R\right\}}\right)=
	\psi_i\times\left[1-\exp\left(-\frac{R}{\psi_i}\right)\left(1+\frac{R}{\psi_i}\right)\right].
\label{eq:exp1}
\end{equation}
If $\psi_i\in\Psi^{(d)}_i=\left\{i^{-p}:p\in(-\infty,1]\right\}$, then $\psi_i=i^{-p}$ for some $p\in(-\infty,1]$. 
Suppose first that $p\in(0,1]$. 
In that case,
\begin{equation}
1-\exp\left(-\frac{R}{\psi_i}\right)\left(1+\frac{R}{\psi_i}\right)\rightarrow 1,~\mbox{as}~i\rightarrow\infty.
\label{eq:exp2}
\end{equation}
It follows from (\ref{eq:exp2}) that for any $\varepsilon>0$, there exist $i_0\geq 1$ such that for $i\geq i_0$, the right hand side of (\ref{eq:exp1}) exceeds
$\psi_i(1-\varepsilon)$. Since $\sum_{i=i_0}^{\infty}\psi_i(1-\varepsilon)=\infty$ for $\psi_i=i^{-p}$ where $p\in(0,1]$, it follows that
%
\begin{equation*}
\sum_{i=1}^{\infty}E\left(\theta_i\mathbb I_{\left\{\theta_i<R\right\}}\right)=\infty~\mbox{for}~\psi_i=i^{-p},~\mbox{with}~ p\in(0,1],~\mbox{for any}~R>0.
\end{equation*}
By Kolmogorov's three series theorem it then follows that $\sum_{i=1}^{\infty}\theta_i=\infty$, almost surely.

Now let us consider the case where $\psi_i=i^{-p}$, with $p\leq 0$. If $p=0$, then $\theta_i$ are $iid$, so that trivially, $\sum_{i=1}^{\infty}\theta_i=\infty$,
almost surely.
So, let $p<0$. Direct calculation shows that 
\begin{equation*}
P\left(\theta_i>R\right)=\exp\left(-\frac{R}{\psi_i}\right)=\exp\left(-Ri^p\right)\rightarrow 1,~\mbox{as}~i\rightarrow\infty.
\end{equation*}
Hence, $\sum_{i=1}^{\infty}P\left(\theta_i>R\right)=\infty$, for any $R>0$, so that by Kolmogorov's three series theorem, $\sum_{i=1}^{\infty}\theta_i=\infty$,
almost surely.

Finally, consider the case $\psi_i=q^{-i}$, $q\in[0,1]$. If $q=1$, then $\theta_i$ are $iid$, so that $\sum_{i=1}^{\infty}\theta_i=\infty$, almost surely. 
So, let $q\in[0,1)$. Then	
\begin{equation*}
P\left(\theta_i>R\right)=\exp\left(-Rq^i\right)\rightarrow 1,~\mbox{as}~i\rightarrow\infty,
\end{equation*}
which leads to $\sum_{i=1}^{\infty}\theta_i=\infty$, almost surely.
\end{proof}

Theorem \ref{lemma:lemma_exp} shows that in the case of independence, $\sum_{i=1}^{\infty}\theta_i<\infty$ if and only if $\psi_i\in\Psi^{(c)}_i$, for $i\geq 1$.
In the case of dependence, it can only be guaranteed that $\sum_{i=1}^{\infty}\theta_i<\infty$ if $\psi_i\in\Psi^{(c)}_i$, for $i\geq 1$. It can not be asserted
that $\sum_{i=1}^{\infty}\theta_i=\infty$ if $\psi_i\in\Psi^{(d)}_i$, for $i\geq 1$. 
The implication is that, if $X_i$ are also conditionally independent given $\theta_i$, $S^{\tilde\theta}_{j,n_j}$ of the form (\ref{eq:S1})
corresponds in the hierarchical exponential setup to the maximal convergent series closest to divergence
in the case of independence, but this need not be the case when $\theta_i$ and/or $X_i$ given $\theta_i$ are dependent.
This leads to the following theorem as a consequence of Theorem \ref{lemma:lemma_exp}.

\begin{theorem}
\label{theorem:theorem_exp}
For $i\geq 1$, let $\tilde\theta_i\sim G_{\tilde\theta_i}$, and $X_i\sim f_{\tilde\theta_i}(x_i)=\frac{1}{\tilde\theta_i}\exp\left(-\frac{x_i}{\tilde\theta_i}\right)$. 
Then the partial sums $\tilde S_{j,n_j}$ of the form (\ref{eq:S1}) in the  hierarchical exponential setup 
correspond to the maximal convergent series $\sum_{i=1}^{\infty}X_i$ that is the closest
to divergence, provided $\theta_i$ are independent and conditionally on $\theta_i$, $X_i$ are also independent.
\end{theorem}


\subsection{Construction of bounds for the partial sums in the general case}
\label{sec:bound_general}
In the general situation where either $X_i$ given $\theta_i$ and $\theta_i$ are not independent and/or $X_i$ 
is supported on the real line, it is not possible to mathematically establish that $S^{\tilde\theta}_{j,n_j}$ corresponds to the maximal convergent series
closest to divergence.

In the general case we propose to construct bounds with arbitrary sequence of $U_i$'s, in the following way. First note that if $\sum_{i=1}^{\infty}X_i<\infty$,
then, letting $S_{j,n_j}$ denote the partial sum associated with the above series, $\left|S_{j,n_j}\right|\rightarrow 0$ as $j\rightarrow\infty$, irrespective of the
choice of the $U_i$'s. 
Theoretically, we need not have $\left|S_{j,n_j}\right|\leq \left|S^{\tilde\theta}_{j,n_j}\right|$ even in the case of convergence, but we can expect that
\begin{equation}
\left|S_{j,n_j}\right|\leq \left|S^{\tilde\theta}_{j,n_j}\right|+\frac{a}{j} 
\label{eq:S2}
\end{equation}
holds in the case of convergence, where $a~(>0)$ is some suitable constant. The idea is to slightly inflate $\left|S^{\tilde\theta}_{j,n_j}\right|$ so that 
(\ref{eq:S2}) holds. 
We propose (\ref{eq:S2}) as an upper bound for the partial sums in the general setup.

\subsubsection{Illustration with normal distribution}
\label{subsec:example_normal}
Assume that $X_i\sim N\left(\mu_i,\sigma^2_i\right)$, independently for $i\geq 1$. Assume also that for $i\geq 1$, independently, $\mu_i\sim N\left(0,\phi^2_i\right)$
and $\sigma^2_i\sim\mathcal E(\vartheta_i)$, that is, the exponential distribution with mean $\vartheta_i$. Let $\phi^2_i\in\Psi^{(c)}_i\cup \Psi^{(d)}_i$ and 
$\vartheta_i\in \Psi^{(c)}_i\cup \Psi^{(d)}_i$. 

It is well-known (see, for example, Exercise 7.7.14 of \ctn{Resnick14}) that $\sum_{i=1}^{\infty}X_i<\infty$ almost surely if and only 
if $\sum_{i=1}^{\infty}\mu_i<\infty$ and $\sum_{i=1}^{\infty}\sigma^2_i<\infty$ almost surely. This result, along with its two different proofs can be found
in page 319 of \ctn{Driver10}. 
Here, letting $\Psi^{(c)}_i=\left\{i^{-p}:p\in[1+\epsilon,M_1]\right\}$ or $\Psi^{(c)}_i=\left\{q^{-i}:q\in[1+\epsilon,M_2]\right\}$, 
where $M_1>1+\epsilon$, $M_2>1+\epsilon$, and $\tilde r_i=\max\{i^{M_1},M^i_2\}$, we have
\begin{equation}
G_{\tilde\mu_i}(\mu)=\left\{\begin{array}{c}\Phi\left(\mu \sqrt{r_i(\epsilon)}\right)~\mbox{if}~\mu\geq 0;\\ 
1-\Phi\left(-\mu\sqrt{\tilde r_i}\right)~\mbox{if}~\mu<0,\end{array}\right.
\label{eq:G_mu}
\end{equation}
and
\begin{equation}
G_{\tilde\sigma^2_i}(\sigma^2)=1-\exp\left(-\sigma^2r_i(\epsilon)\right),
\label{eq:G_sigmasq}
\end{equation}
where $r_i(\epsilon)=\min\left\{i^{(1+\epsilon)},(1+\epsilon)^i\right\}$.
In this case, due to independence, (\ref{eq:G_mu}) and (\ref{eq:G_sigmasq}) do correspond to maximal convergent series for $\sum_{i=1}^{\infty}\mu_i$ and
$\sum_{i=1}^{\infty}\sigma^2_i$, and it holds that $\mu_i\leq\tilde\mu_i$ and $\sigma^2_i\leq\tilde\sigma^2_i$, but since $X_i$ is supported on the entire
real line, these do not guarantee that even $X_i\leq X_{\tilde\theta_i}$ holds, where $\tilde\theta_i=(\tilde\mu_i,\tilde\sigma^2_i)$. 
For further clarity, note that $X_i=\mu_i+\sigma_iZ_i$, where $Z_i\stackrel{iid}{\sim}N(0,1)$, for $i\geq 1$. Even though it is possible to theoretically ensure 
$\mu_i\leq\tilde\mu_i$ and $\sigma^2_i\leq\tilde\sigma^2_i$, $Z_i$ takes values on the entire real line, and hence $X_i\leq X_{\tilde\theta_i}$ can not be
guaranteed.
Moreover, it is not possible to simulate from $G_{\tilde\mu_i}$ by inverting the distribution function.
However, we can still expect (\ref{eq:S2}) to hold, for appropriate choice of $a~(>0)$.

An important point to observe is that as $i\rightarrow\infty$, $1-\Phi\left(-\mu\sqrt{\tilde r_i}\right)\rightarrow 0$, so that under (\ref{eq:G_mu})
the distribution of $\tilde\mu_i$ supports only non-negative values, as $i\rightarrow\infty$. This results in too large an upper bound, which makes it hard to detect
divergences. Replacing this distribution of $\tilde\mu_i$ with $\tilde\mu_i\sim N\left(0,\sigma^2_{\tilde \mu_i}\right)$, with
$\sigma^2_{\tilde \mu_i}=1/r_i(\epsilon)$, resulted in more useful bounds for the partial sums in our simulation examples.

Note that in the case of independence, study of convergence of $S_{1,\infty}(\omega)$ for only one $\omega\in\mathfrak S$ is necessary,
since $S_{1,\infty}(\omega)$ either converges for almost all $\omega\in\mathfrak S$ or diverges 
for almost all $\omega\in\mathfrak S$.
The rest of the theory remains the same as that of \ctn{Roy17}.

\section{Simulation experiments with parametric upper bound}
\label{sec:simstudy1}

\subsection{Example 1: Hierarchical exponential distribution}
\label{subsec:bayesian_exp}
We first consider the setup $X_i\sim\mathcal E(\theta_i)$ and $\theta_i\sim\mathcal E(\psi_i)$; $i\geq 1$. Thus $X_i$ has a two-stage hierarchical exponential distribution.
Following the bound construction method detailed in Section \ref{subsec:exponential}, setting $\epsilon=0.001$
we considered the upper bound given by $c_j=S^{\tilde\theta}_{j,n_j}$, where $n_j=1000$, for $j=1,\ldots,K$, with $K=2000$. 

We implement our recursive Bayesian procedure on an ordinary dual core laptop, splitting the sum of $1000$ terms at each step of $2000$ stages into the two processors
using the Message Passing Interface (MPI) protocol in our C programming environment.
In our implementation, the Bayesian recursive algorithm takes less than a second to yield result.

The results of our convergence analyses of this setup are depicted in Figure \ref{fig:example_exp}, which shows that
the convergence behaviour of the random series are always correctly determined by our recursive Bayesian procedure with the aforementioned upper bound.
That the method performs so well in spite of such small sample size, seems to very encouraging.


\begin{figure}
\centering
\subfigure [Divergence.]{ \label{fig:exp1}
\includegraphics[width=6cm,height=5cm]{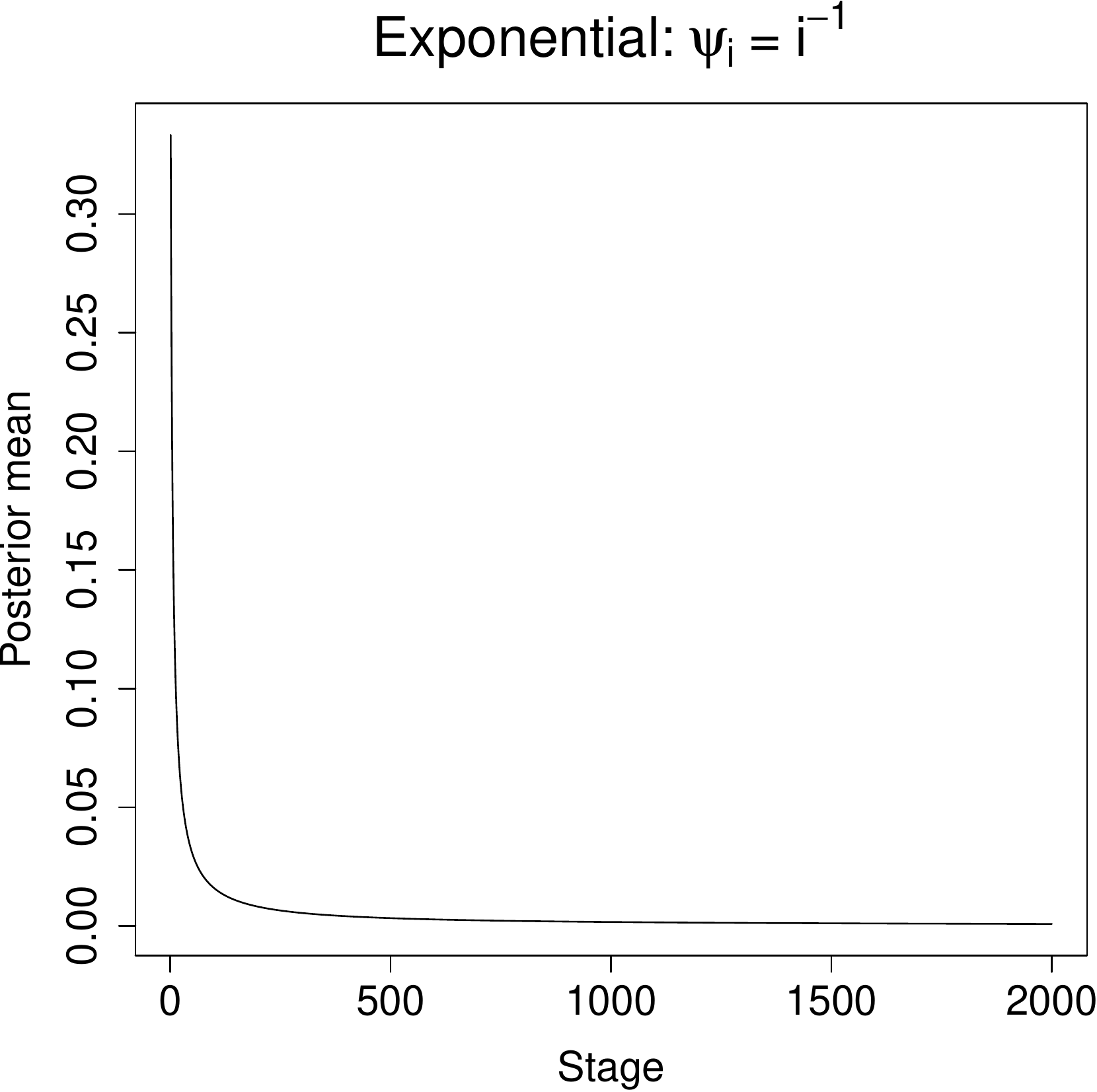}}
\hspace{2mm}
\subfigure [Convergence.]{ \label{fig:exp2}
\includegraphics[width=6cm,height=5cm]{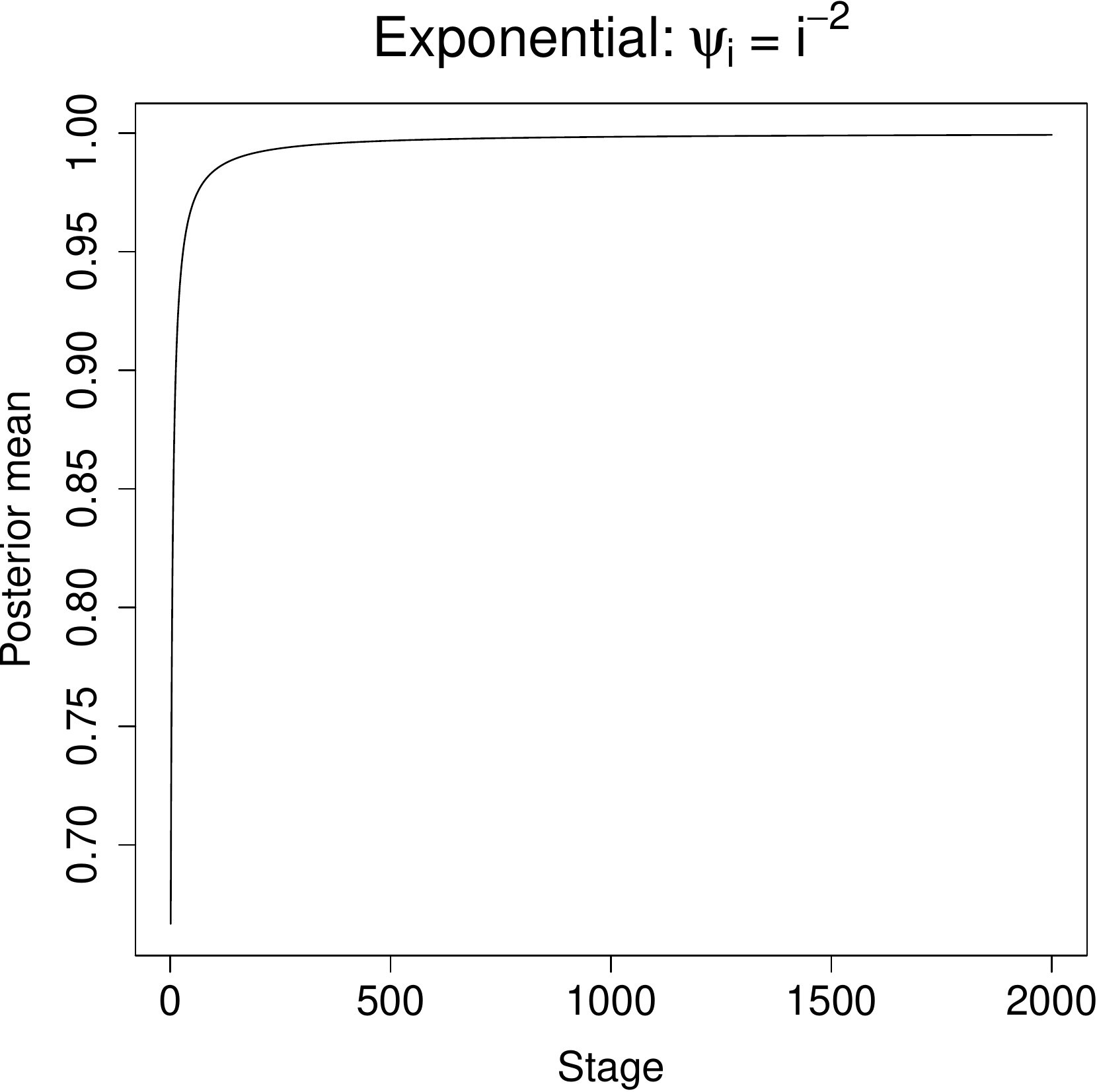}}\\
\subfigure [Convergence.]{ \label{fig:exp3}
\includegraphics[width=6cm,height=5cm]{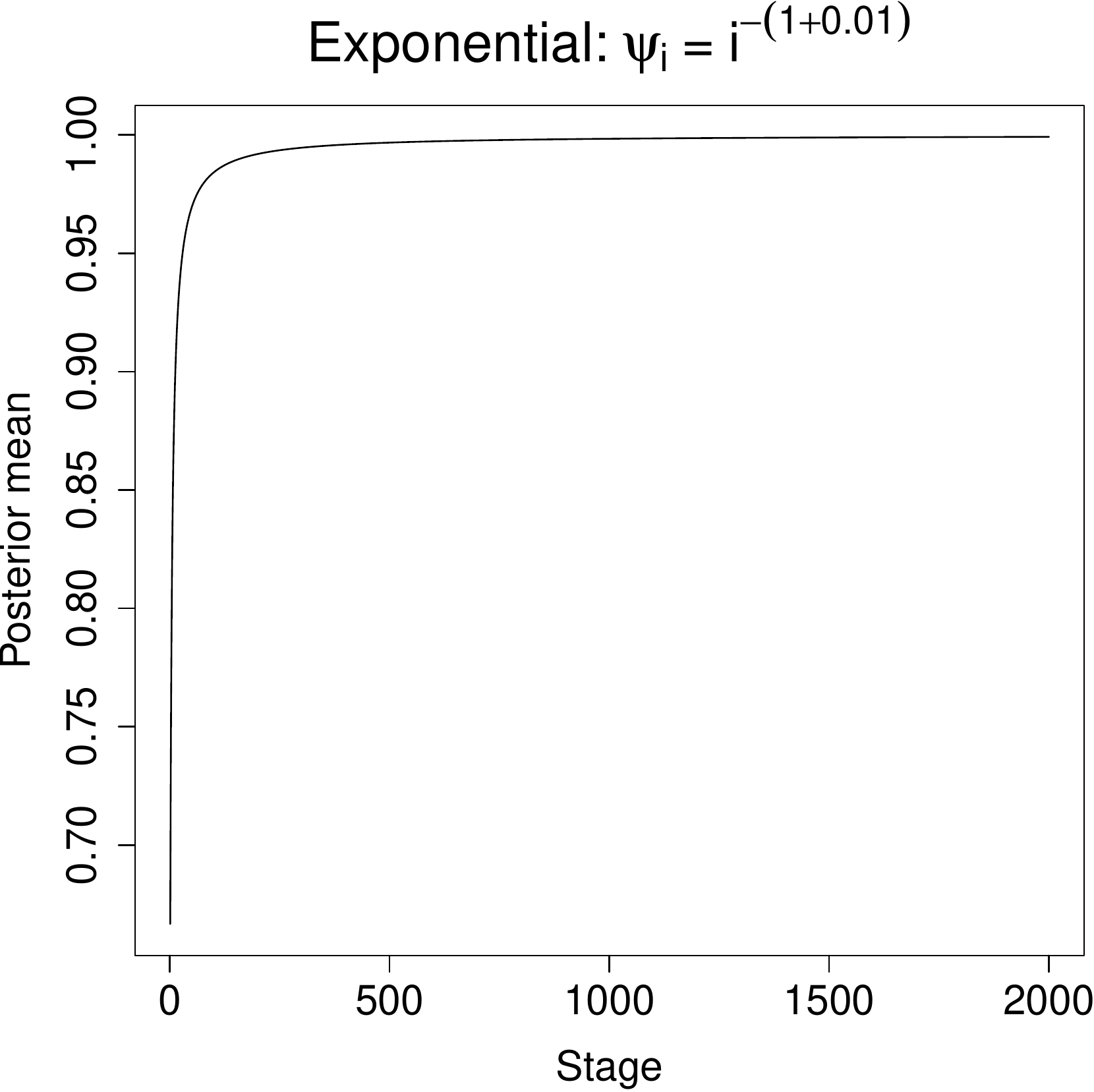}}
\hspace{2mm}
\subfigure [Convergence.]{ \label{fig:exp4}
\includegraphics[width=6cm,height=5cm]{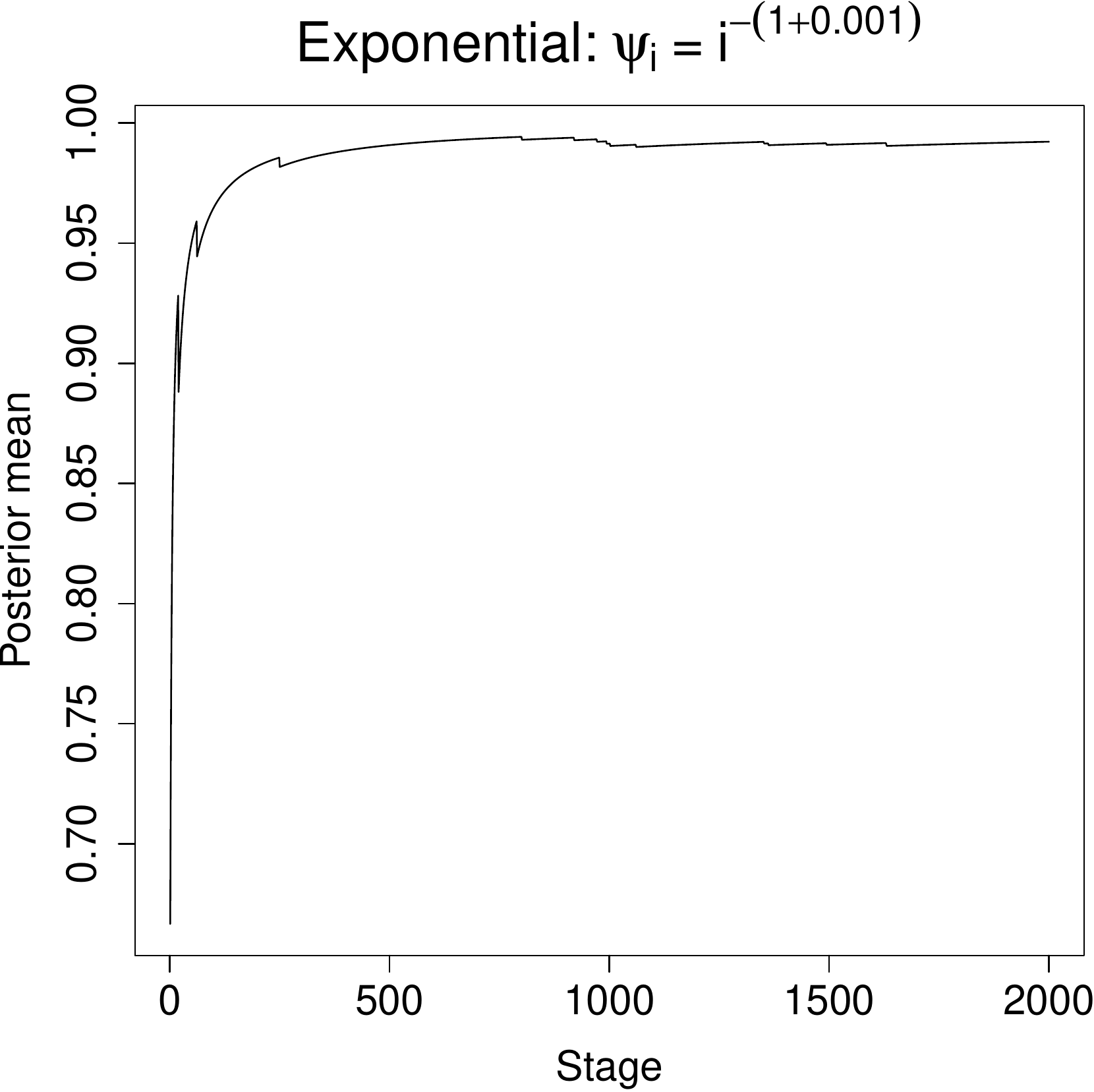}}\\
\subfigure [Convergence.]{ \label{fig:exp5}
\includegraphics[width=6cm,height=5cm]{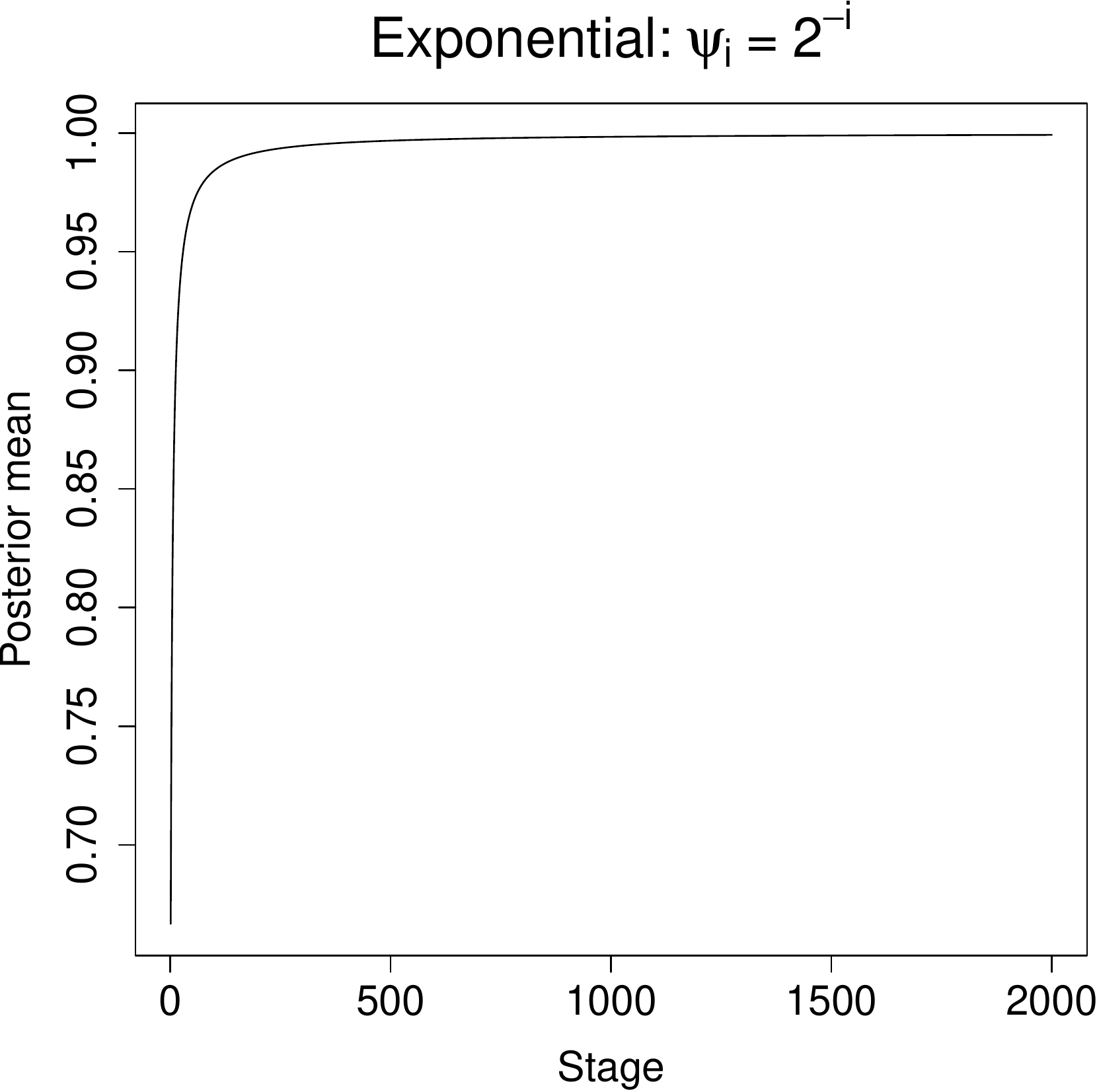}}
\hspace{2mm}
\subfigure [Convergence.]{ \label{fig:exp6}
\includegraphics[width=6cm,height=5cm]{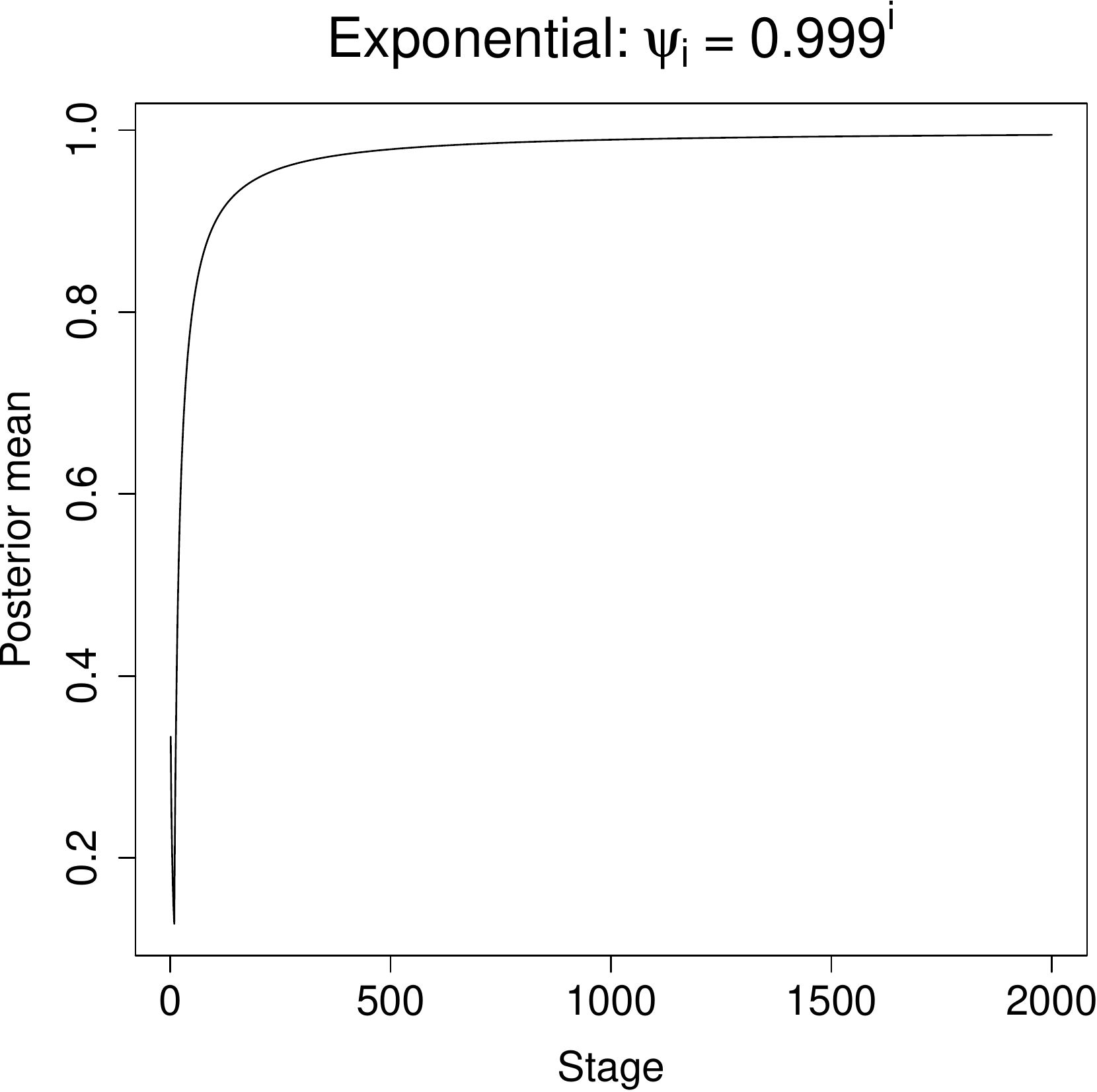}}\\
\subfigure [Divergence.]{ \label{fig:exp7}
\includegraphics[width=6cm,height=5cm]{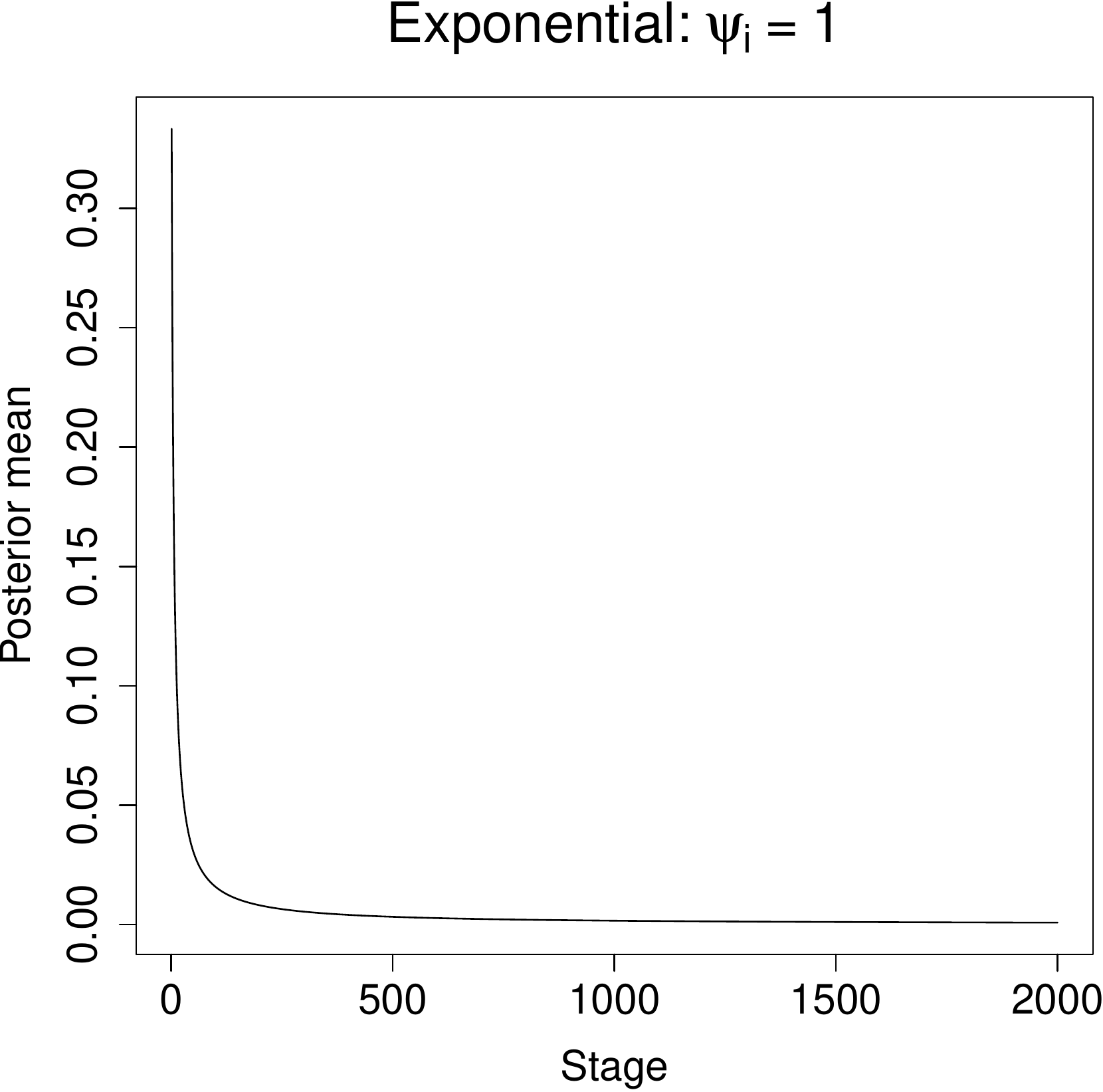}}
\caption{Example 1: Convergence and divergence for exponential series.}
\label{fig:example_exp}
\end{figure}

\subsection{Example 2: Hierarchical normal distribution}
\label{subsec:bayesian_normal}
Now let $X_i\sim N\left(\mu_i,\sigma^2_i\right)$, $\mu_i\sim N\left(0,\phi^2_i\right)$
and $\sigma^2_i\sim\mathcal E(\vartheta_i)$; $i\geq 1$. This specifies a two-stage hierarchical normal distribution for $X_i$.
%
For this setup, our results of convergence analyses are provided in Figure \ref{fig:example_normal}. 
Following the later discussion in Section \ref{subsec:example_normal} we
construct $S^{\tilde\theta}_{j,n_j}$ using $\tilde\mu_i\sim N\left(0,\sigma^2_{\tilde \mu_i}\right)$, with
$\sigma^2_{\tilde \mu_i}=1/r_i(\epsilon)$.
Consequently, setting $\epsilon=0.001$, we consider the upper bound
given by $c_j=\left|S^{\tilde\theta}_{j,n_j}\right|+\frac{0.1}{j}$, with $n_j=10^6$; $j=1,\ldots,K$, with
$K=10^6$. This many times longer run compared to the exponential simulation study setup detailed in Section \ref{subsec:bayesian_exp} 
is required since mathematically valid parametric upper bound for the partial sums
does not seem to be available in this case of normality. Indeed, as we shall see, even such enormously long runs turn out to be less than adequate in most cases.

Recall that in the case of exponential distribution, $n_j=1000$ for $j=1,\ldots,K$, with $K=2000$. Thanks to such small sample, it has been possible to obtain
the results in less than a second, even on an ordinary dual core laptop. However, in the current normality scenario, such pleasant computational perspective
is unimaginable.
Fortunately, we have access to 
a parallel computing architecture associated with a VMWare consisting of $100$ 64-bit cores, running at 2.80 GHz speed, and having 1 TB memory. 
Implementation of our parallelized C codes on the available $100$ cores takes about $52$ minutes. 

The convergence behaviour of the random series are correctly determined, but panels (f) and (g)
of Figures \ref{fig:example_normal} indicate very slow divergence. Indeed, these figures depict the posterior means in the last $5\times 10^5$ 
iterations of the total $K=10^6$ iterations.
We found that slow divergence is generally the case when one of $\sum_{i=1}^{\infty}\mu_i$ or $\sum_{i=1}^{\infty}\sigma^2_i$ is a divergent series of the form
$\sum_{i=1}^{\infty}i^{-p}$, with $1-\zeta\leq p\leq 1$, where $\zeta~(>0)$ is small.
\begin{figure}
\centering
\subfigure [Divergence.]{ \label{fig:normal1}
\includegraphics[width=6cm,height=5cm]{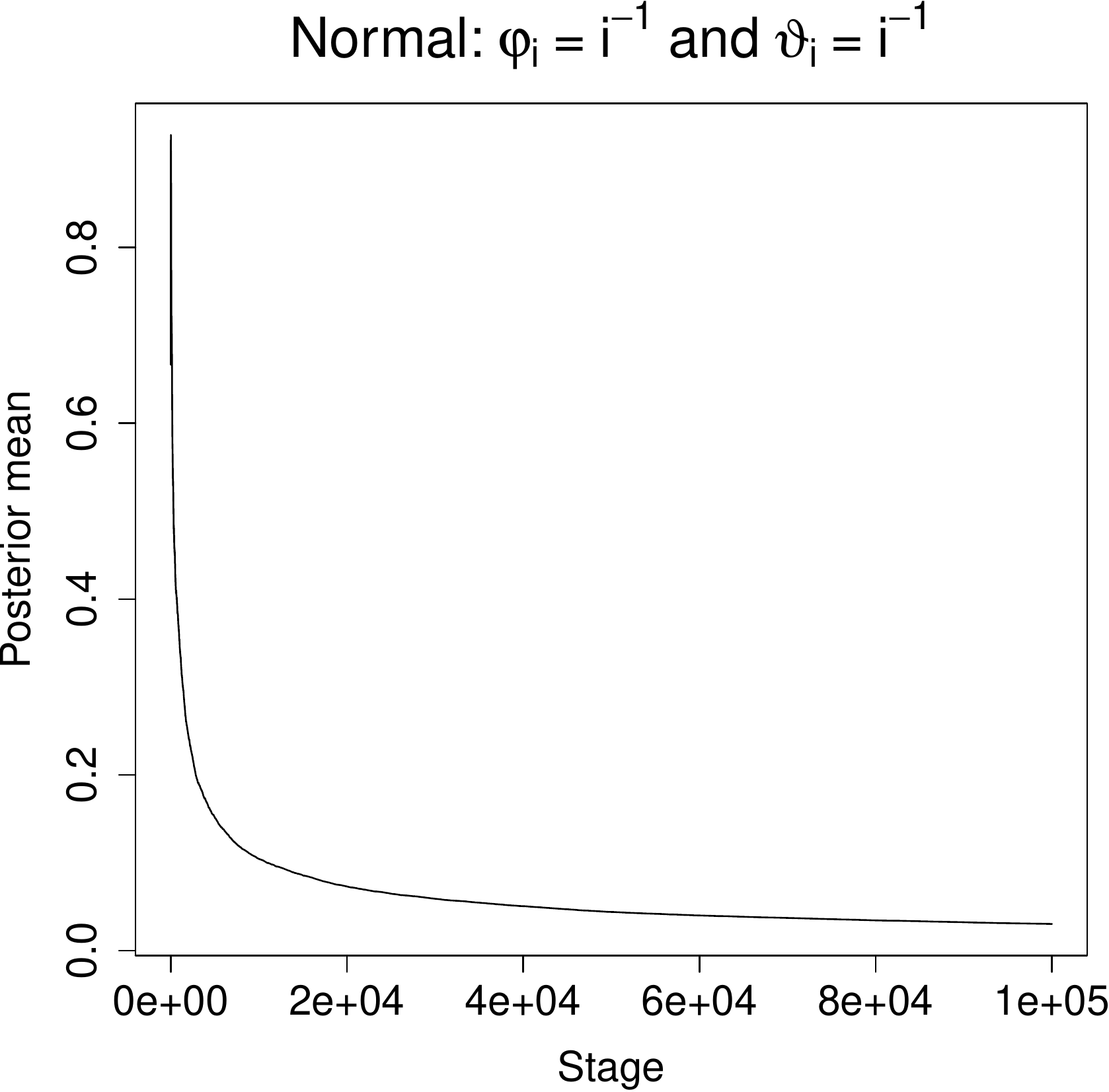}}
\hspace{2mm}
\subfigure [Convergence.]{ \label{fig:normal2}
\includegraphics[width=6cm,height=5cm]{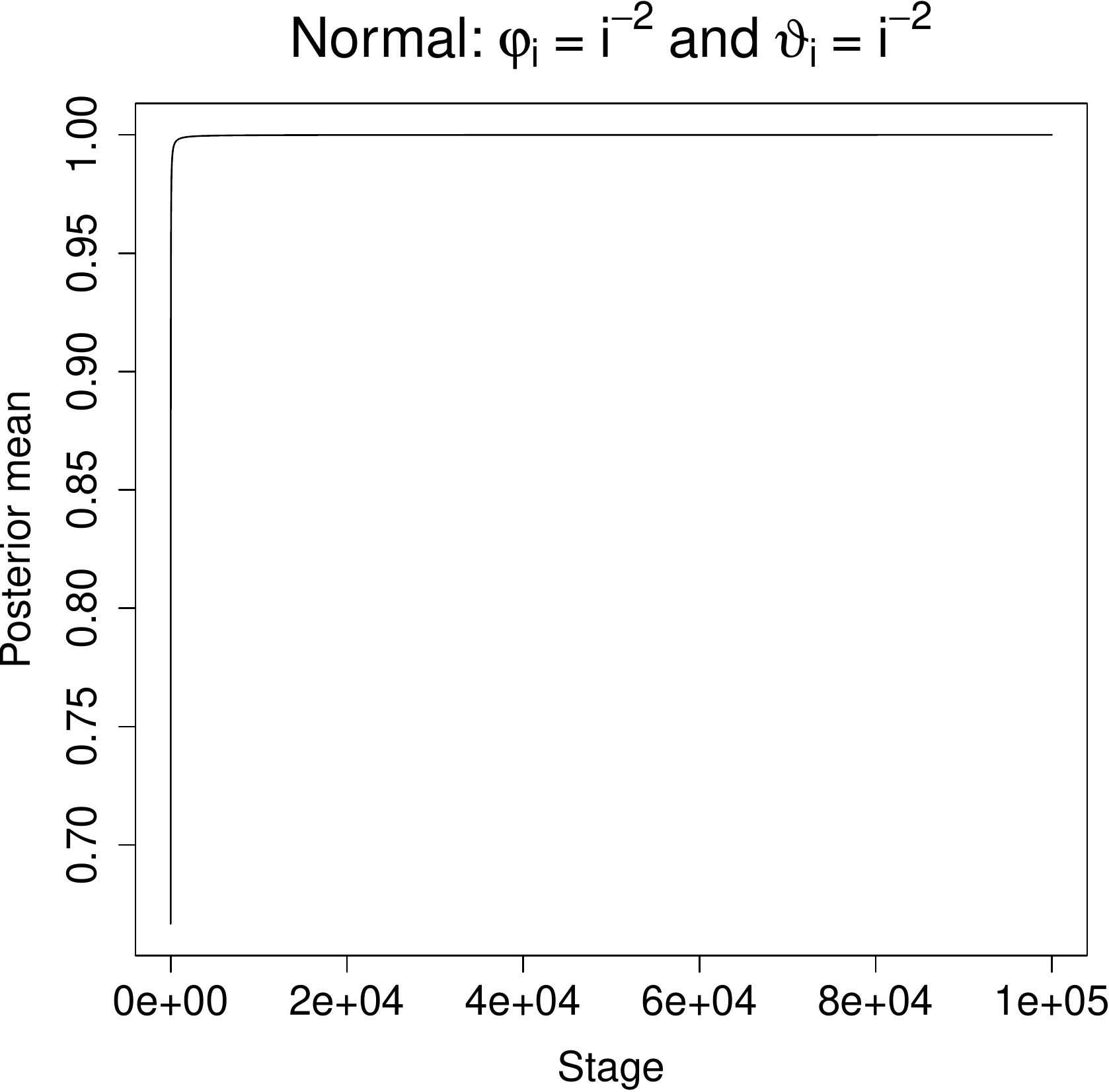}}\\
\subfigure [Convergence.]{ \label{fig:normal3}
\includegraphics[width=6cm,height=5cm]{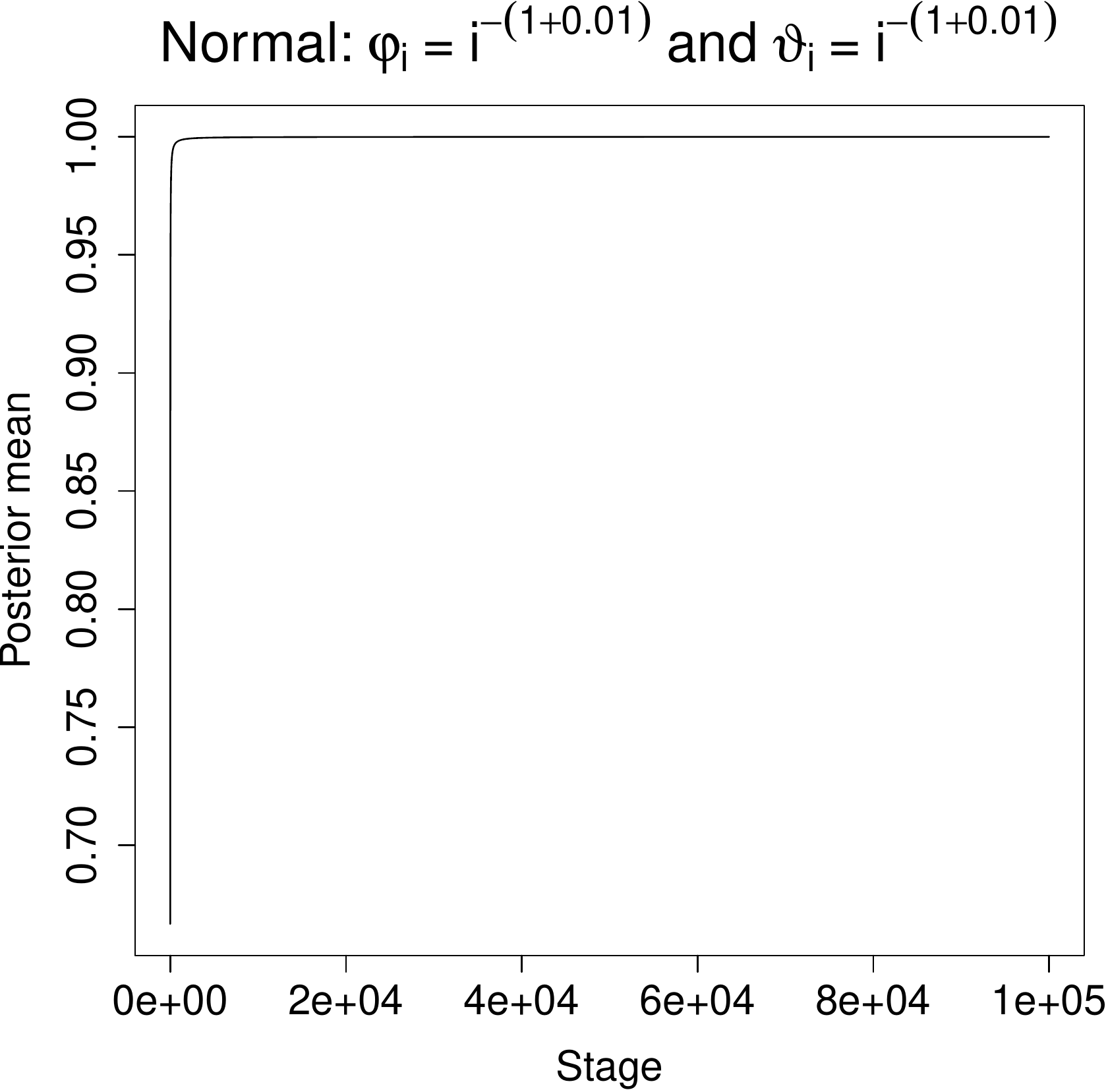}}
\hspace{2mm}
\subfigure [Divergence.]{ \label{fig:normal4}
\includegraphics[width=6cm,height=5cm]{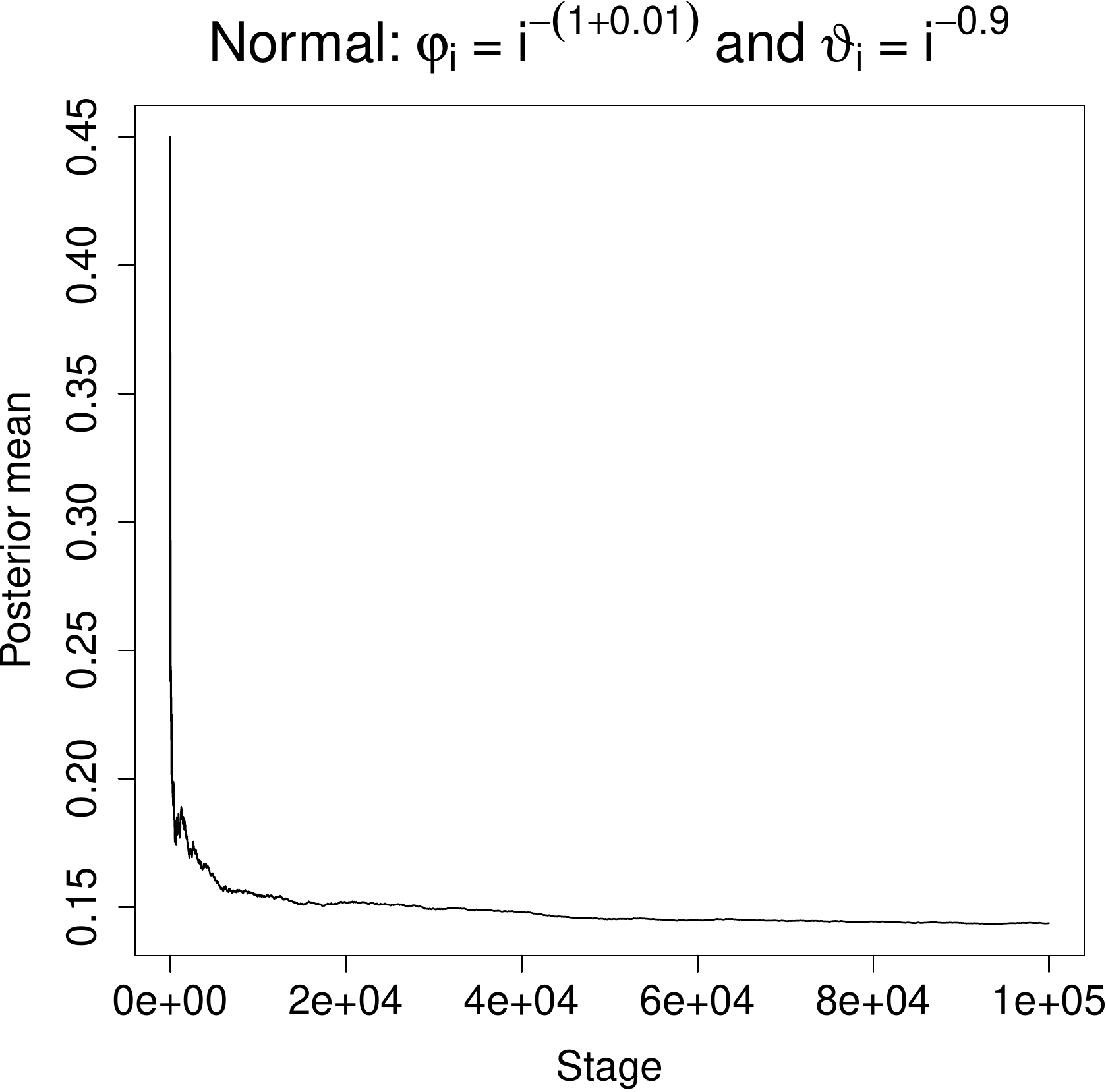}}\\
\subfigure [Divergence.]{ \label{fig:normal5}
\includegraphics[width=6cm,height=5cm]{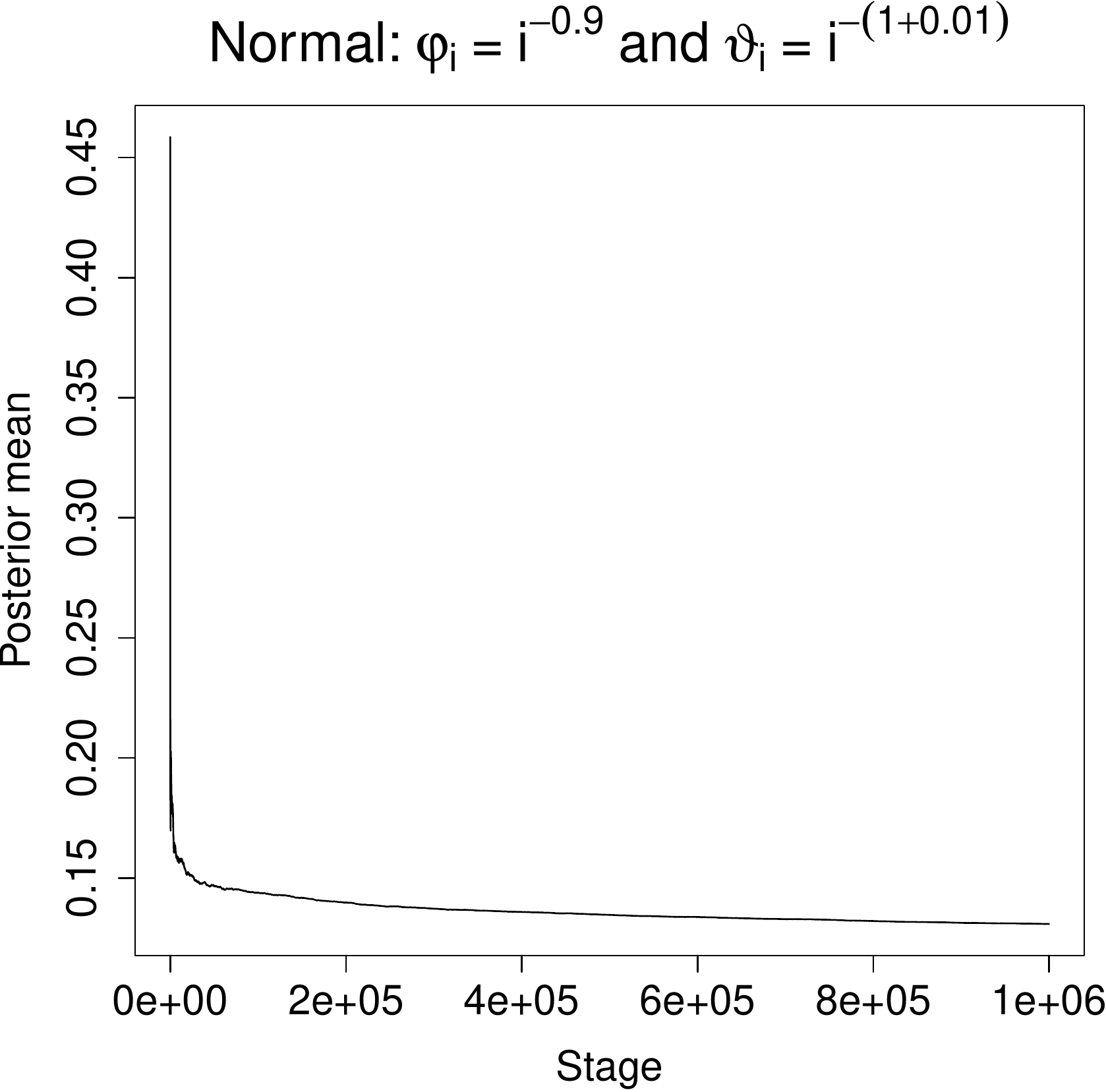}}
\hspace{2mm}
\subfigure [Divergence.]{ \label{fig:normal6}
\includegraphics[width=6cm,height=5cm]{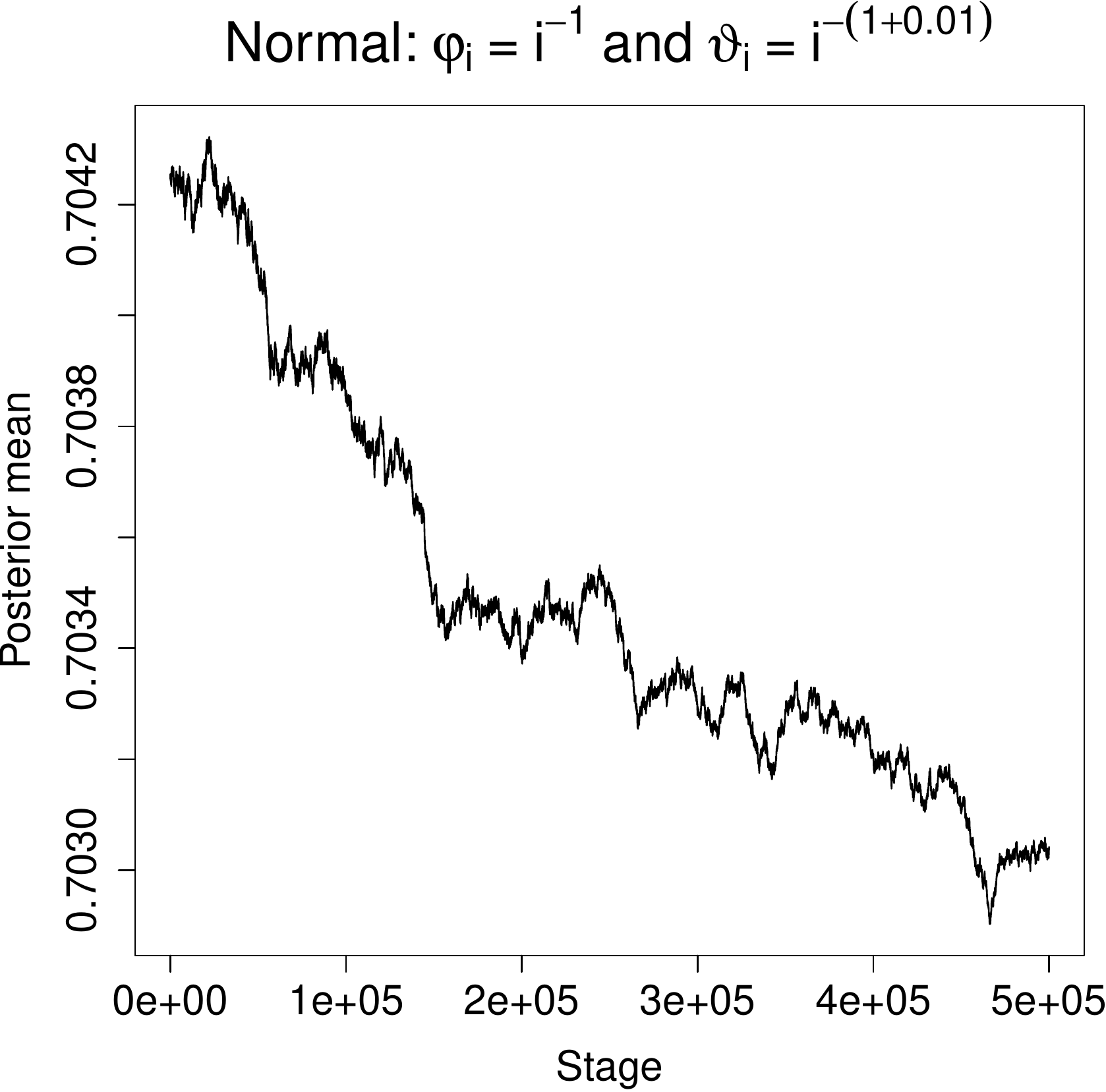}}\\
\subfigure [Divergence.]{ \label{fig:normal7}
\includegraphics[width=6cm,height=5cm]{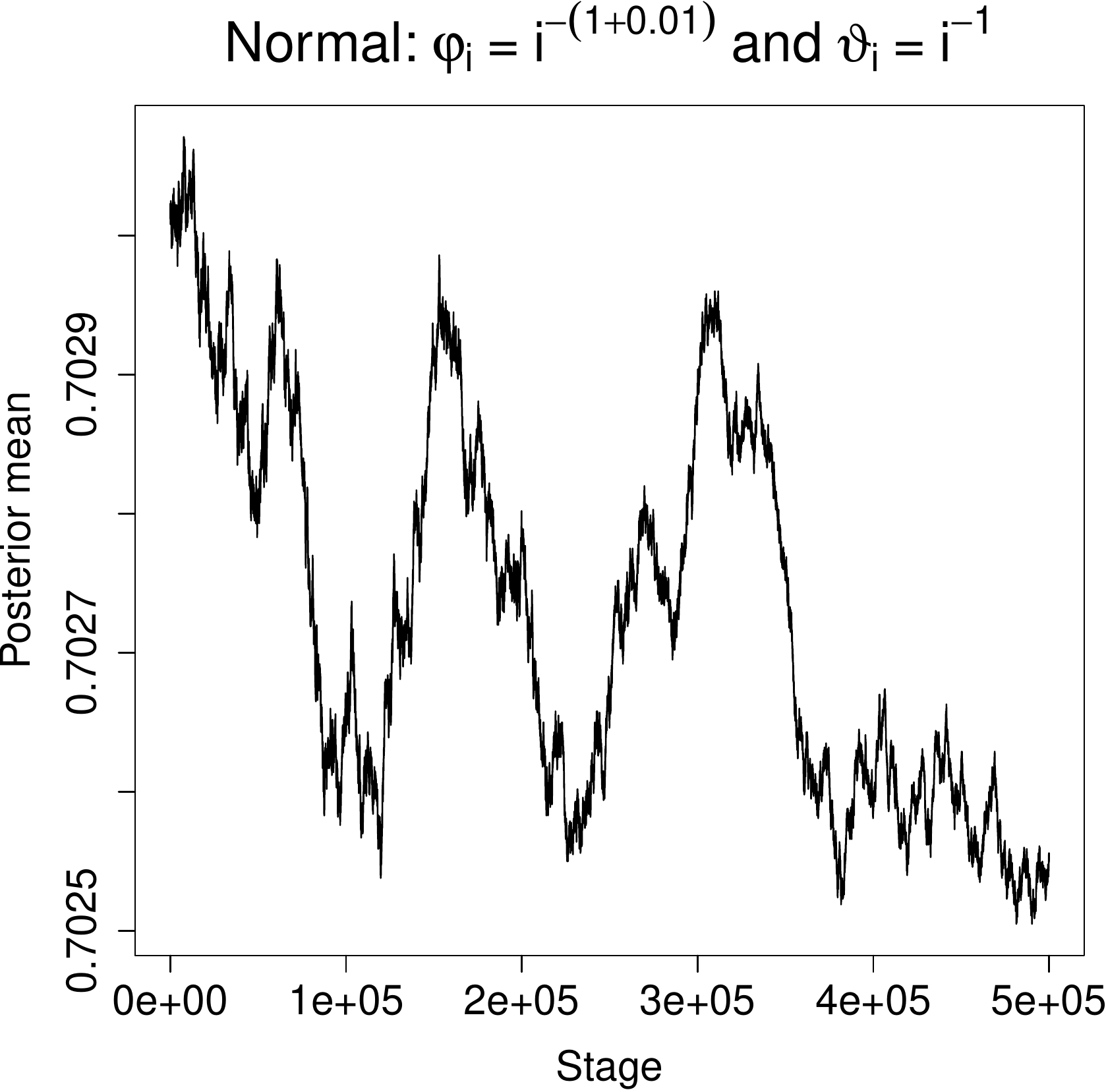}}
\caption{Example 2: Convergence and divergence for normal series.}
\label{fig:example_normal}
\end{figure}

\subsection{Example 3: Dependent hierarchical normal distribution}
\label{subsec:bayesian_normal_dependent}

So far we have considered examples of random series where the terms are independent. The actual convergence properties of these random series are
known by Kolmogorov's three series theorem, and knowledge of the convergence properties helped validate our Bayesian idea in these cases.

Since theoretically our Bayesian method characterizes all random series irrespective of their dependence structure, we now turn to empirical validation of our Bayesian method 
even in dependent situations. Note that Kolmogorov's three series theorem no longer holds for dependent situations, and 
we need to create examples where the actual convergence properties are known, in spite of dependence.

A simple example is as follows. We consider $[X_i|\xi]\sim N\left(\mu_i,\xi\sigma^2_i\right)$, independently, for $i\geq 1$, where $\xi\sim U(0,1)$. 
Thus, $X_i$ are conditionally independent given $\xi$, but unconditionally, they are dependent. As in the case
of the independent normal example, we assume that $\mu_i\sim N\left(0,\phi^2_i\right)$ and $\sigma^2_i\sim\mathcal E(\vartheta_i)$. 
Hence, we now deal with a dependent, hierarchical
normal setup for the $X_i$.
Since given $\xi$, Kolmogorov's three series theorem is applicable and the series is either convergent or divergent almost surely, integrating over the finite random
variable $\xi$ does not alter the convergence properties, in spite of dependence.
To see this, note that if almost surely $\sum_{i=1}^{\infty}X_i<\infty$ given $\xi$, then letting $P$ stand for the probability of events corresponding to $X_i$ 
as well as the probability measure associated with $\xi$, the following hold:
\begin{align}
	P\left(\sum_{i=1}^{\infty}X_i<\infty\right)&=\int P\left(\sum_{i=1}^{\infty}X_i<\infty\bigg |\xi\right)dP(\xi)\notag\\
	&=\int 1\times dP(\xi)\notag\\
	&=1.\notag
\end{align}
Similarly, if $\sum_{i=1}^{\infty}X_i=\infty$ almost surely, given $\xi$, then
\begin{align}
	P\left(\sum_{i=1}^{\infty}X_i=\infty\right)&=\int P\left(\sum_{i=1}^{\infty}X_i=\infty\bigg |\xi\right)dP(\xi)\notag\\
	&=\int 1\times dP(\xi)\notag\\
	&=1.\notag
\end{align}

Setting $\epsilon=0.001$, as in the independent normal case
we considered the upper bound $c_j=\left|S^{\tilde\theta}_{j,n_j}\right|+\frac{0.1}{j}$, with $n_j=10^6$ for $j=1,\ldots,K$, where $K=10^6$. 
VMWare implementation of our parallel codes again takes about $52$ minutes with $100$ cores.
Convergence analyses for our dependent normal distribution are provided in Figure \ref{fig:example_normal_dependent}. 
Again, convergence behaviour of the random series are correctly determined, but as is evident from the figures, the rates of convergence and divergence turned
out to be very slow in general. All these figures depict the posterior means in the last $5\times 10^5$ 
iterations of a total $10^6$ iterations.
\begin{figure}
\centering
\subfigure [Divergence.]{ \label{fig:normal_dep1}
\includegraphics[width=6cm,height=5cm]{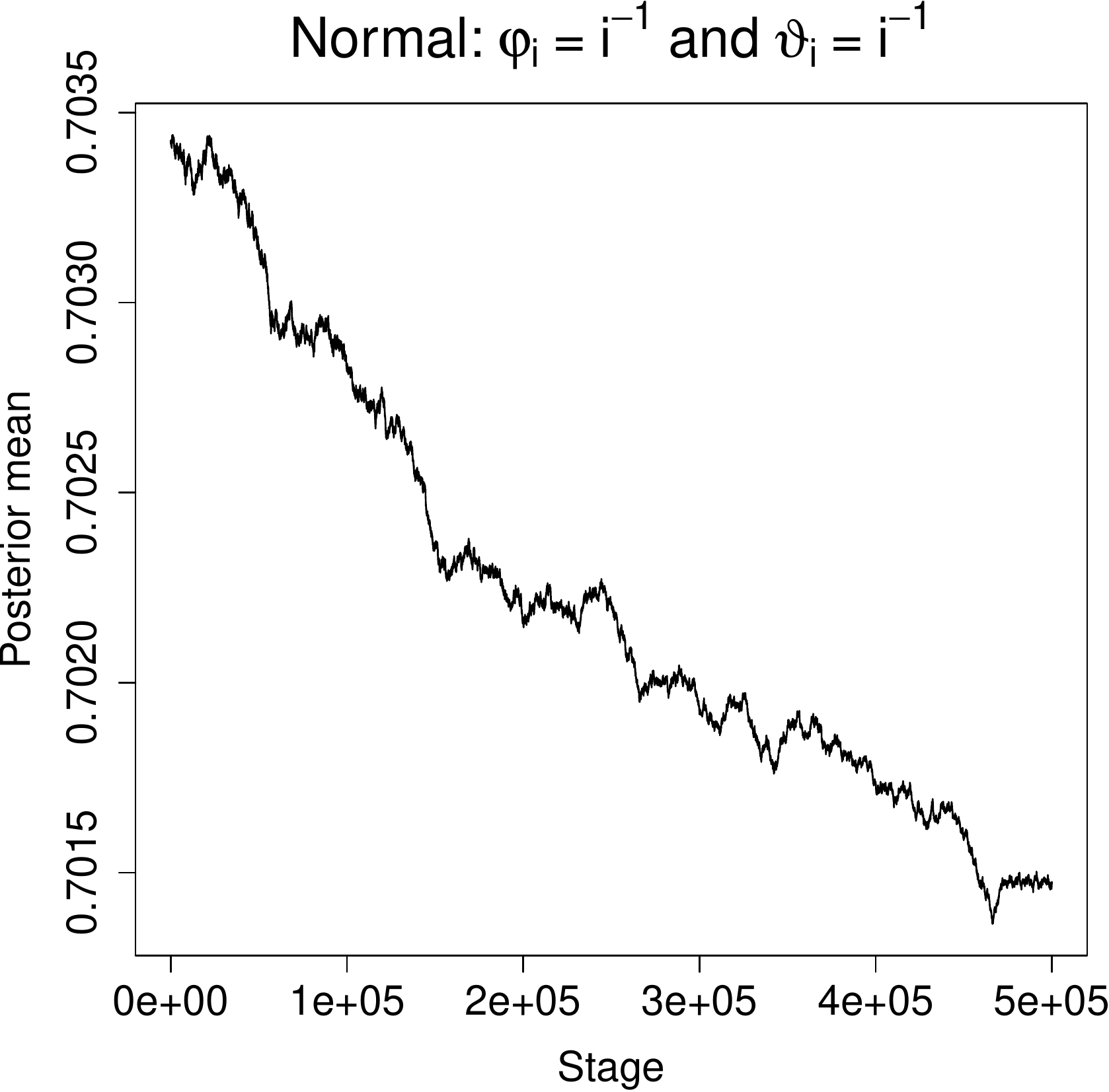}}
\hspace{2mm}
\subfigure [Convergence.]{ \label{fig:normal_dep2}
\includegraphics[width=6cm,height=5cm]{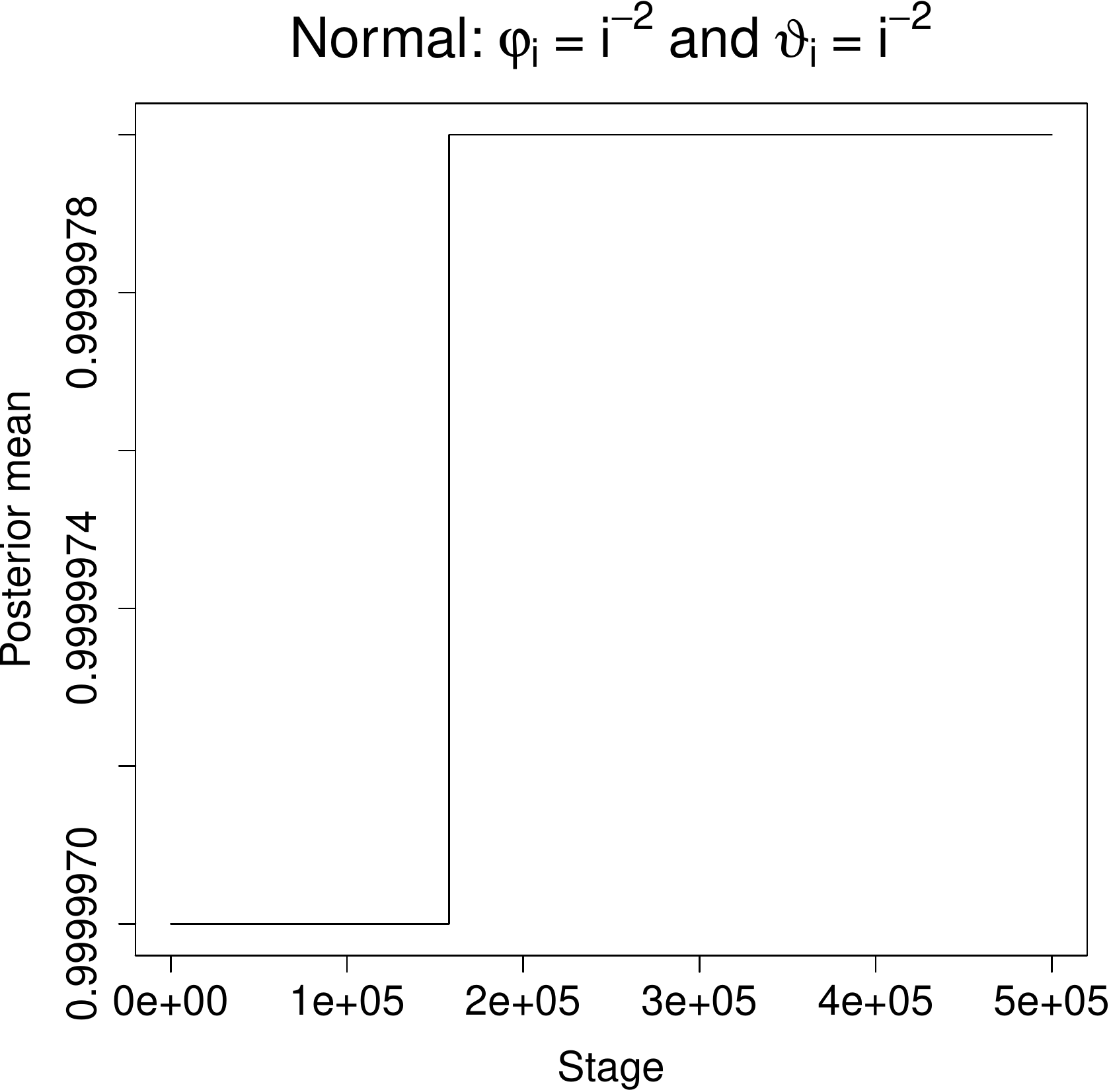}}\\
\subfigure [Convergence.]{ \label{fig:normal_dep3}
\includegraphics[width=6cm,height=5cm]{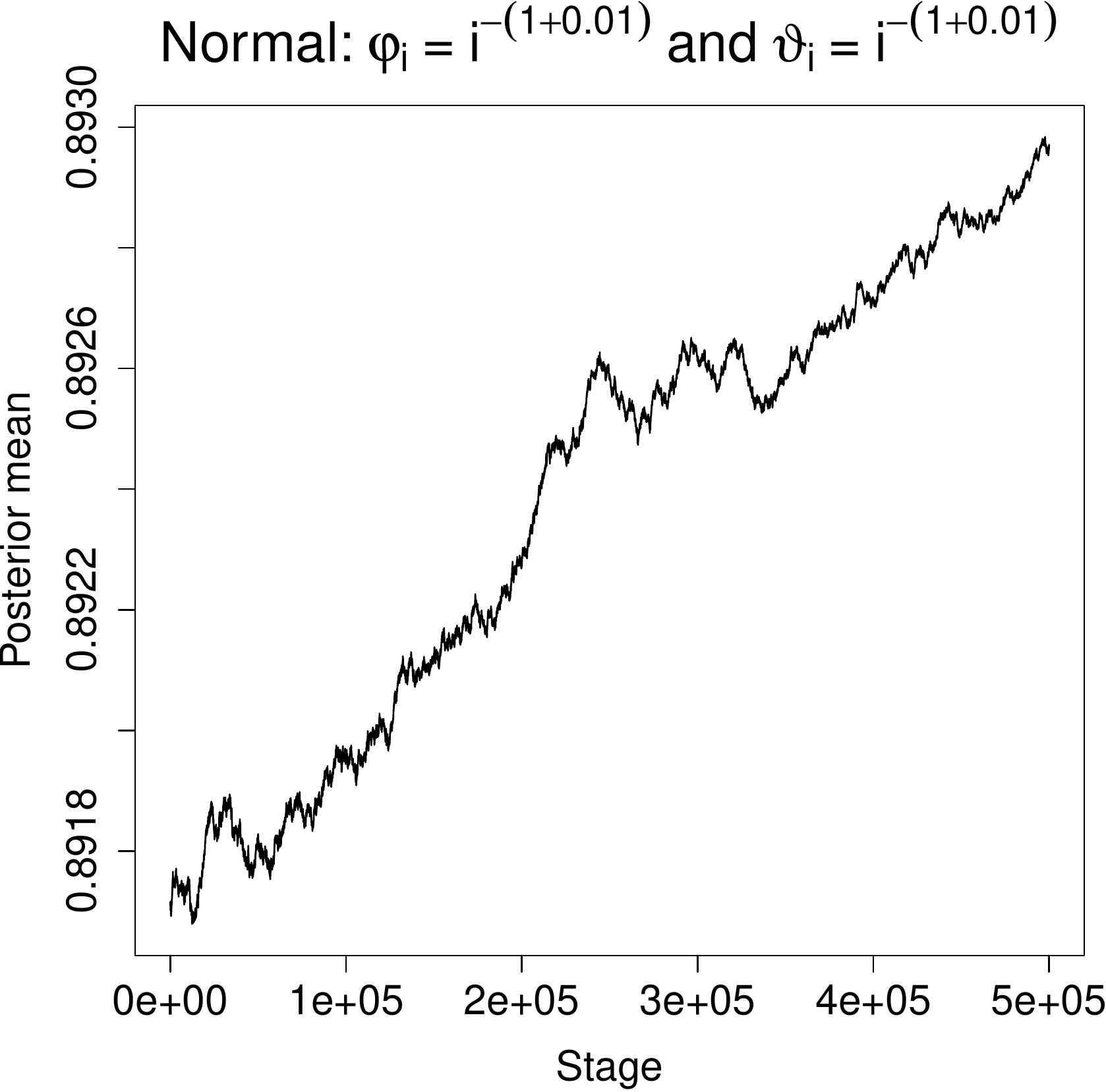}}
\hspace{2mm}
\subfigure [Divergence.]{ \label{fig:normal_dep4}
\includegraphics[width=6cm,height=5cm]{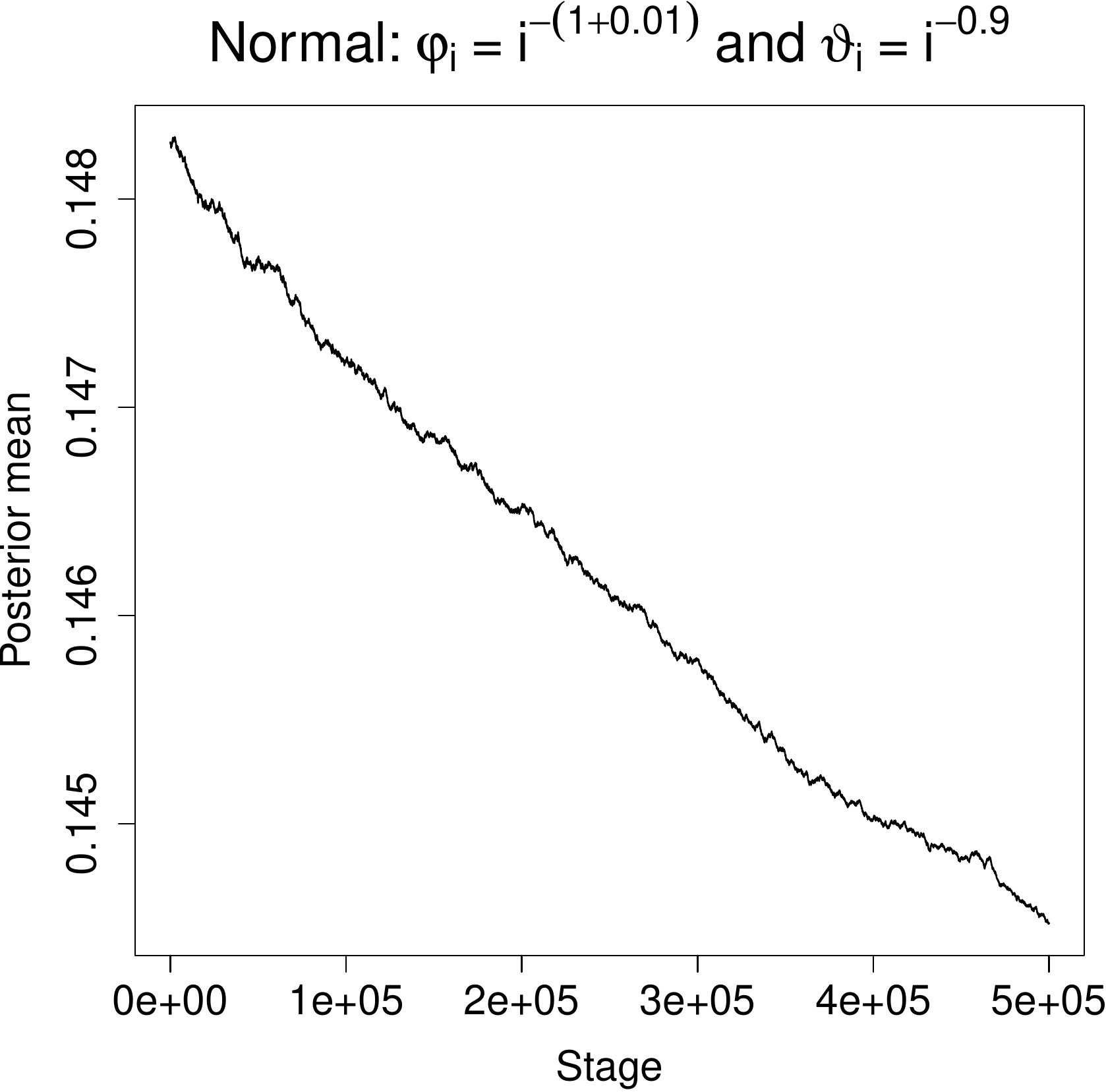}}\\
\subfigure [Divergence.]{ \label{fig:normal_dep5}
\includegraphics[width=6cm,height=5cm]{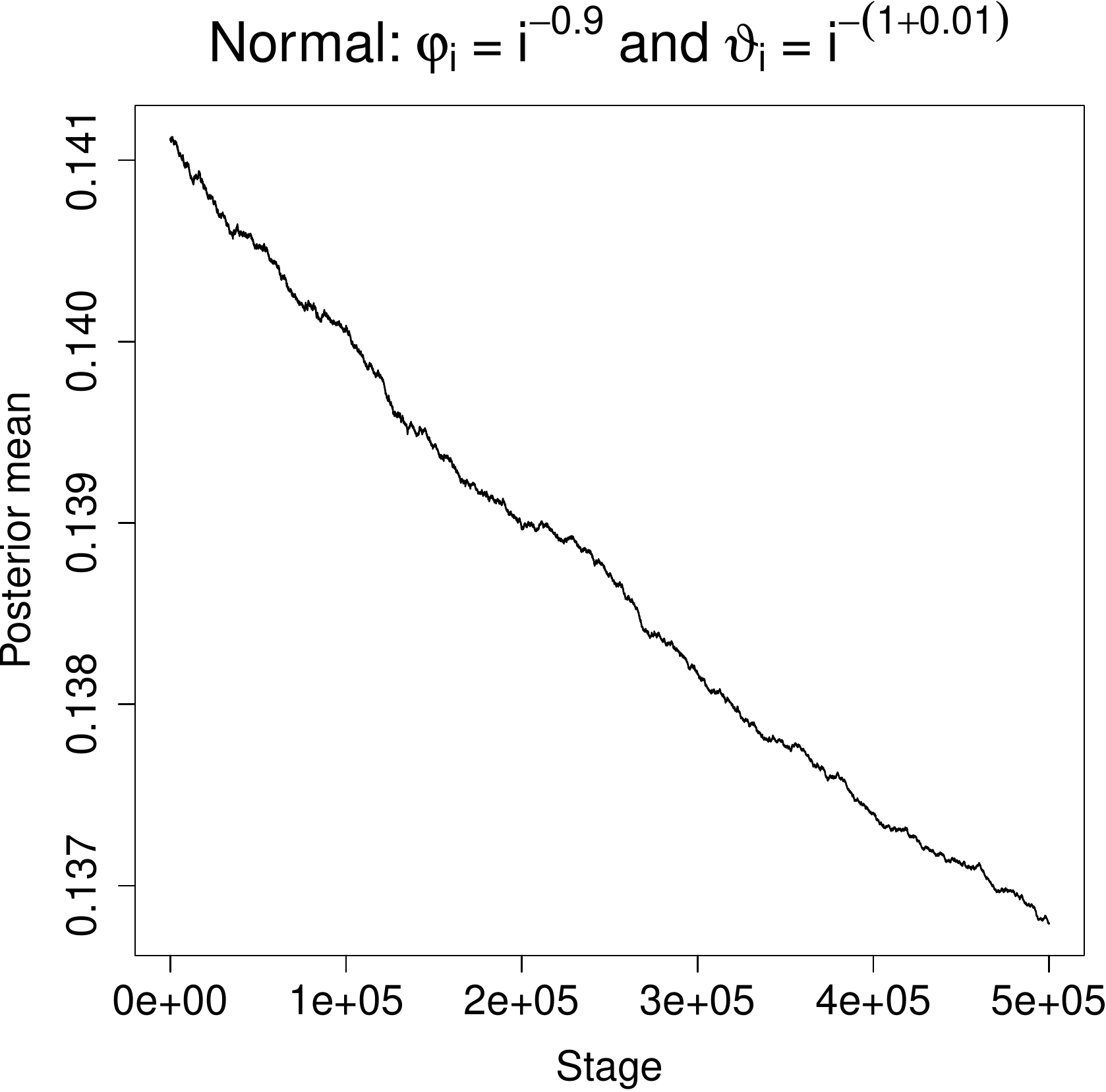}}
\hspace{2mm}
\subfigure [Divergence.]{ \label{fig:normal_dep6}
\includegraphics[width=6cm,height=5cm]{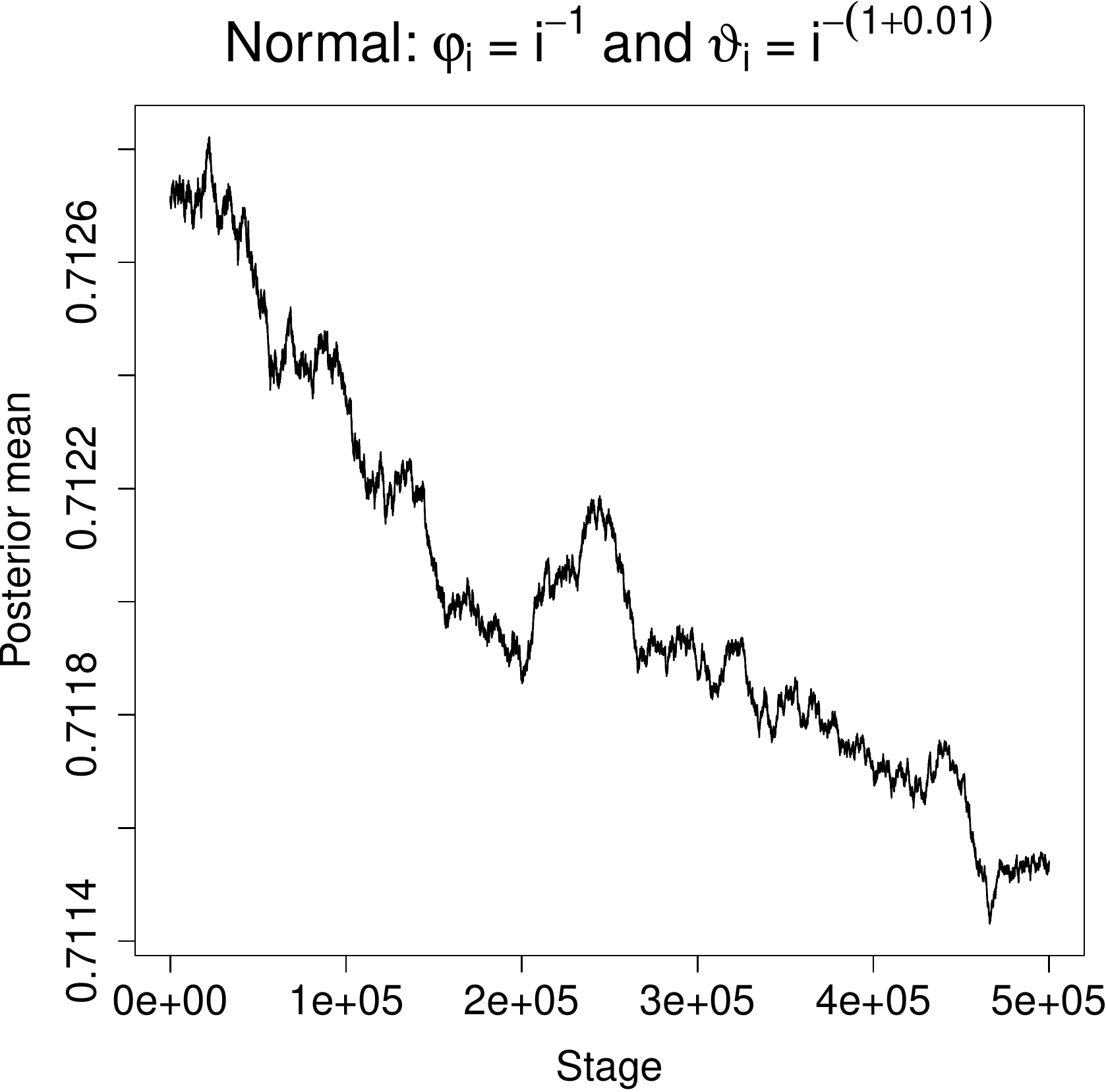}}\\
\subfigure [Divergence.]{ \label{fig:normal_dep7}
\includegraphics[width=6cm,height=5cm]{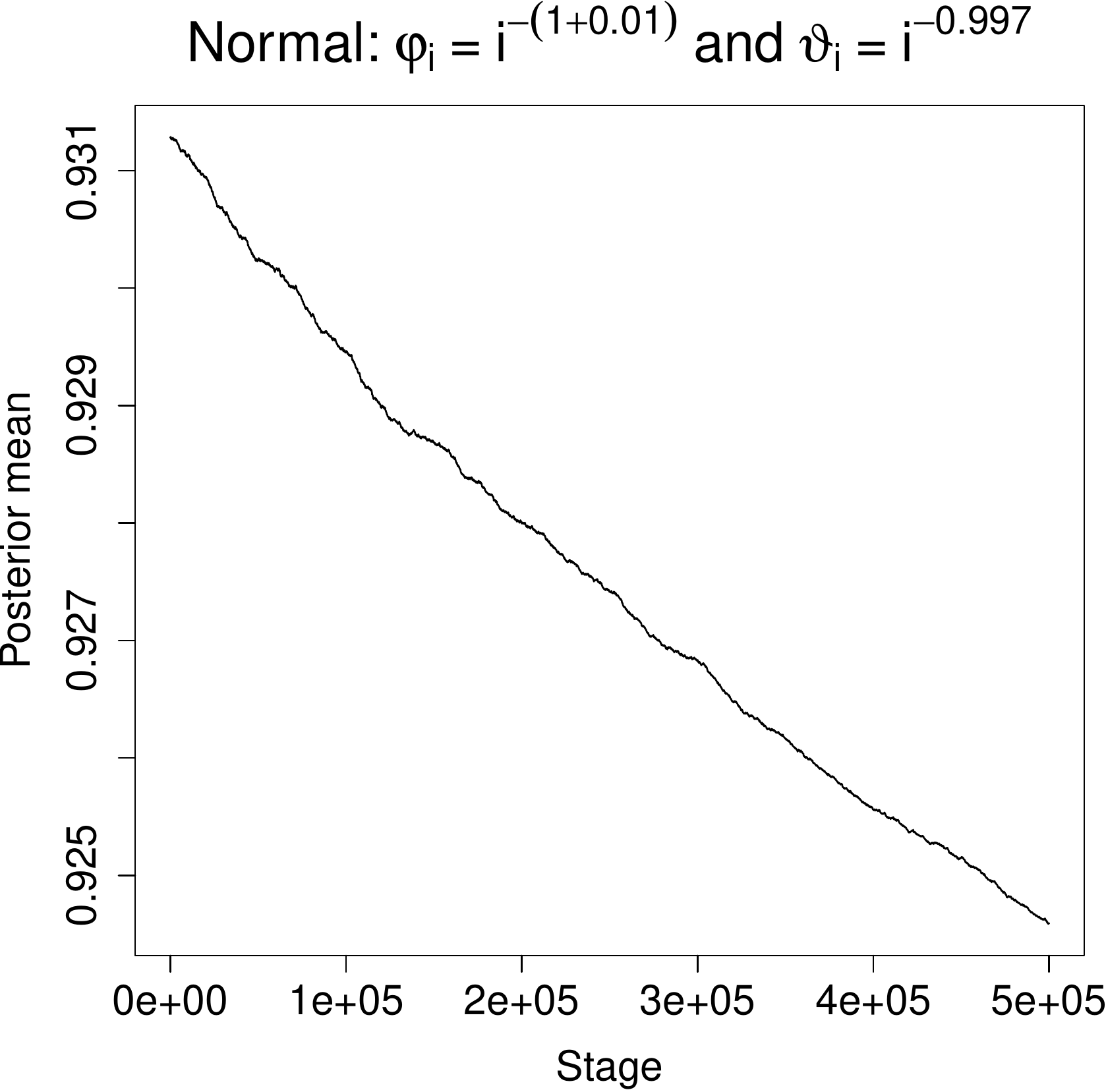}}
\caption{Example 3: Convergence and divergence for dependent normal series.}
\label{fig:example_normal_dependent}
\end{figure}

\subsection{Example 4: Dependent state-space random series}
\label{subsec:bayesian_ss}
We now consider the following random series:
\begin{equation}
\sum_{i=1}^{\infty}X_i\theta_i,
\label{eq:ss1}
\end{equation}
where for $i\geq 1$, $\theta_i\sim \mathcal E(\psi_i)$ independently, 
and $X_i$ admits the following state-space representation: 
\begin{align}
	X_i&=\alpha+\beta Z_i+\epsilon_i;\label{eq:ss2}\\
	Z_i&=\rho Z_{i-1}+\eta_i,\label{eq:ss3}
\end{align}
where $Z_0, \alpha,\beta,\rho\stackrel{iid}{\sim} U(a,b)$, $a=\varepsilon$, $b=\varepsilon+1$, with $\varepsilon>0$, 
and $\epsilon_i,\eta_i\stackrel{iid}{\sim}N(0,1)\mathbb I_{[a,b]}$,
that is the standard normal distribution truncated on $[a,b]$. It follows from the above representation that $X_i$ are dependent, positive, and bounded random variables.
Thus, the terms $X_i\theta_i$ in (\ref{eq:ss1}) are also dependent, positive, but unbounded random variables. Since $X_i$ are both upper and lower bounded, 
the convergence properties of (\ref{eq:ss1}) are dictated by the $\theta_i$'s. 

In our simulation experiment, we generate $\theta_i$ and
$X_i$ following the above model specifications, setting $\varepsilon=0.001$. 
Thus, data $Y_i=X_i\theta_i$, for $i\geq 1$, are available for convergence analysis of (\ref{eq:ss1}). 

Since the exponential distribution dominates the convergence properties in this case, mathematically valid bound construction for the partial sums is possible in this case.
Here we provide the details of our bound construction procedure. 
We first generate 
$X^*_i$ following (\ref{eq:ss2}) and (\ref{eq:ss3}) and 
set $Y_i=X^*_i\theta_i$, with $\theta_i=-\psi_i\log U_i$. Combining these yields $\log U_i=-Y_i/(\psi_iX^*_i)$. We then set 
$\tilde Y_i=X^*_i\tilde\theta_i$, where $\tilde\theta_i=-r^{-1}_i(\epsilon)\log U_i=(r^{-1}_i(\epsilon)Y_i)/(\psi_iX^*_i)$; as before, set set $\epsilon=0.001$. 
Letting $S^{\tilde\theta}_{j,n_j}$ be 
the partial sums associated with $\left\{\tilde Y_i\right\}_{i=1}^{\infty}$, we set $c_j=S^{\tilde\theta}_{j,n_j}$ as the upper bound for the partial sums
associated with $\left\{Y_i\right\}_{i=1}^{\infty}$.


In this setup, as in Section \ref{subsec:bayesian_exp} for the hierarchical exponential series, 
we set $n_j=1000$ for $j=1,\ldots,K$, where $K=2000$. 
As before, with such small sample size, parallel implementation of this setup on our dual-core laptop takes less than a second to yield the results.

Figure \ref{fig:example_ss} shows that the convergence behaviour of the random series are correctly and convincingly 
determined in all the cases despite the small sample sizes. 
\begin{figure}
\centering
\subfigure [Divergence.]{ \label{fig:ss1}
\includegraphics[width=6cm,height=5cm]{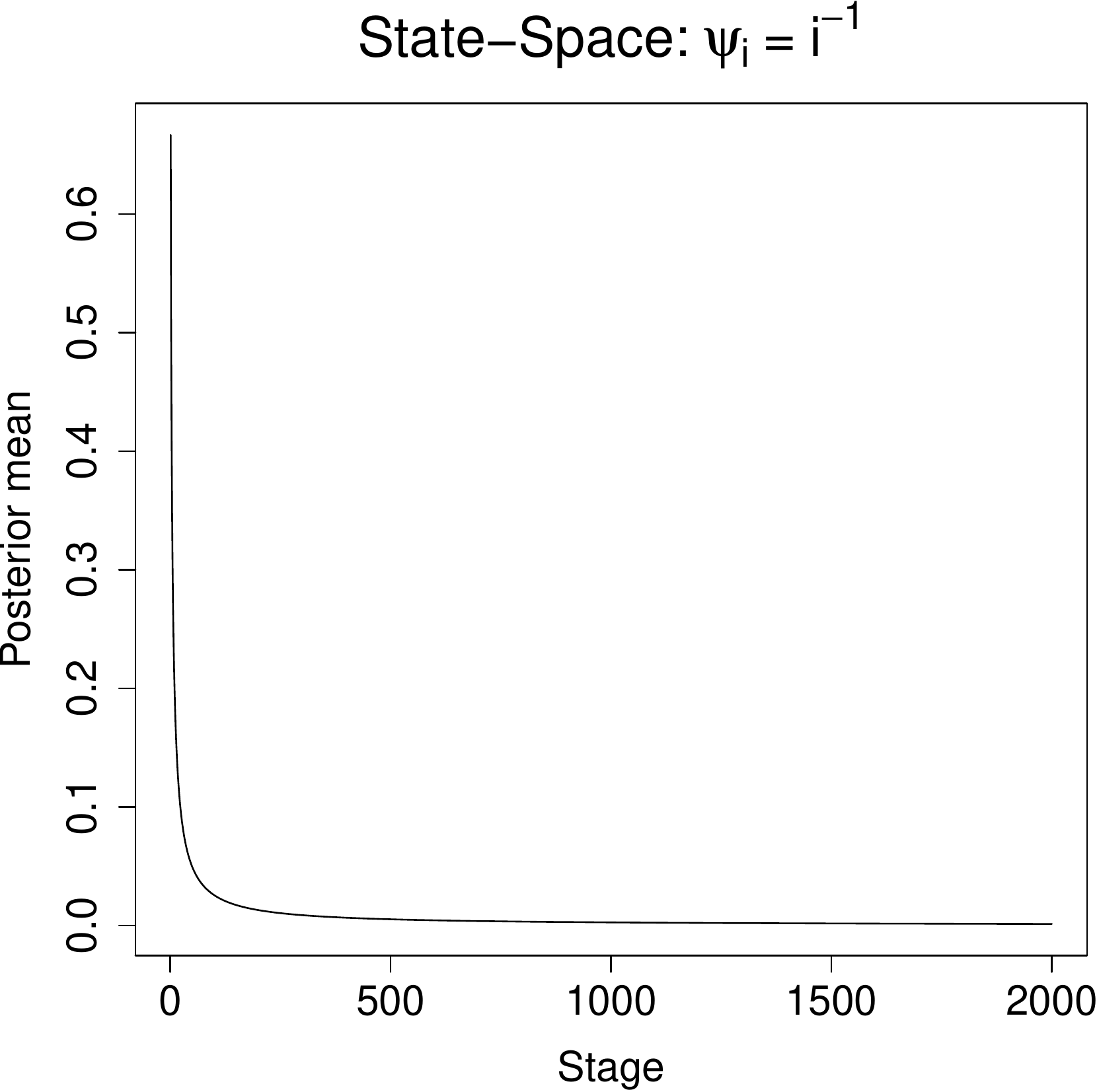}}
\hspace{2mm}
\subfigure [Convergence.]{ \label{fig:ss2}
\includegraphics[width=6cm,height=5cm]{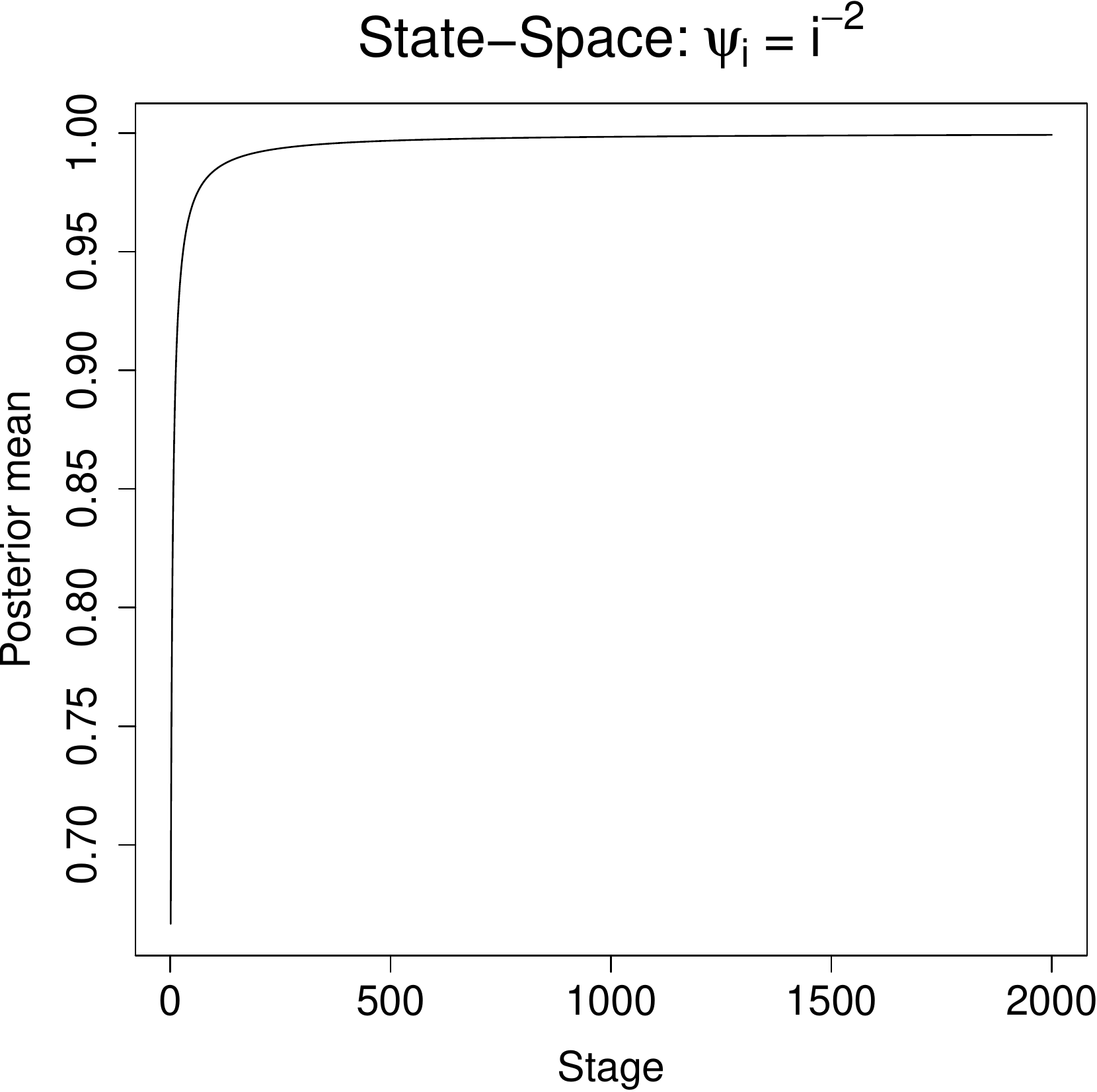}}\\
\subfigure [Convergence.]{ \label{fig:ss3}
\includegraphics[width=6cm,height=5cm]{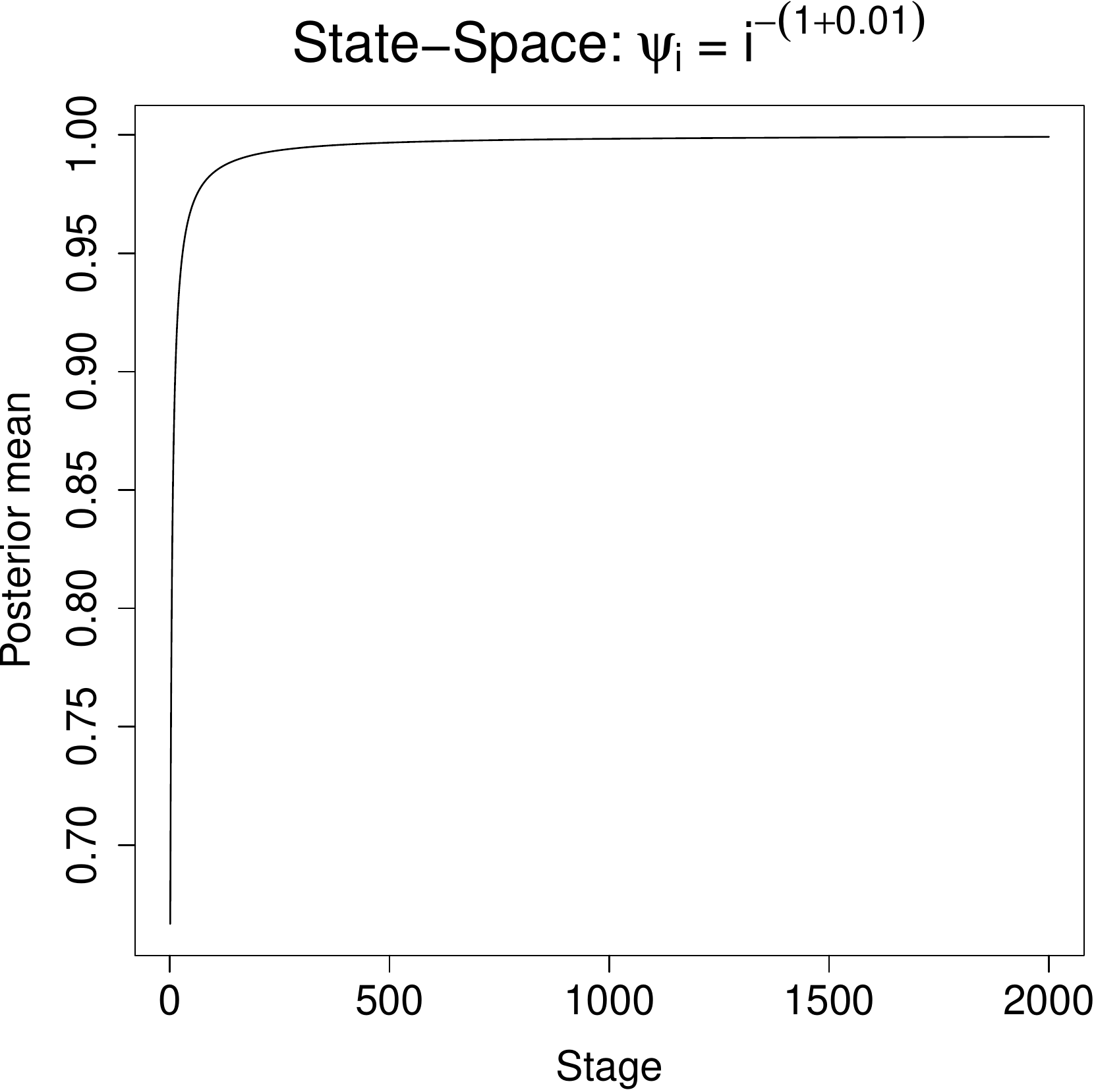}}
\hspace{2mm}
\subfigure [Convergence.]{ \label{fig:ss4}
\includegraphics[width=6cm,height=5cm]{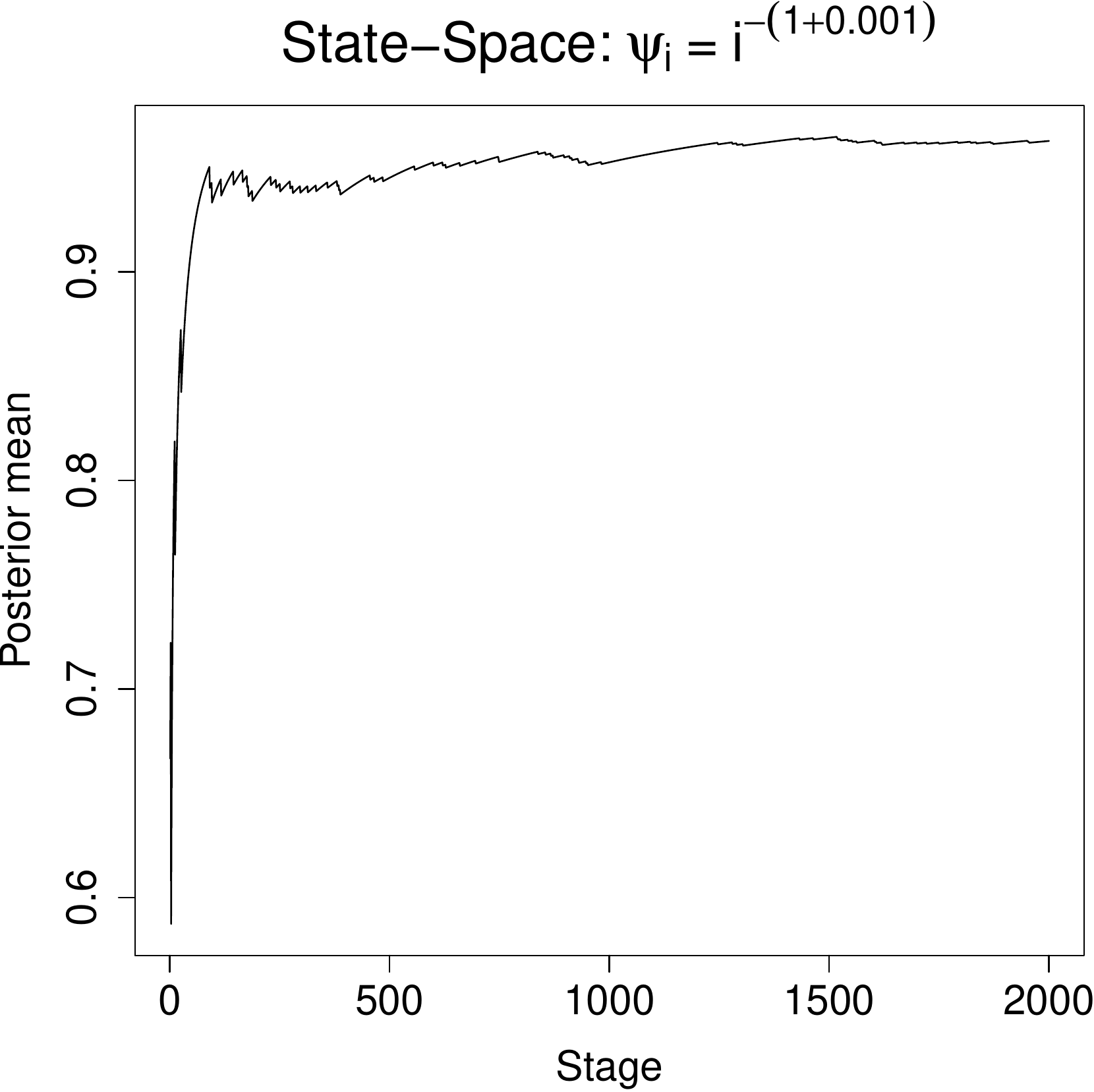}}\\
\subfigure [Convergence.]{ \label{fig:ss5}
\includegraphics[width=6cm,height=5cm]{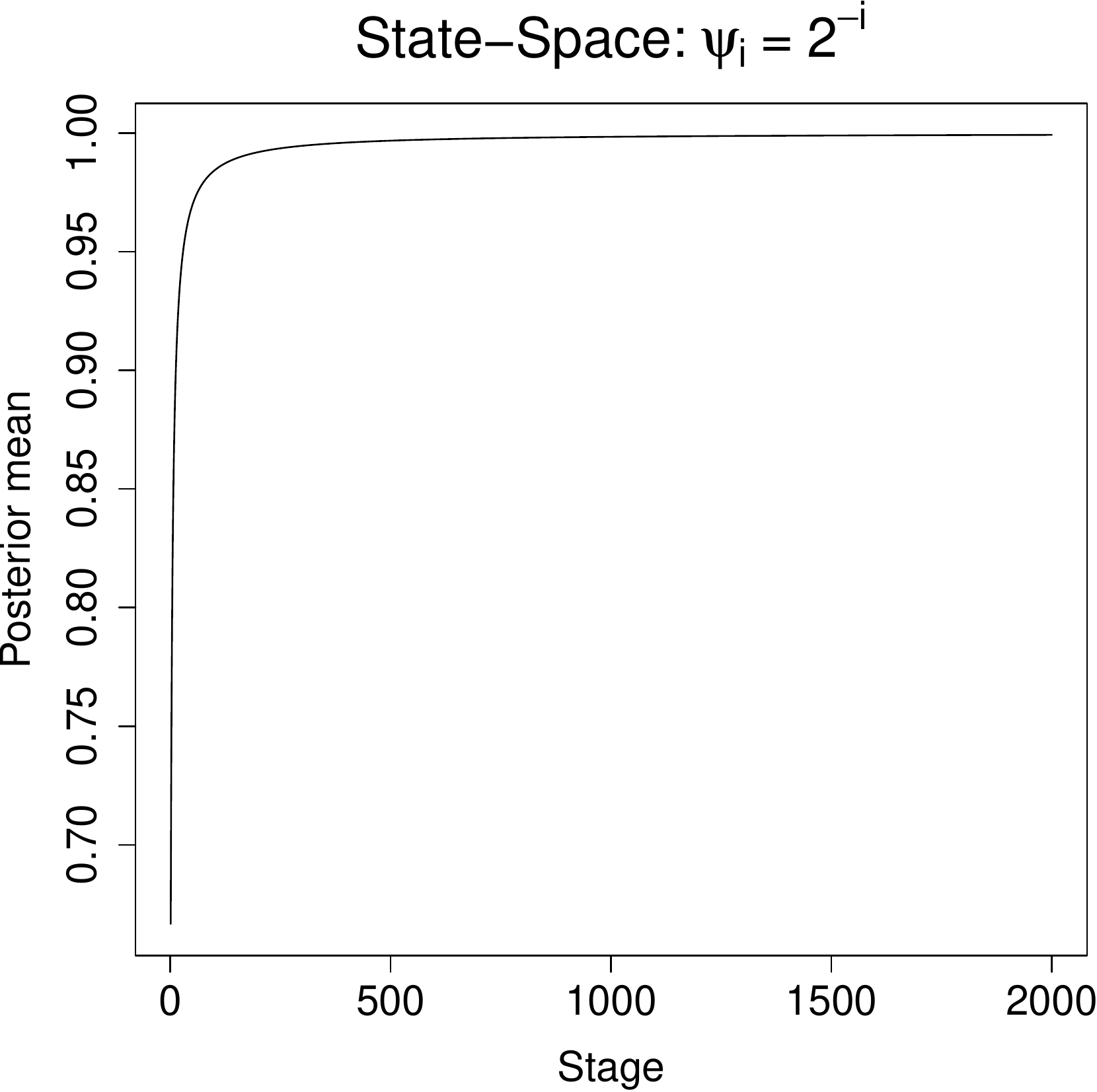}}
\hspace{2mm}
\subfigure [Convergence.]{ \label{fig:ss6}
\includegraphics[width=6cm,height=5cm]{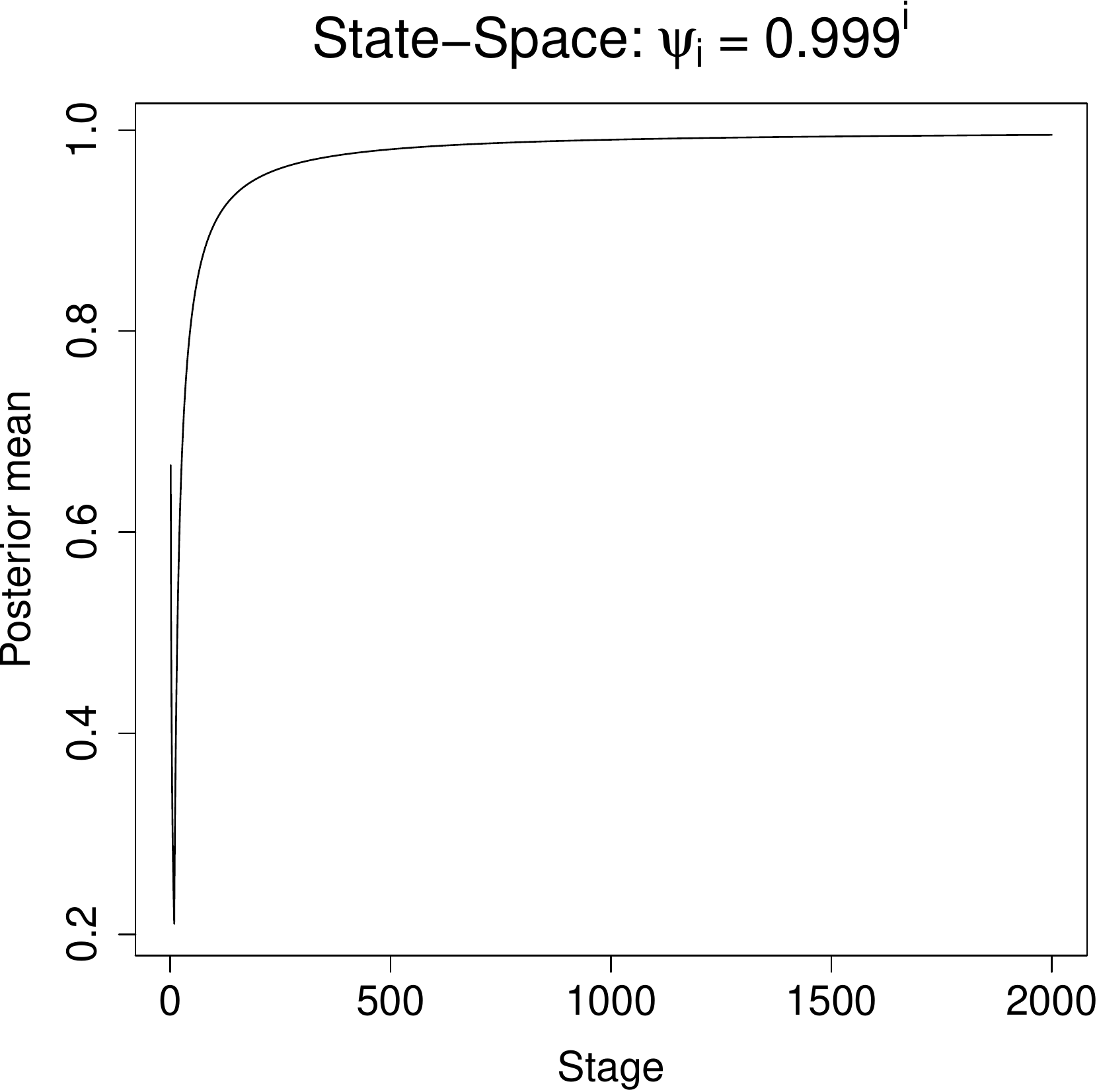}}\\
\subfigure [Divergence.]{ \label{fig:ss7}
\includegraphics[width=6cm,height=5cm]{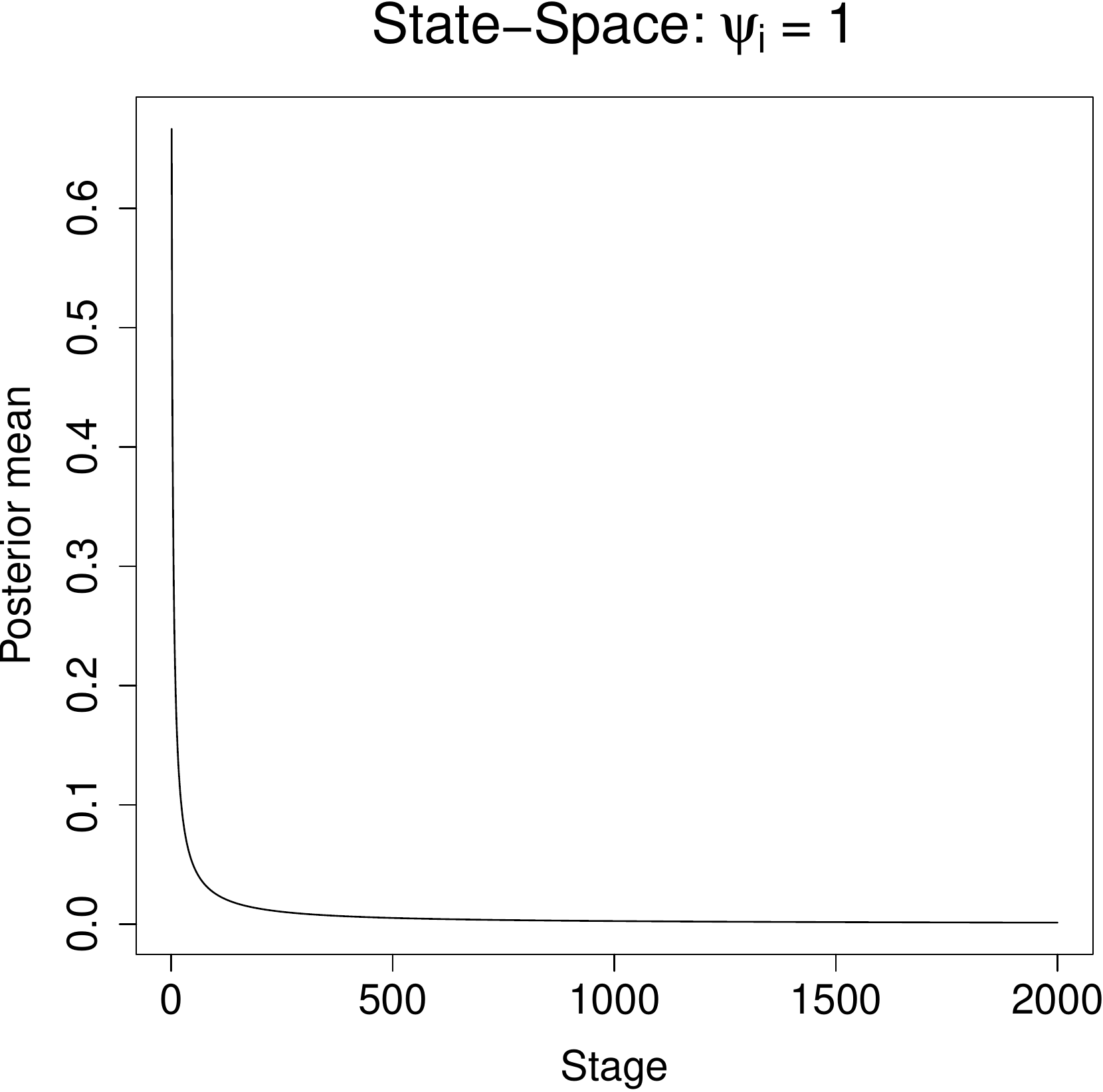}}
\caption{Example 4: Convergence and divergence for state-space series.}
\label{fig:example_ss}
\end{figure}

\subsection{Example 5: Dependent state-space random series with hierarchical exponential distribution}
\label{subsec:bayesian_ss2}

In the state-space setup of Section \ref{subsec:bayesian_ss} we considered $\theta_i\sim \mathcal E(\psi_i)$. Now we add an extra hierarchy to the
exponential distribution by specifying, as in Section \ref{subsec:bayesian_exp}, 
that $\theta_i\sim \mathcal E(\vartheta_i)$ and $\vartheta_i\sim\mathcal E(\psi_i)$. Thus, this state-space model is dominated by the hierarchical exponential distribution.

As before, let $Y_i=X_i\theta_i$ be available. In our simulation experiment, we generate $\theta_i$ and
$X_i$ following the hierarchical exponential driven state-space model specifications, setting $\varepsilon=0.001$. 

To obtain the bound $c_j$ for the partial sums, we employ the following strategy.
We first generate $X^*_i$ following (\ref{eq:ss2}) and (\ref{eq:ss3}) and 
set $Y_i=X^*_i\theta_i$, with $\theta_i=-\vartheta_i\log U_i$. Combining these yields $\log U_i=-Y_i/(\vartheta_iX^*_i)$, where 
$\vartheta_i=-\psi_i\log U^*_i$. Here $U_i$ and $U^*_i$ are mutually independent $iid$ $U(0,1)$ random variables for $i\geq 1$.
We then set 
$\tilde Y_i=X^*_i\tilde\theta_i$, where $\tilde\theta_i=-\tilde\vartheta_i\log U_i$, and
$\tilde\vartheta_i=-r^{-1}_i(\epsilon)\log U^*_i$; as before, we set $\epsilon=0.001$. 
Combining, we obtain $\tilde Y_i=Y_i\left(\tilde\vartheta_i/\vartheta_i\right)$.
Letting $S^{\tilde\theta}_{j,n_j}$ be 
the partial sums associated with $\left\{\tilde Y_i\right\}_{i=1}^{\infty}$, we set $c_j=S^{\tilde\theta}_{j,n_j}$ as the upper bounds for the partial sums
associated with $\left\{Y_i\right\}_{i=1}^{\infty}$.

As before, we set $n_j=1000$, for $j=1,\ldots,K$, where $K=2000$, and our parallel computing procedure implemented in our laptop takes less than a second to complete
each exercise.

Figure \ref{fig:example_ssp2} shows that in all the cases, our Bayesian procedure correctly detects convergence and divergence of the underlying series, even with
such small sample size.

\begin{figure}
\centering
\subfigure [Divergence.]{ \label{fig:ssp1}
\includegraphics[width=6cm,height=5cm]{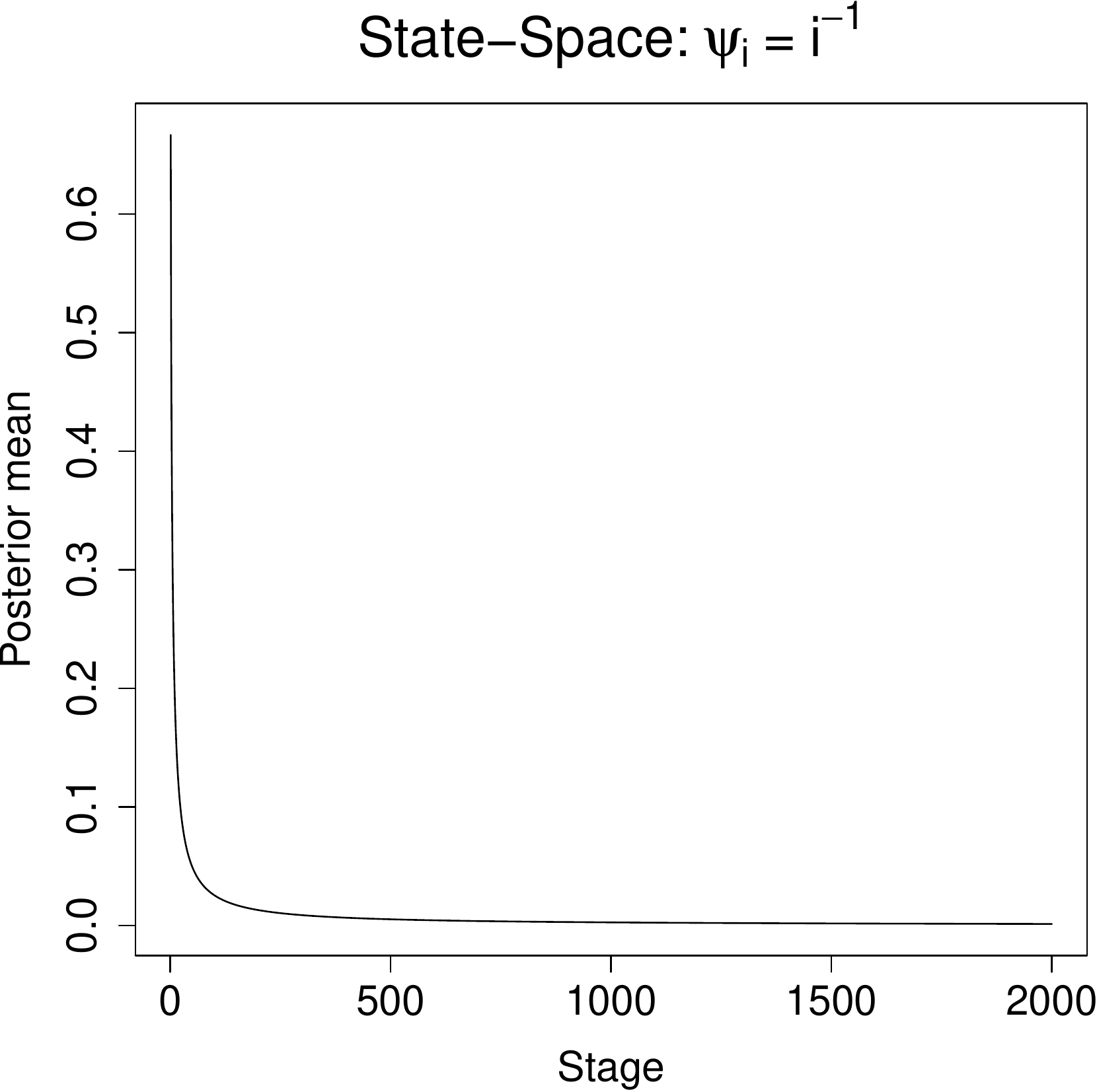}}
\hspace{2mm}
\subfigure [Convergence.]{ \label{fig:ssp2}
\includegraphics[width=6cm,height=5cm]{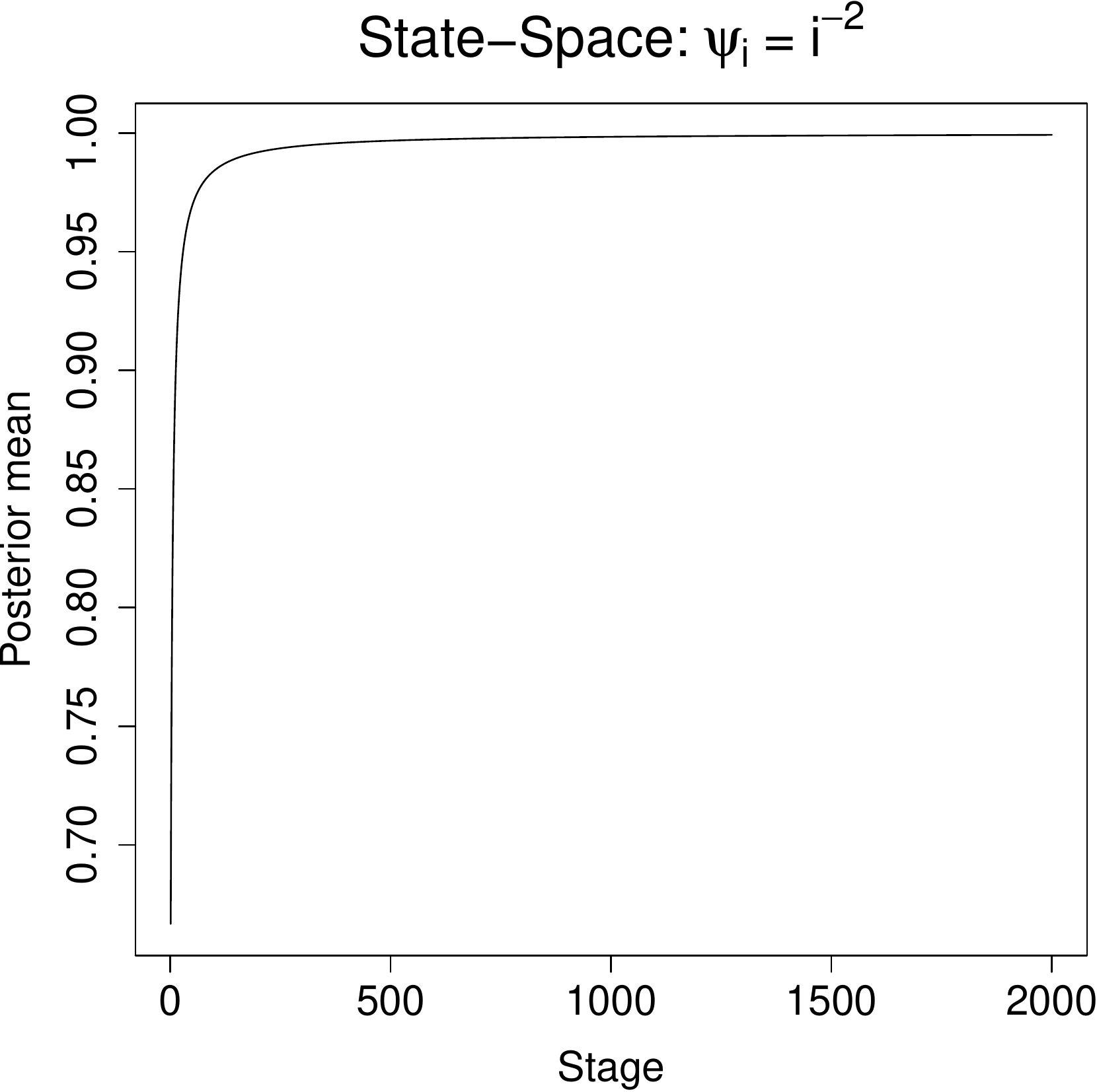}}\\
\subfigure [Convergence.]{ \label{fig:ssp3}
\includegraphics[width=6cm,height=5cm]{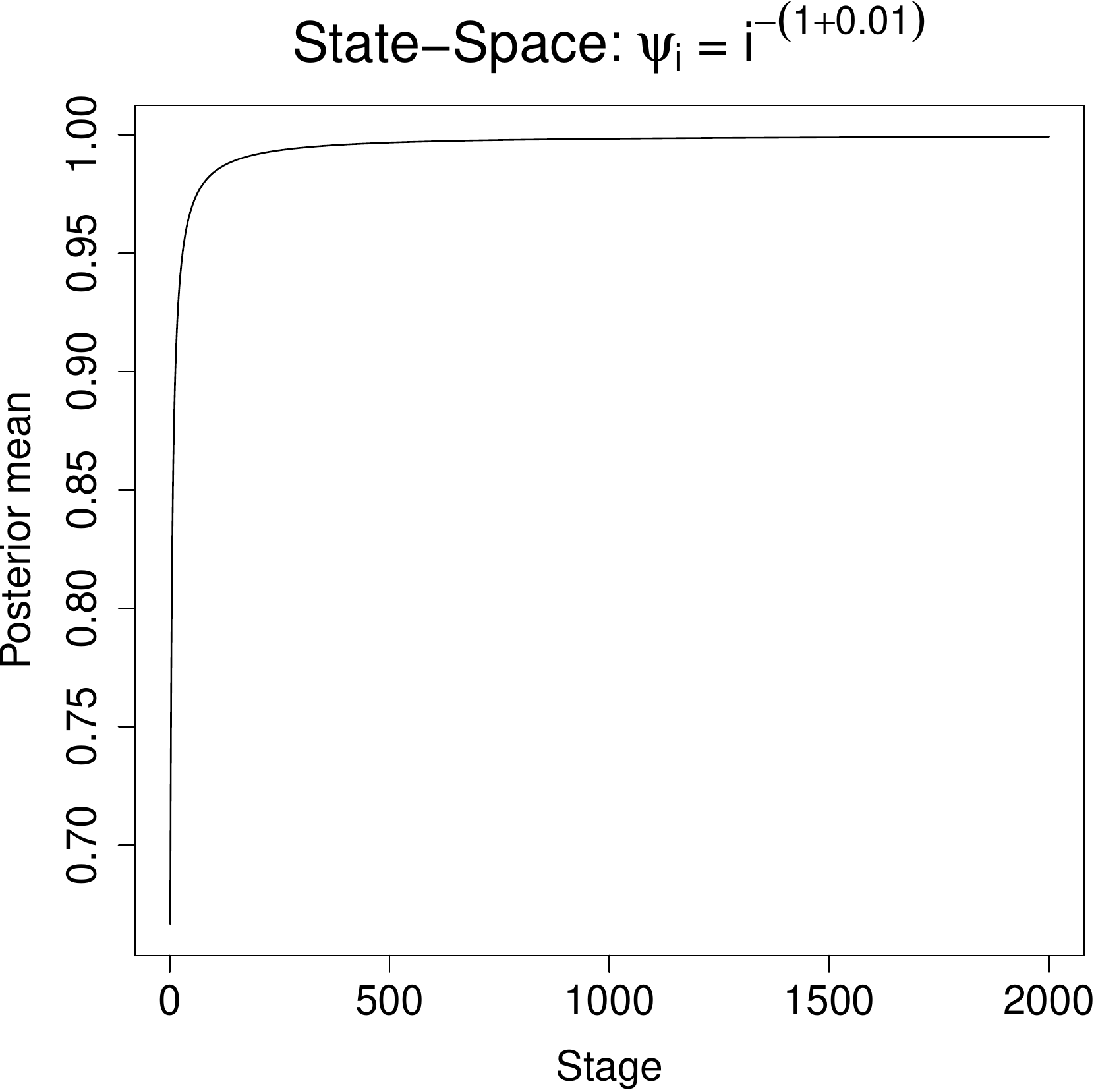}}
\hspace{2mm}
\subfigure [Convergence.]{ \label{fig:ssp4}
\includegraphics[width=6cm,height=5cm]{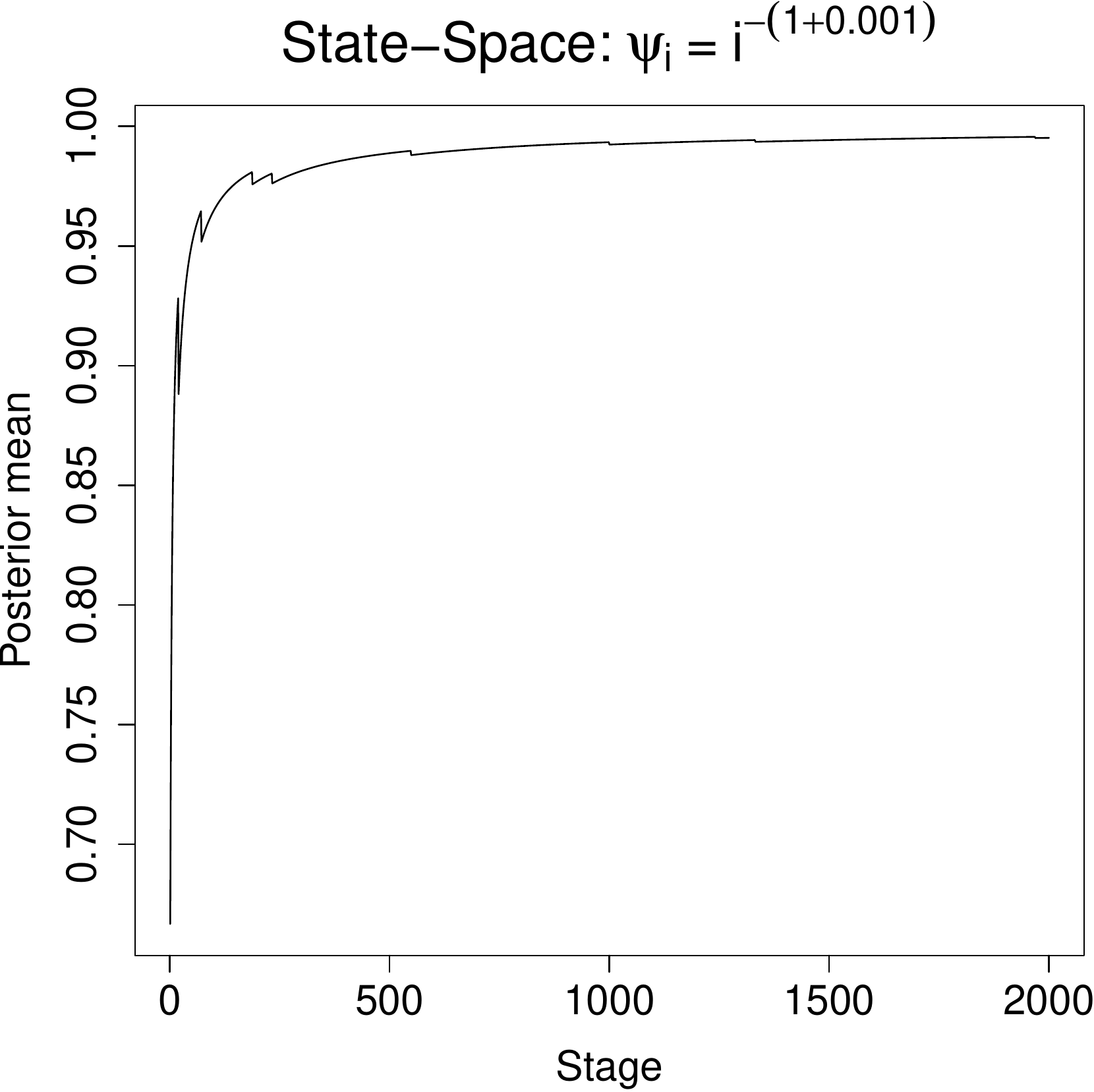}}\\
\subfigure [Convergence.]{ \label{fig:ssp5}
\includegraphics[width=6cm,height=5cm]{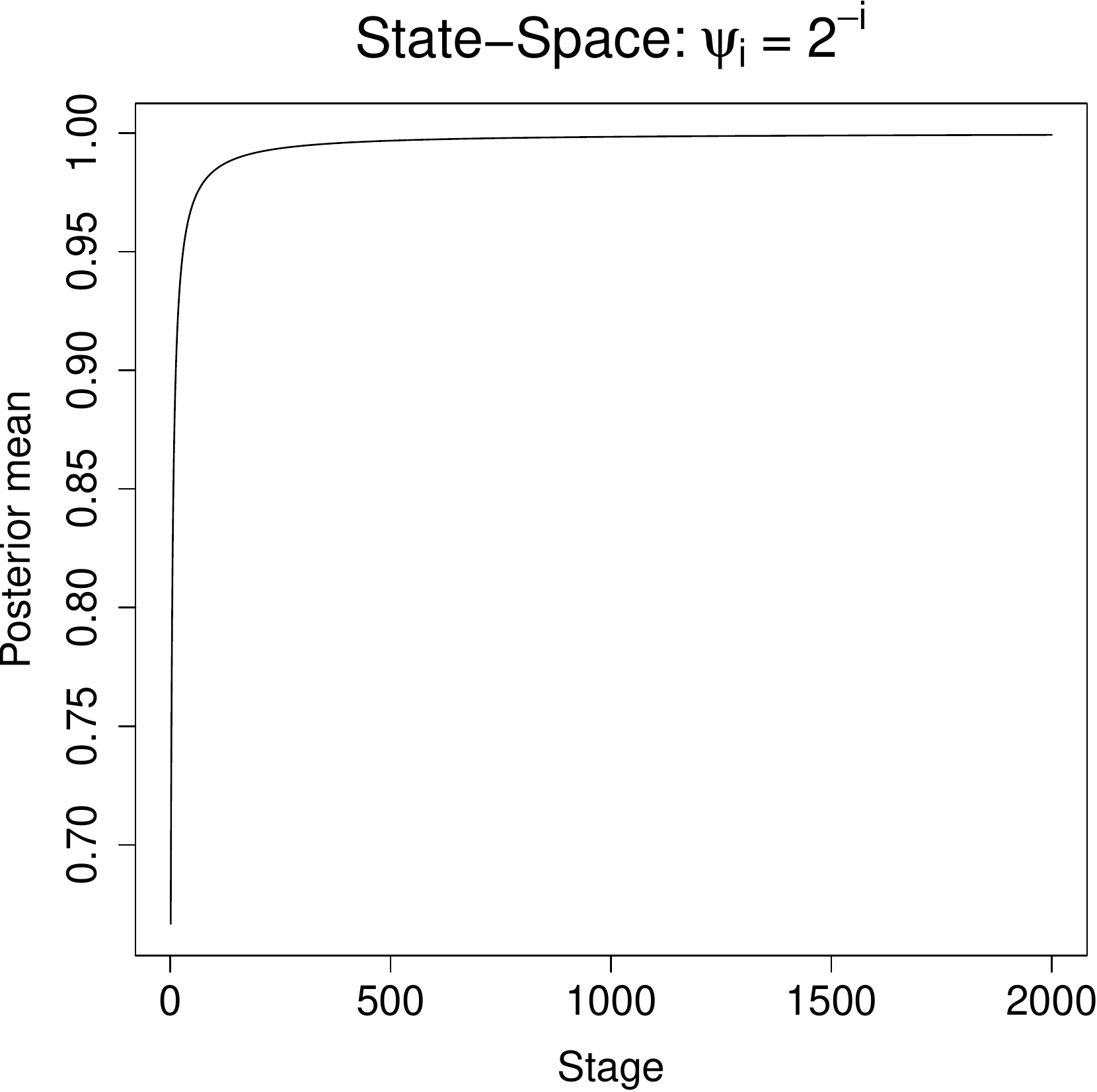}}
\hspace{2mm}
\subfigure [Convergence.]{ \label{fig:ssp6}
\includegraphics[width=6cm,height=5cm]{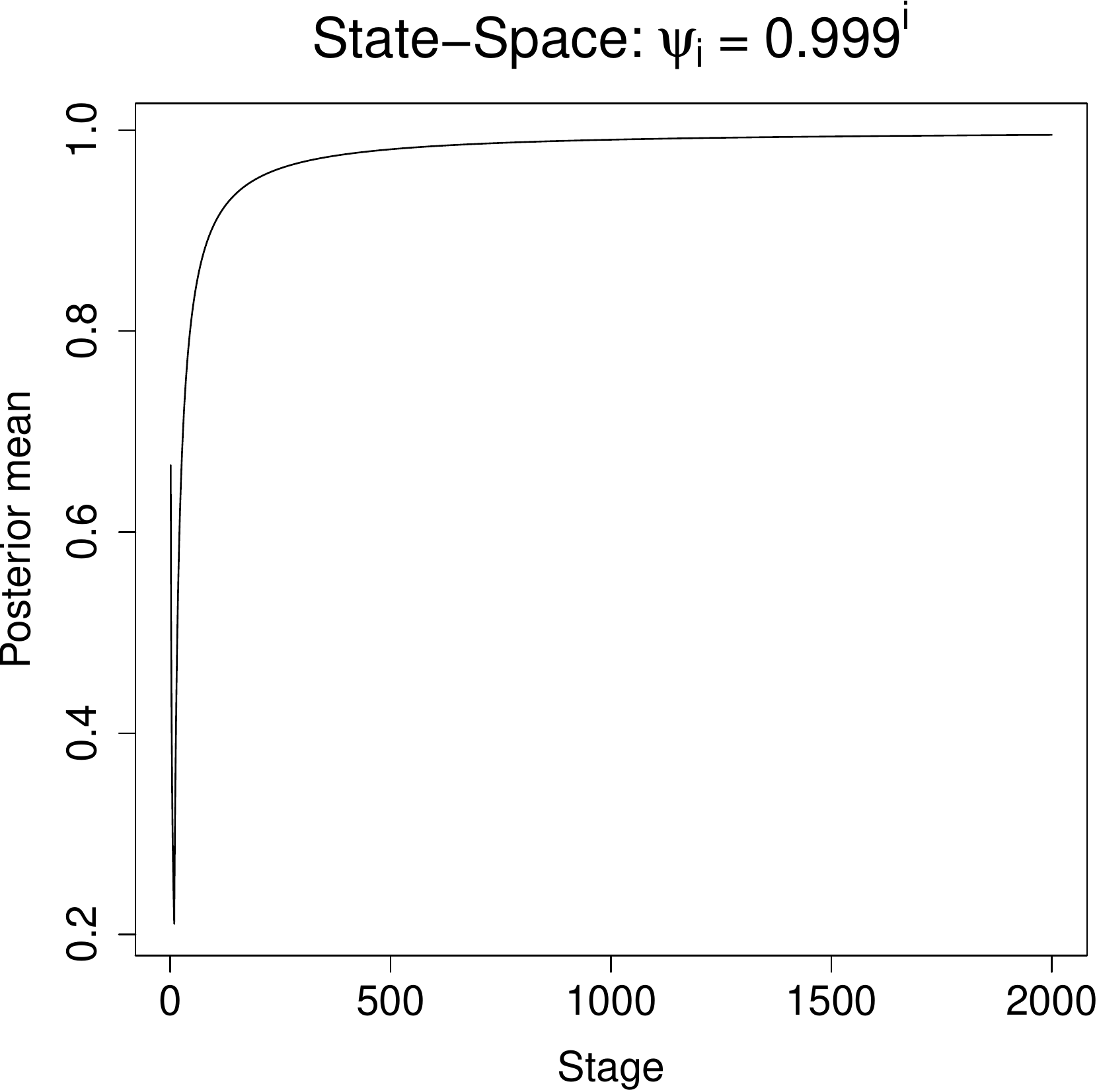}}\\
\subfigure [Divergence.]{ \label{fig:ssp7}
\includegraphics[width=6cm,height=5cm]{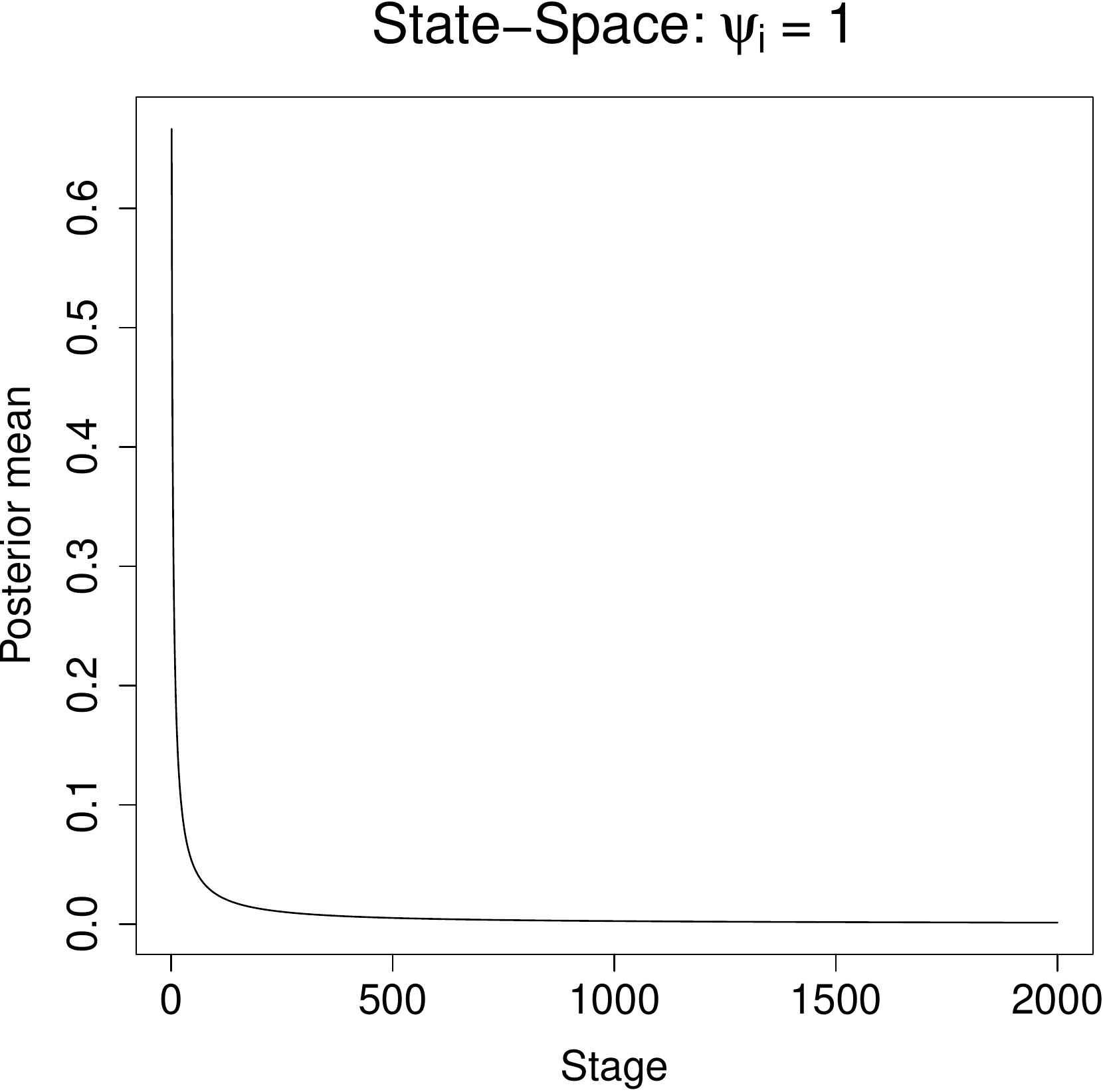}}
\caption{Example 5: Convergence and divergence for state-space series with hierarchical exponential distribution.}
\label{fig:example_ssp2}
\end{figure}

\subsection{Example 6: Random Dirichlet series}
\label{subsec:bayesian_ds}
Consider the random Dirichlet series (RDS) given by
\begin{equation}
	\sum_{i=1}^{\infty}\frac{X_i}{i^p},
	\label{eq:ds}
\end{equation}
where $X_i$ are $iid$ random variables taking values $-1$ and $1$ with probabilities $1/2$, and $p$ is a real number. Since $|X_i|=1$ almost surely, it follows that
for any $R>0$, there exists $i_0$, such that for $i\geq i_0$, $\frac{X_i}{i^p}<R$, provided $p>0$. Hence, for $p>0$, 
$\mathbb I_{\left\{\frac{|X_i|}{i^p}<R\right\}}=1$ almost surely, for $i\geq i_0$.
With this, it follows by a simple application of Kolmogorov's three series theorem that the random series converges almost surely for $p>1/2$ and diverges almost surely
for $0<p\leq 1/2$. If $p=0$, then the summands of (\ref{eq:ds}) are $iid$ and hence (\ref{eq:ds}) diverges. Now, if $p\in(-\infty,0)$, then for any $R>0$,
there exists $i_0\geq 1$ such that $P\left(\frac{\left|X_i\right|}{i^p}>R\right)=1$, for $i\geq i_0$. Hence, 
$\sum_{i=1}^{\infty}P\left(\frac{\left|X_i\right|}{i^p}>R\right)=\infty$, for any $R>0$. Consequently, by Kolmogorov's three series theorem, (\ref{eq:ds}) diverges
for $p\in(-\infty,0)$.
Combining the above arguments it follows that (\ref{eq:ds}) converges almost surely for $p>1/2$ and diverges almost surely for $p\leq 1/2$.

Since $X_i$ takes both positive and negative values with positive probabilities, application of the mathematically valid parametric upper bound is infeasible. 
Hence, we consider application of (\ref{eq:S2})
where $\tilde\theta$ in $S^{\tilde\theta}_{j,n_j}$ corresponds to $p=1+\epsilon$ in this case. Here we set $\epsilon=0.001$ as before.
We experimented with various choices of the tuning parameter $a$ on the right hand side of (\ref{eq:S2}) and all of them yielded the same inference. Hence, we report
our results with respect to $a=1$.

Figure \ref{fig:ds_para} shows the results of our Bayesian application to this problem for various values of $p$, for
$n_j=1000$; $j=1,\ldots,K$, with $K=2000$. Note that for $p=0.501$ (panel (e) of Figure \ref{fig:ds_para}), 
we obtain the wrong result of divergence, whereas convergence is the correct result. 
This is a subtle situation as it may be difficult to distinguish divergence for $p=0.5$ and convergence for $p=0.501$, but wrong results are obtained   
in many cases for $p\in (0.5,0.79)$. Thus, effectiveness of the general upper bound (\ref{eq:S2}) is again challenged in this example.
\begin{figure}
\centering
\subfigure [Divergence.]{ \label{fig:ds_para1}
\includegraphics[width=6cm,height=5cm]{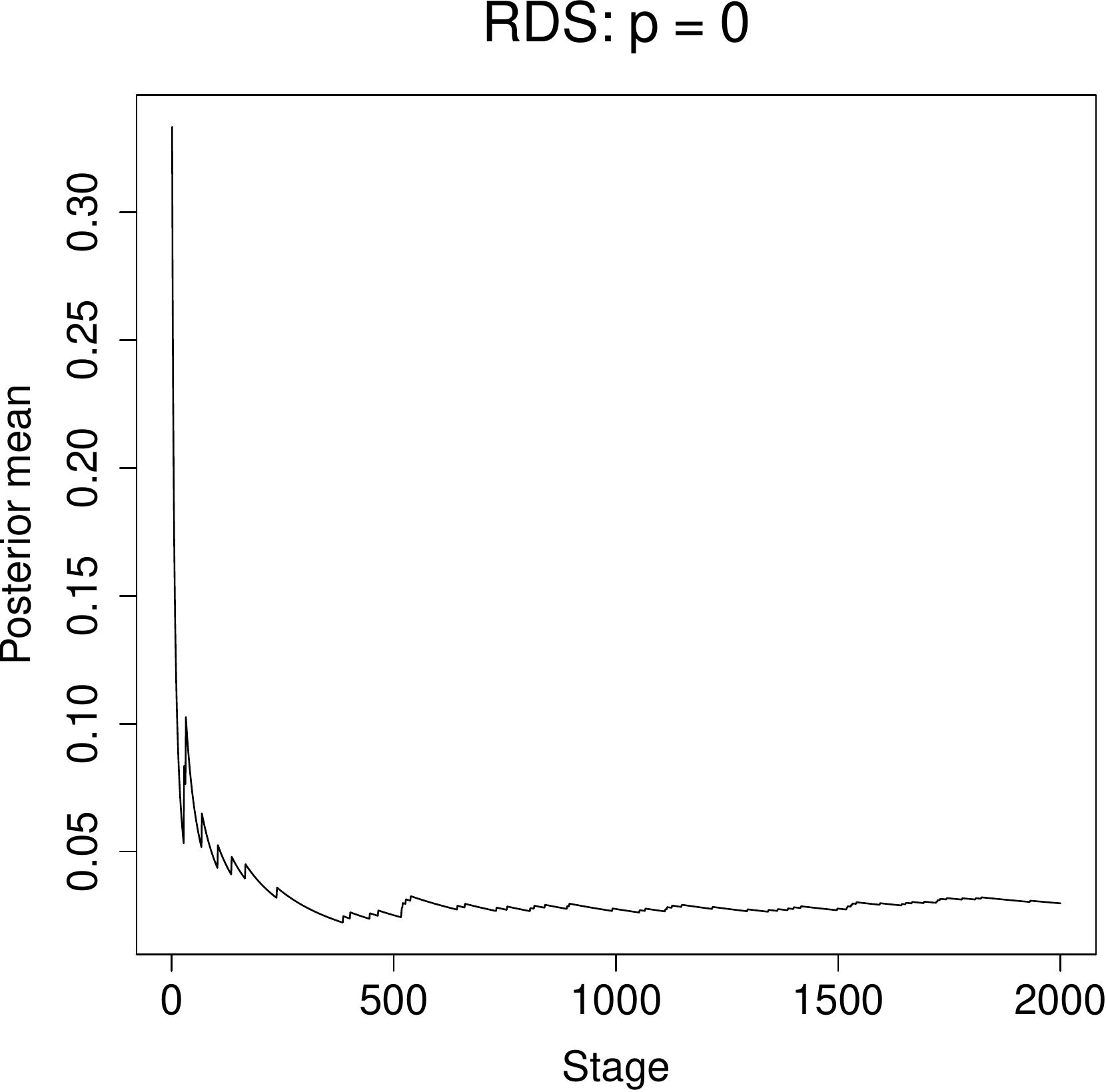}}
\hspace{2mm}
\subfigure [Divergence.]{ \label{fig:ds_para2}
\includegraphics[width=6cm,height=5cm]{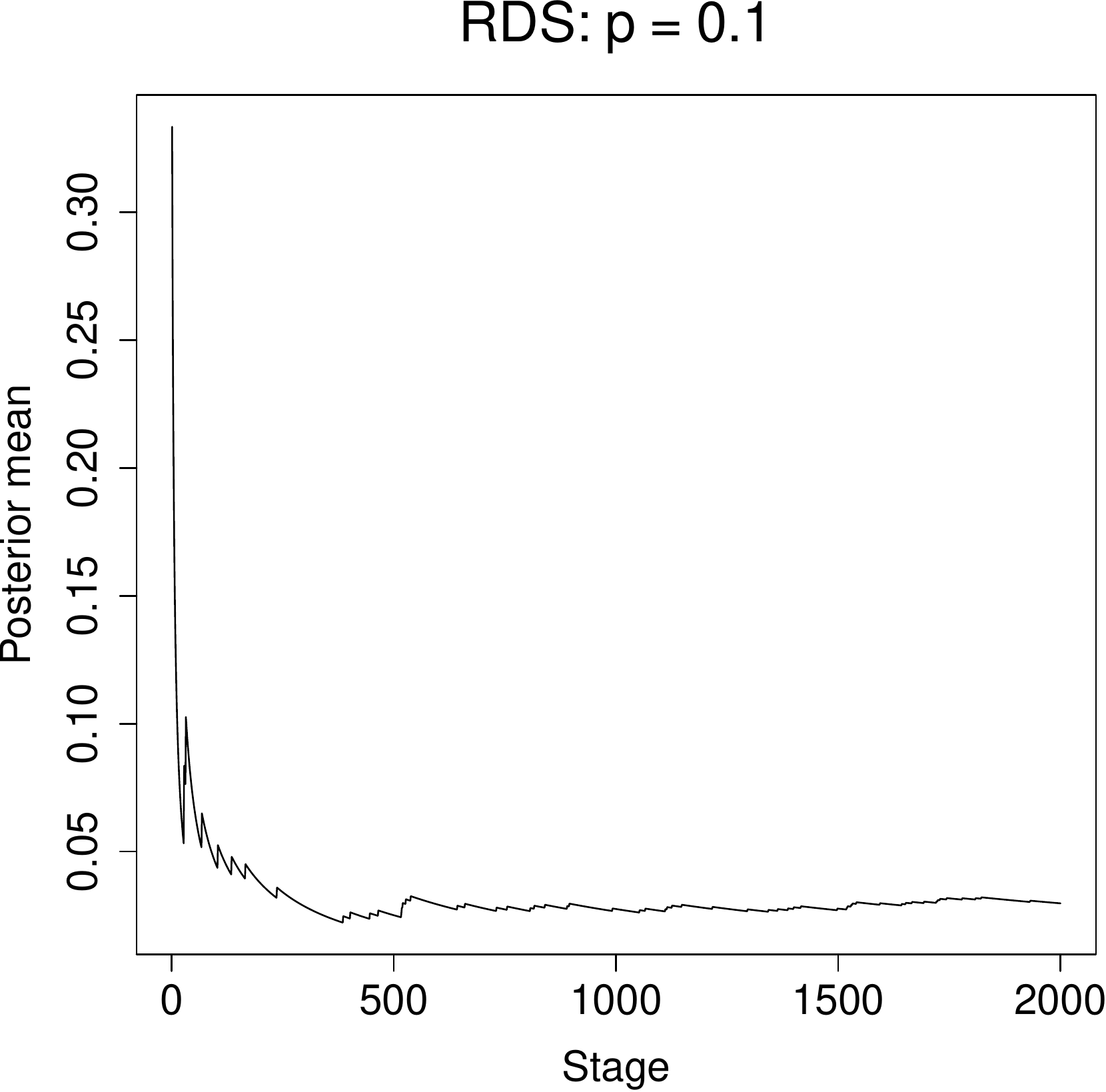}}\\
\subfigure [Divergence.]{ \label{fig:ds_para3}
\includegraphics[width=6cm,height=5cm]{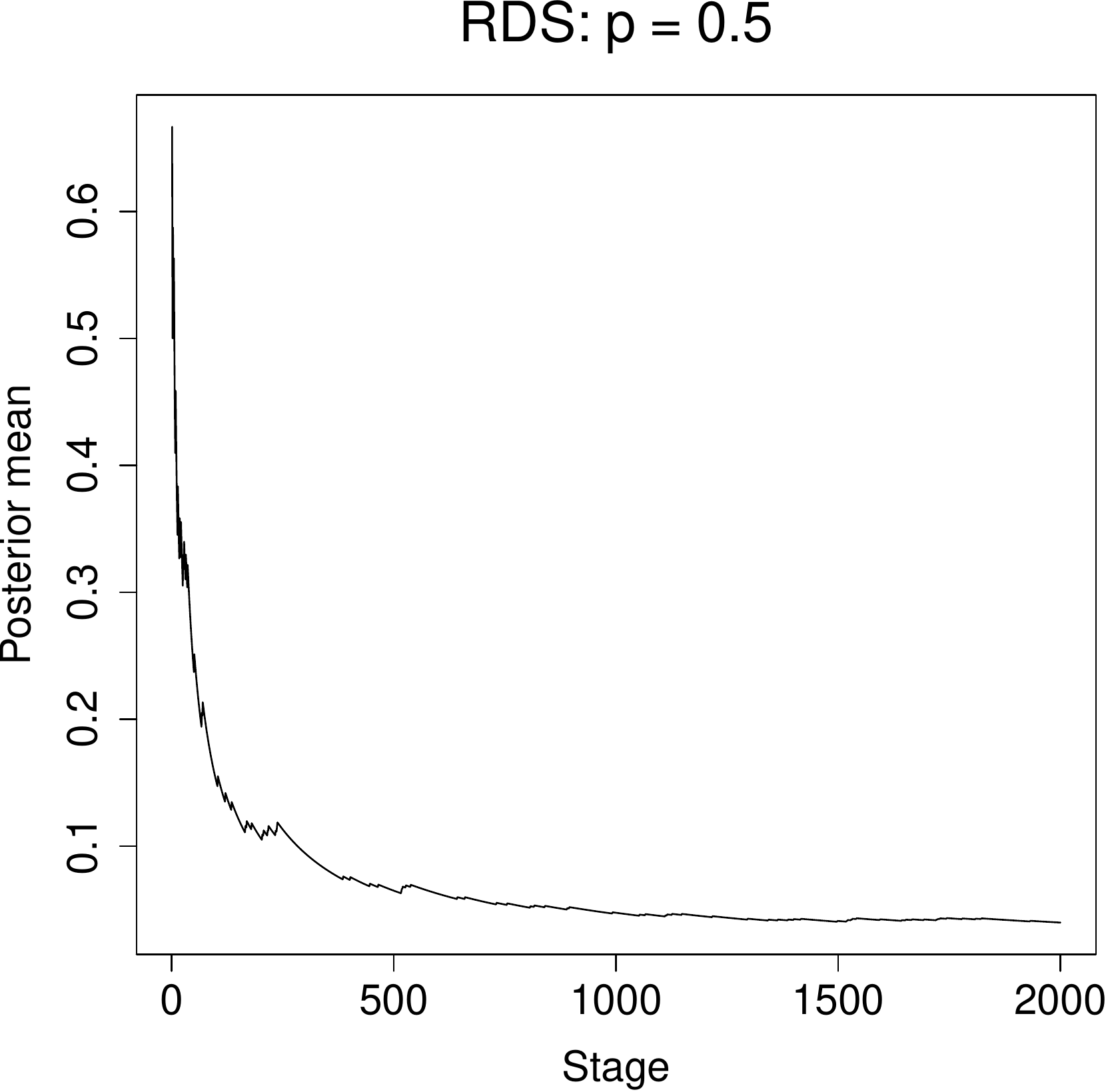}}
\hspace{2mm}
\subfigure [Divergence.]{ \label{fig:ds_para4}
\includegraphics[width=6cm,height=5cm]{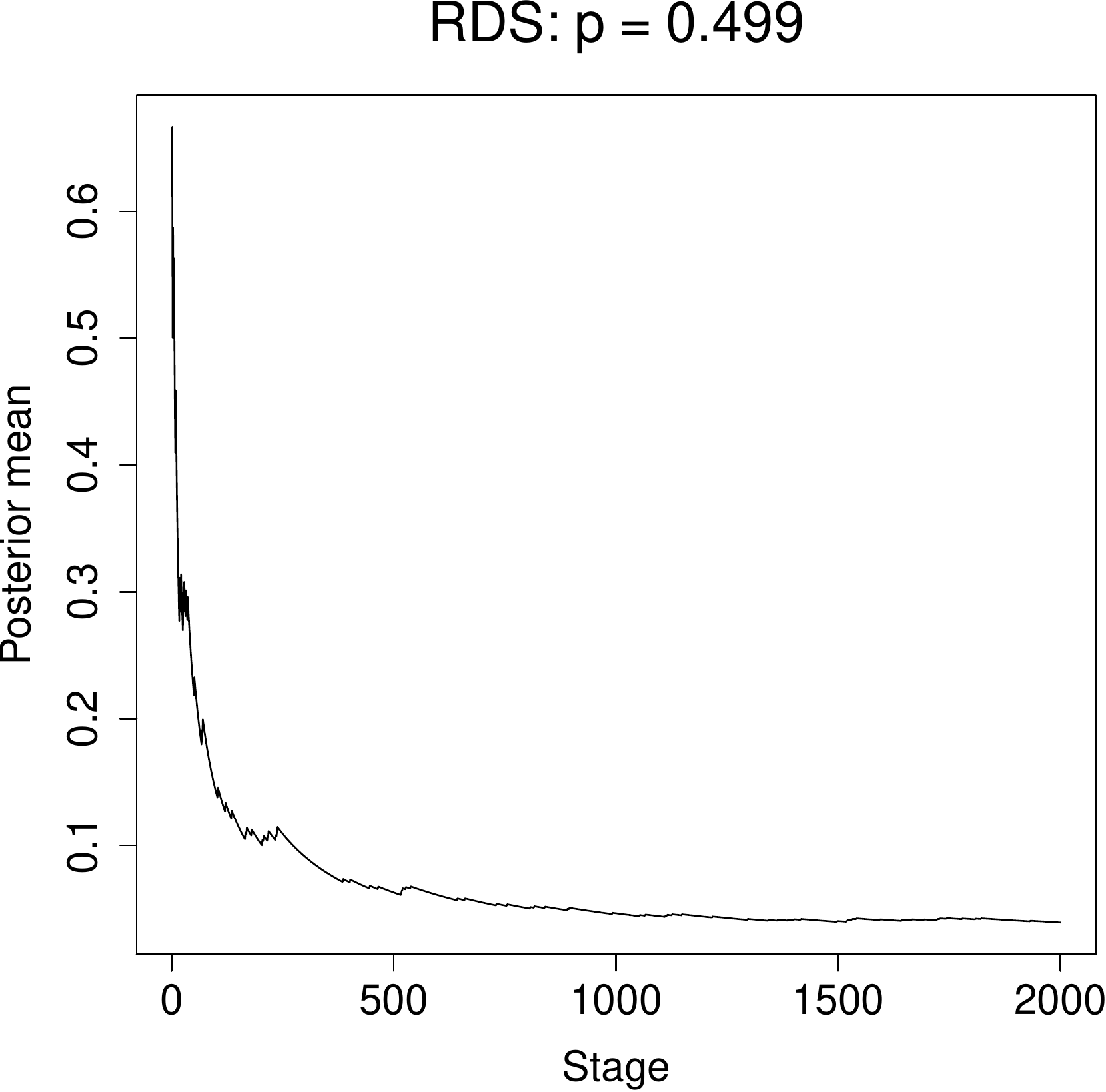}}\\
\subfigure [Convergence.]{ \label{fig:ds_para5}
\includegraphics[width=6cm,height=5cm]{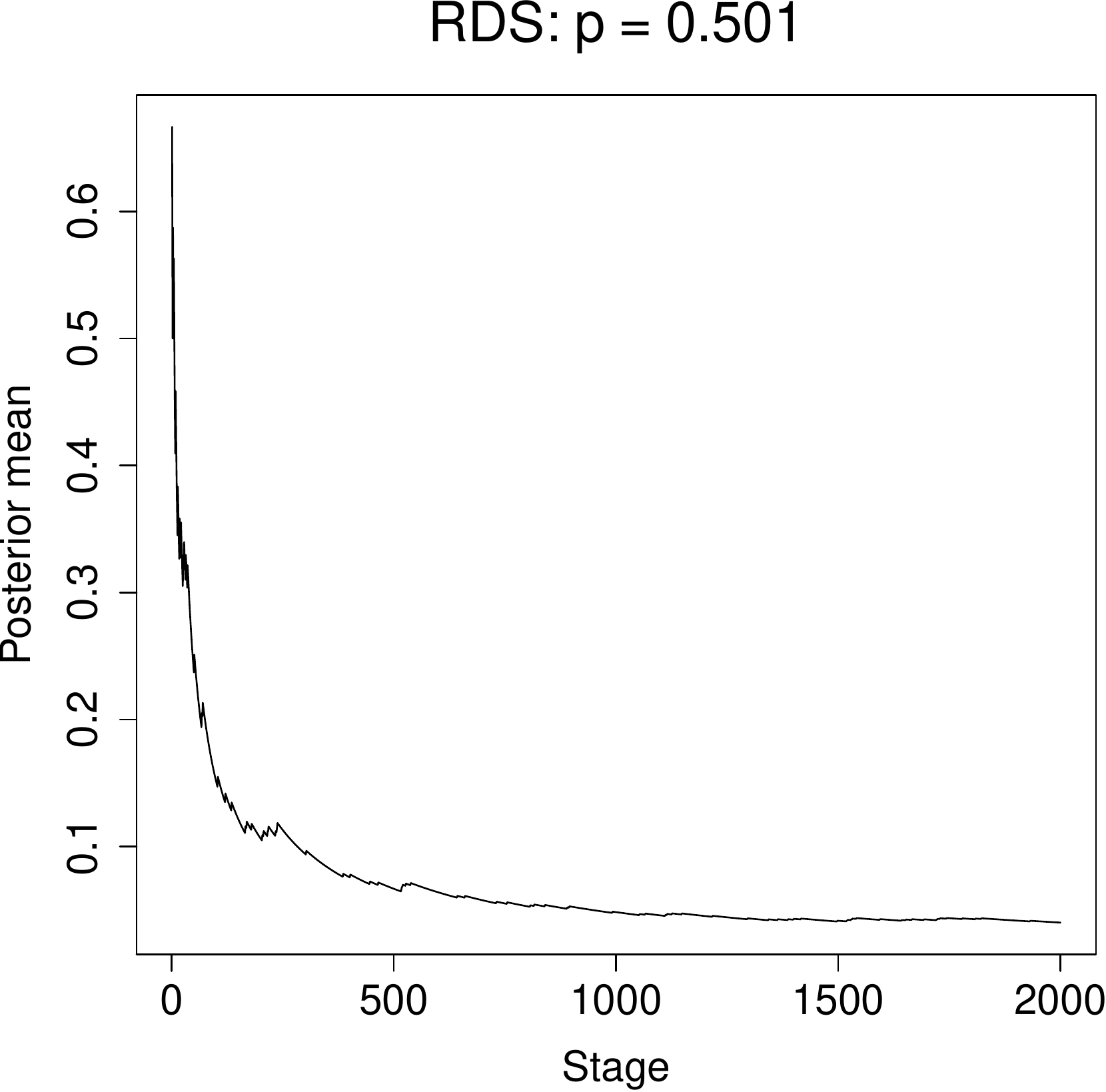}}
\hspace{2mm}
\subfigure [Convergence.]{ \label{fig:ds_para6}
\includegraphics[width=6cm,height=5cm]{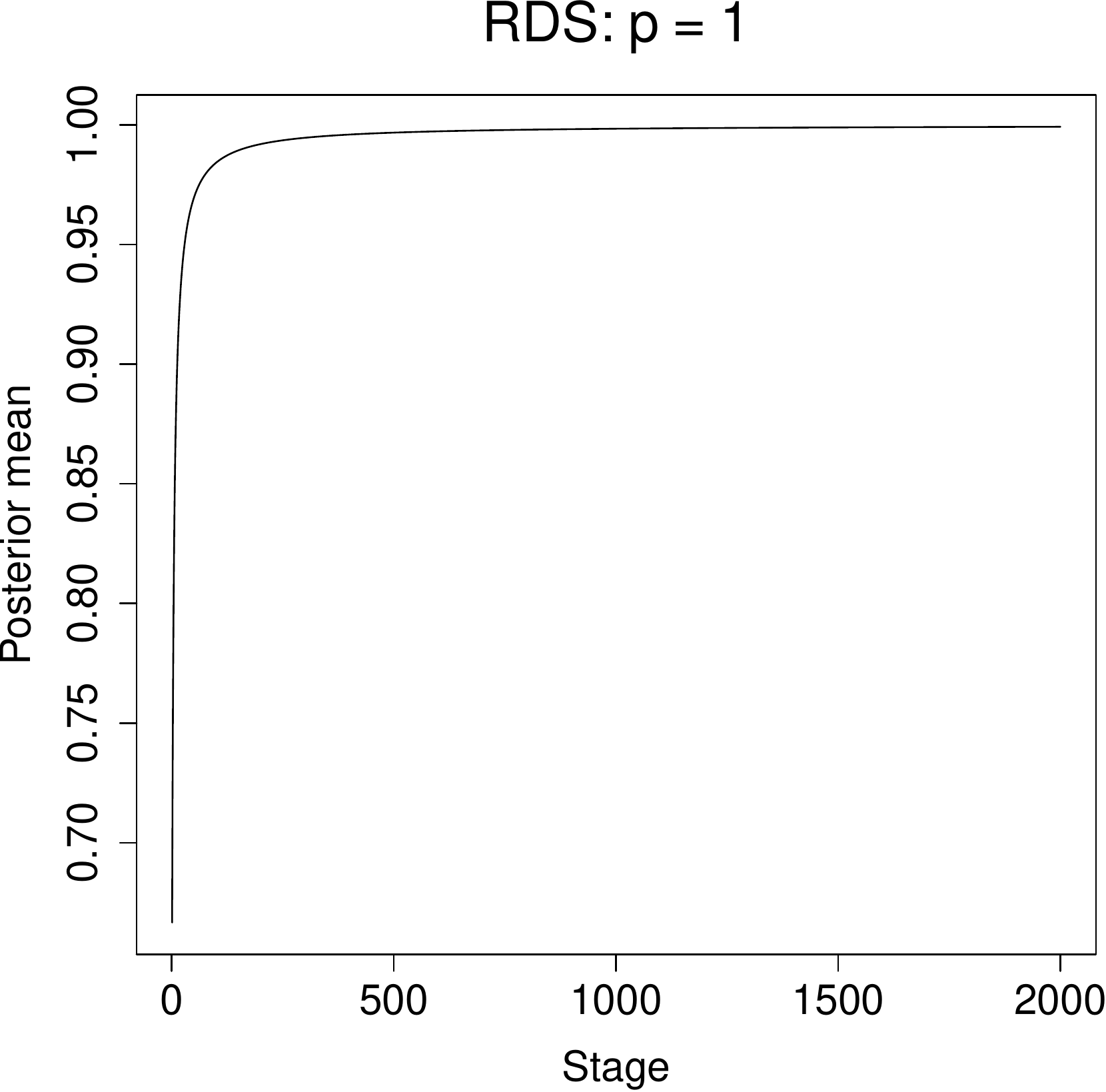}}\\
\subfigure [Divergence.]{ \label{fig:ds_para7}
\includegraphics[width=6cm,height=5cm]{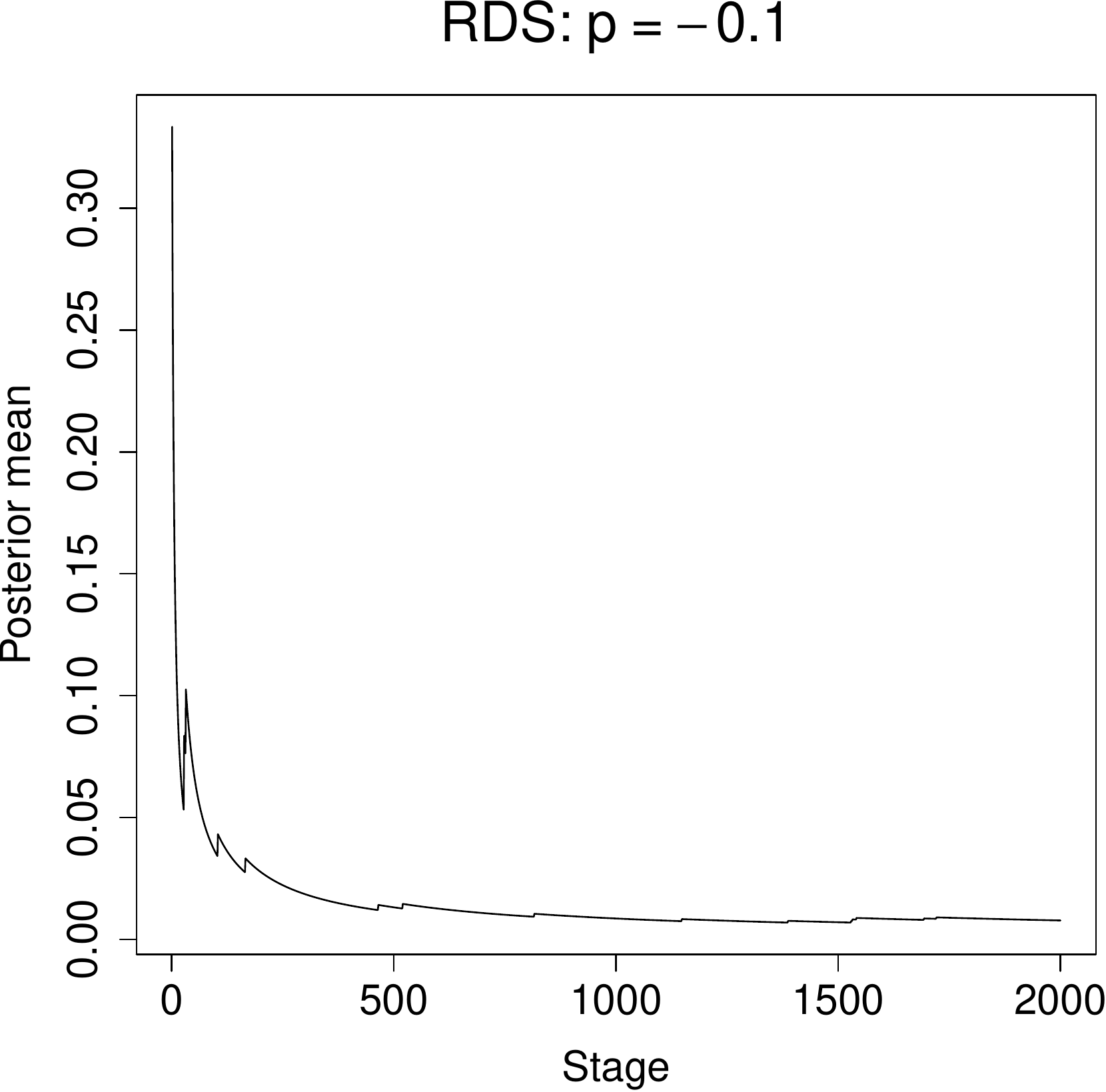}}
\caption{Example 6: Convergence and divergence for RDS.}
\label{fig:ds_para}
\end{figure}

\section{Nonparametric bounds for the partial sums and simulation experiments}
\label{sec:nonpara}

The parametric upper bounds for the partial sums are quite restrictive in the sense of requiring non-negative supports. The general upper bound 
(\ref{eq:S2}) is not theoretically sound and although it works well for exponential series and state-space series driven by exponential distributions (results
not shown for the sake of brevity), we have shown that its performance for series driven by normal
distributions is far from satisfactory, as very large number of iterations, with very large number of summands for the partial sums are required. Even then,
the independent and dependent normal setups do not exhibit convergence of our Bayesian procedure adequately close to $1$ and $0$ for convergent and divergent random series,
in many cases. Also in the RDS setup, incorrect results are obtained in a lot of cases with (\ref{eq:S2}). 
Thus, the general bound is not expected to work well for distributions supported on the real line.
Moreover, the bound construction methods require specific knowledge of the form of the underlying distribution $f_{\theta_i}$ of the $i$-th element 
$X_i$ of the random series. In reality, such information can not be expected to be available. 

Hence, effective bounds, which are independent of supports of the summands and the underlying distributional assumptions, are desirable. 
To this end, we propose the nonparametric
bounds introduced by \ctn{Roy20} in the context of Bayesian characterization of stochastic process properties. Although the context is different, the key Bayesian idea
employed by \ctn{Roy20} is the same as ours. Since their bounds turned out to be very effective in most of their varied examples, we expect ours to be no different.

Specifically, we set
\begin{equation}
c_j=\hat C_j/\log(j+1),
\label{eq:ar1_bound3}
\end{equation}	
where $\hat C_1$ is a chosen constant, and 
for $j>1$, $\hat C_j=\hat C_{j-1}+0.05$ if $y_{j-1}=1$ and $\hat C_j=\hat C_{j-1}-0.05$ if $y_{j-1}=0$. 

Thus, we favour convergence at the next, $(j+1)$-th stage, if at the current stage convergence is supported ($y_j=1$), and favour divergence otherwise. 
The $\log(j+1)$ scale ensures that the rate of convergence of $c_j$ to zero as $j\rightarrow\infty$, is neither too fast, nor too slow.

The choice of the initial value $\hat C_1$ is an important issue and if chosen without utmost care, can yield wrong results regarding series convergence properties.
The choice is also expected to to be problem specific in general. However, in our examples involving normal and exponential based models, 
we find $\hat C_1=0.71$ and $0.725$, respectively, to be quite appropriate.
This is somewhat in keeping with \ctn{Roy20} who found $\hat C_1=1$ or values close to $1$ to be adequate in most cases, in spite of their wide variety of examples.
In the case of RDS we exploit the corresponding deterministic Dirichlet series to obtain an appropriate value of $\hat C_1$.

\subsection{Simulation experiments with the nonparametric bound form}
\label{subsec:simexp_nonpara}

We now conduct simulation experiments with this new, nonparametric bound form (\ref{eq:ar1_bound3}) applied to the setups considered in Section \ref{sec:simstudy1}. 
For all the cases, we now consider $n_j=1000$ for $j=1,\ldots,K$, with $K=2000$. Thus, even for the series driven by normal and dependent normal distributions 
we now consider situations where the number of summands in each partial sum,
as well as the number of stages (iterations) for our Bayesian procedure are significantly smaller compared to those in Sections \ref{subsec:bayesian_normal}
and \ref{subsec:bayesian_normal_dependent}. Needless to mention,
the time taken for the implementations of the Bayesian procedure with the nonparametric bound are less than a second. 
As we shall see, in almost all the cases, the bound form (\ref{eq:ar1_bound3}) yields the correct answer, even for the normal driven series,
in spite of many times smaller sample size as used in Sections \ref{subsec:bayesian_normal} and \ref{subsec:bayesian_normal_dependent}. 
Importantly, in all the cases, the Bayesian method gets sufficiently close to $1$ and $0$
for convergent and divergent series, respectively. Recall that this was not the case for independent and dependent normal setups, even with extremely large sample sizes,
and incorrect results were obtained for the RDS.
Thus, the bound (\ref{eq:ar1_bound3}), in spite of having a nonparametric form, turns out to be far more effective and efficient than the previous general 
parametric bound (\ref{eq:S2}). However, for the hierarchical exponential setup and the state-space hierarchical exponential setup, 
the nonparametric bound performs slightly worse in a very subtle situation compared to the mathematically valid parametric bound.
On the other hand, the nonparametric bound slightly outperforms the mathematically sound parametric counterpart in a subtle situation of 
the state-space non-hierarchical exponential setup.
Thus, the nonparametric bound seems to be very much comparable with the valid parametric bound when the latter is available, and emphatically 
outperforms the general parametric bound (\ref{eq:S2}).

\subsubsection{Example 1 revisited: Hierarchical exponential distribution}
\label{subsubsec:bayesian_exp}

As in Section \ref{sec:simstudy1}, we first consider the setup $X_i\sim\mathcal E(\theta_i)$ and $\theta_i\sim\mathcal E(\psi_i)$; $i\geq 1$.
Here experimentation reveals that $\hat C_1=0.725$ is an appropriate choice that can detect most convergent and divergent series driven by 
exponential distributions of the above form.

Figure \ref{fig:example_exp2} displays the results of our Bayesian analyses of different exponential series of the above form. 
Not only does the Bayesian procedure with the nonparametric bound captures the correct result even for such small sample sizes, it does so quite convincingly, as
the method gets adequately close to $1$ and $0$ for convergent and divergent series, respectively. However, it is important to mention that for $\psi_i=i^{-p}$,
for $p\in(0.95,1]$, our method with the nonparametric bound failed to yield correct results. Thus, a little subtlety seems to have been sacrificed due to the small 
sample size. Indeed, increasing $n_j$ led to increasing shrinkage of the offending interval $(0.95,1]$ towards $1$.

\begin{figure}
\centering
\subfigure [Divergence.]{ \label{fig:exp1_2}
\includegraphics[width=6cm,height=5cm]{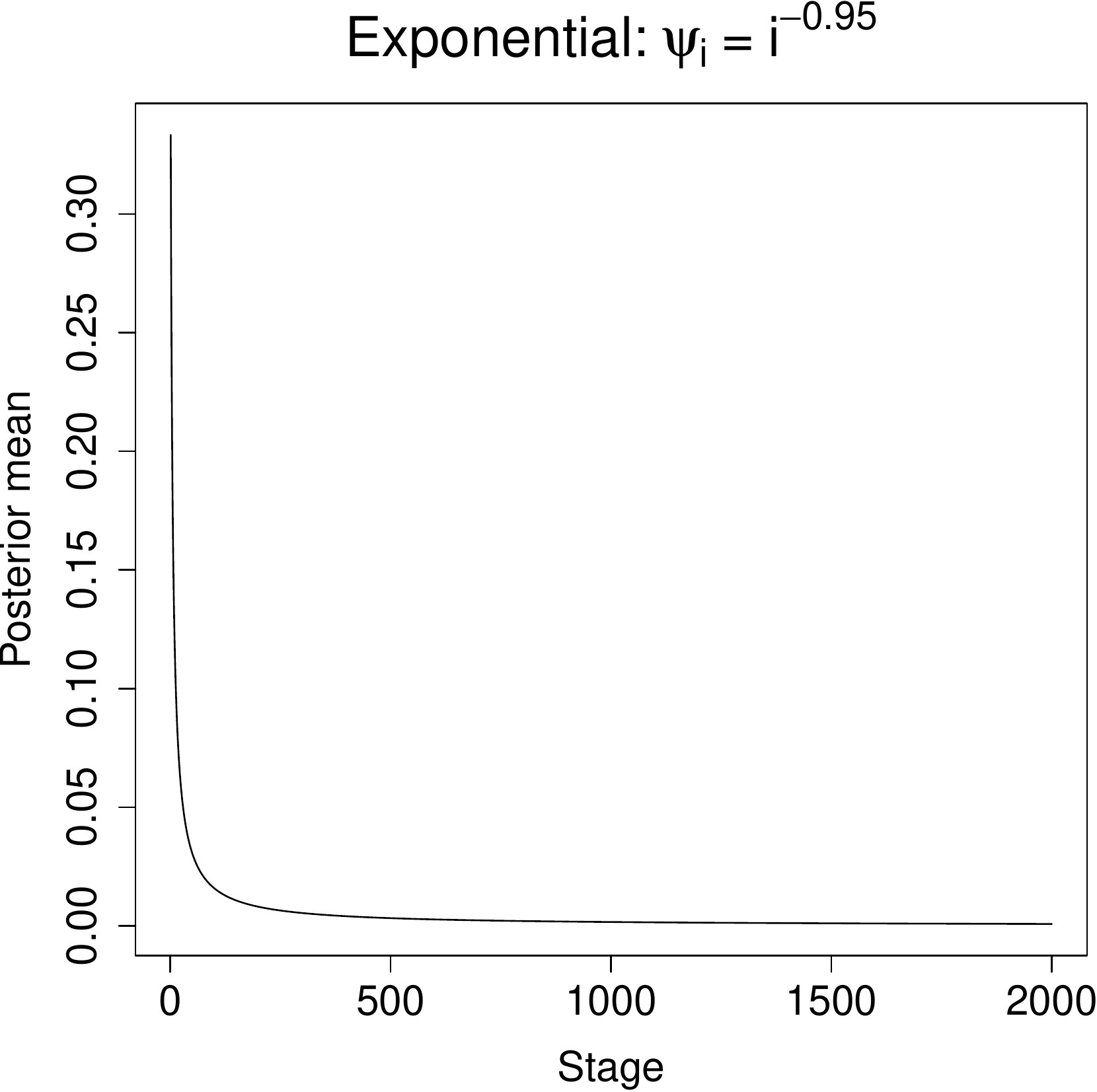}}
\hspace{2mm}
\subfigure [Convergence.]{ \label{fig:exp2_2}
\includegraphics[width=6cm,height=5cm]{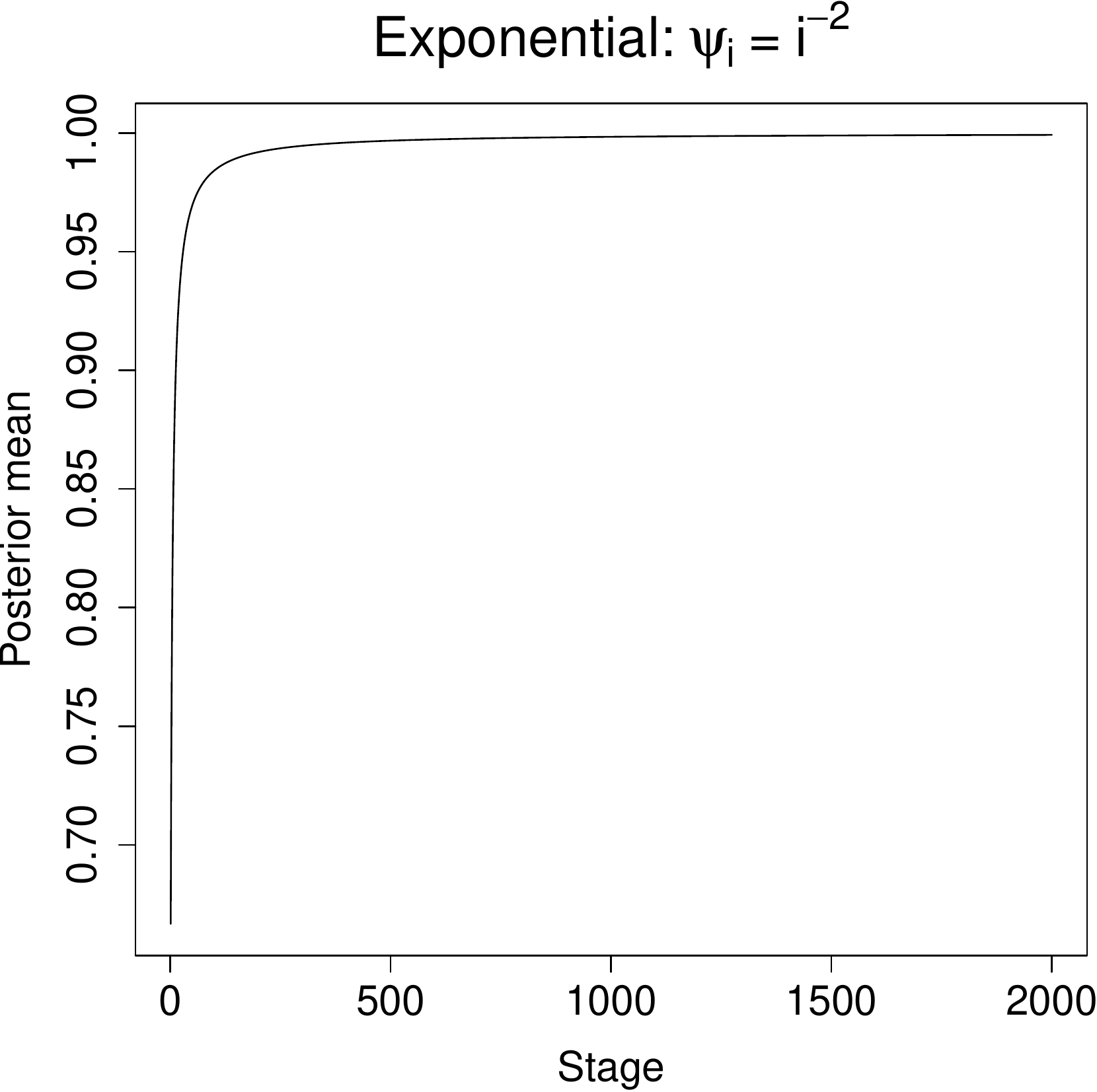}}\\
\subfigure [Convergence.]{ \label{fig:exp3_2}
\includegraphics[width=6cm,height=5cm]{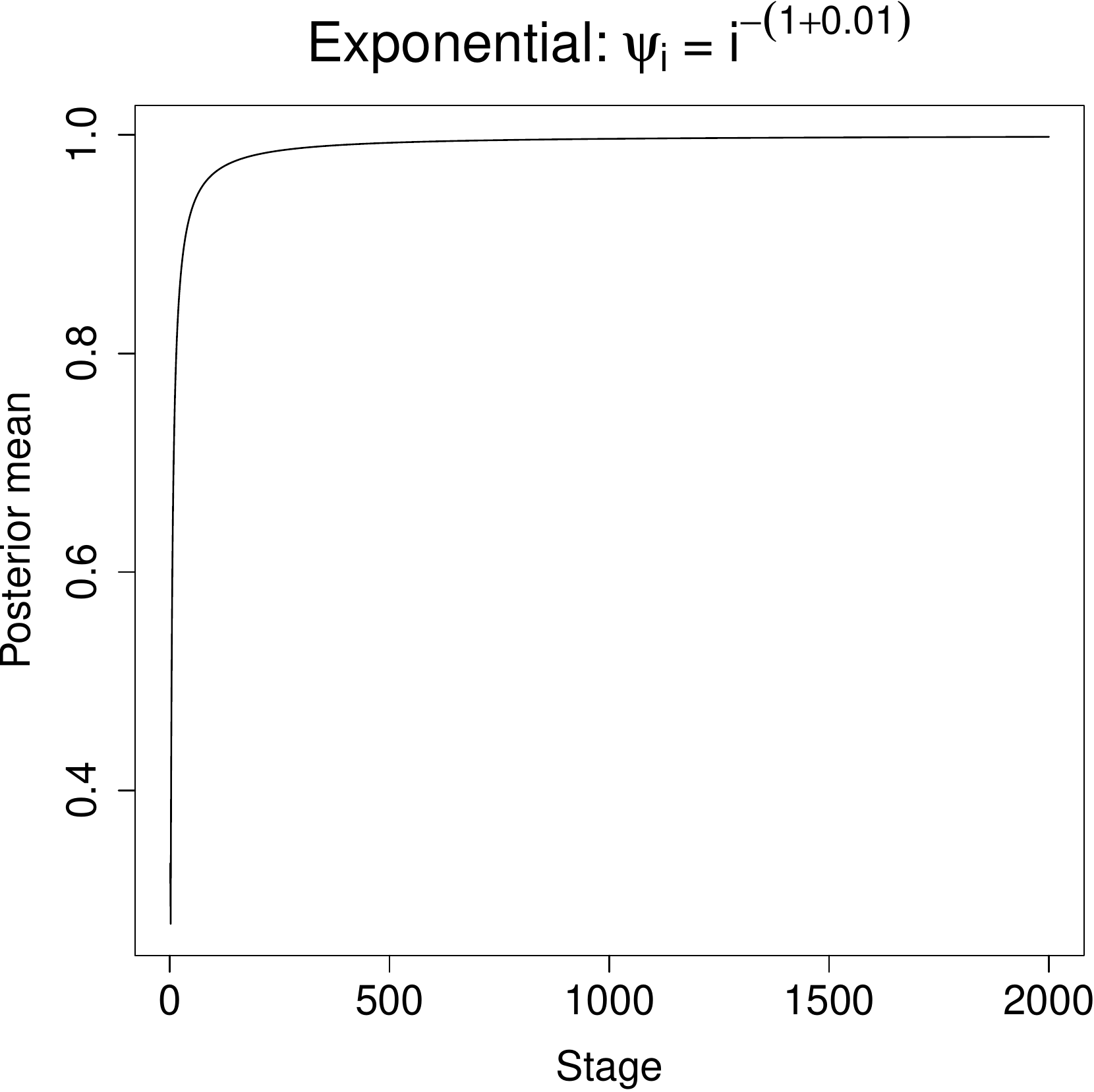}}
\hspace{2mm}
\subfigure [Convergence.]{ \label{fig:exp4_2}
\includegraphics[width=6cm,height=5cm]{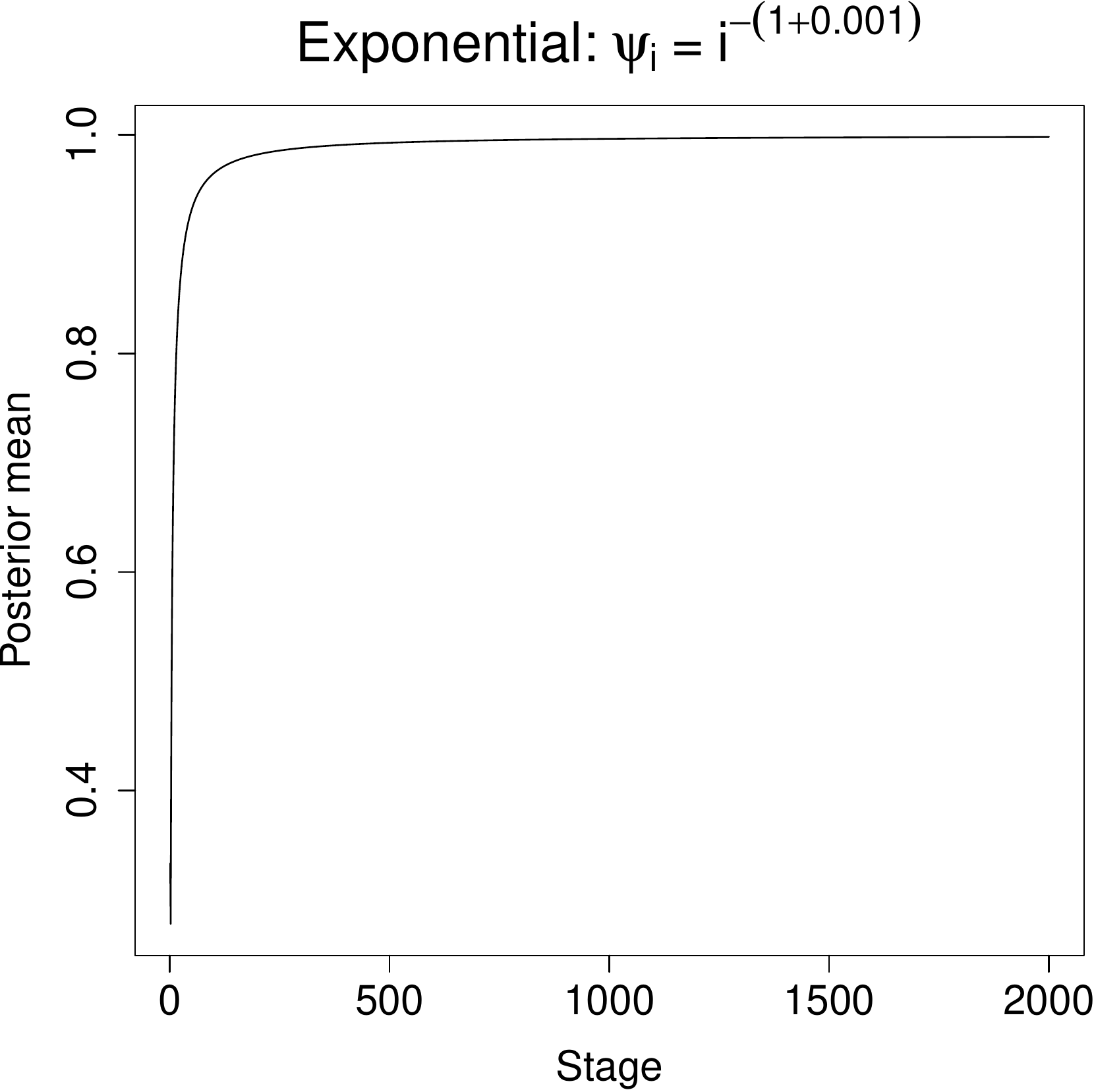}}\\
\subfigure [Convergence.]{ \label{fig:exp5_2}
\includegraphics[width=6cm,height=5cm]{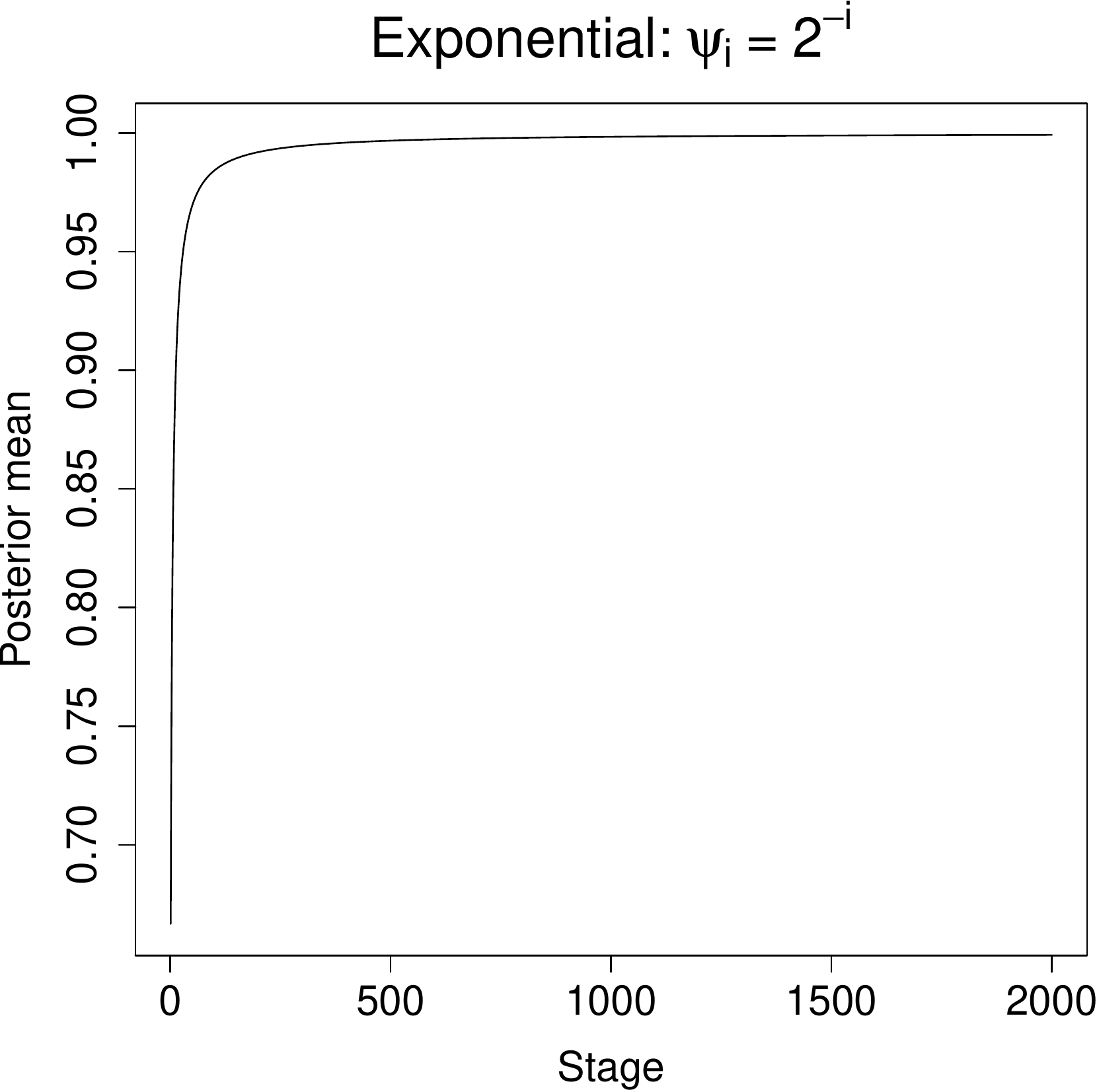}}
\hspace{2mm}
\subfigure [Convergence.]{ \label{fig:exp6_2}
\includegraphics[width=6cm,height=5cm]{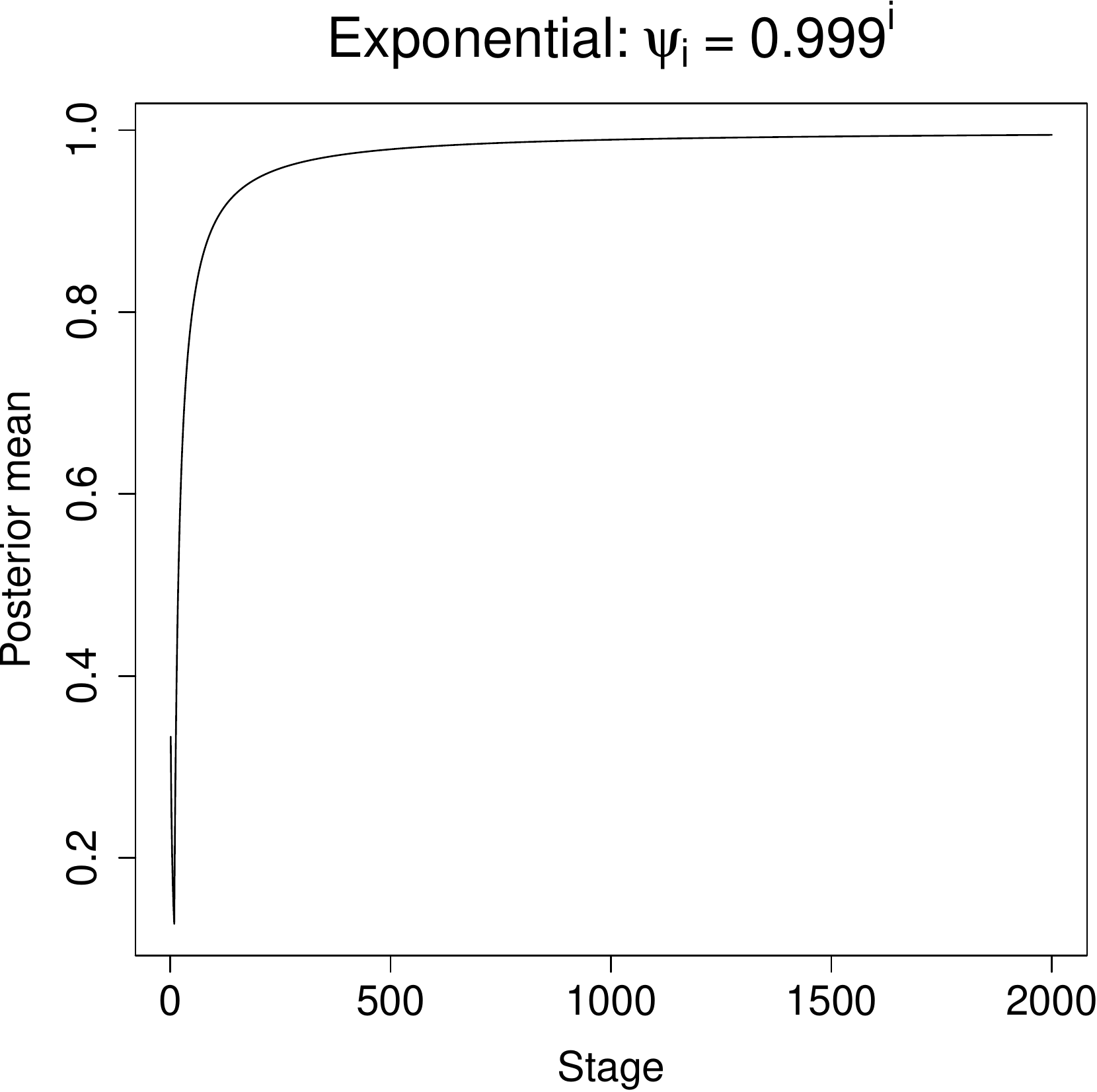}}\\
\subfigure [Divergence.]{ \label{fig:exp7_2}
\includegraphics[width=6cm,height=5cm]{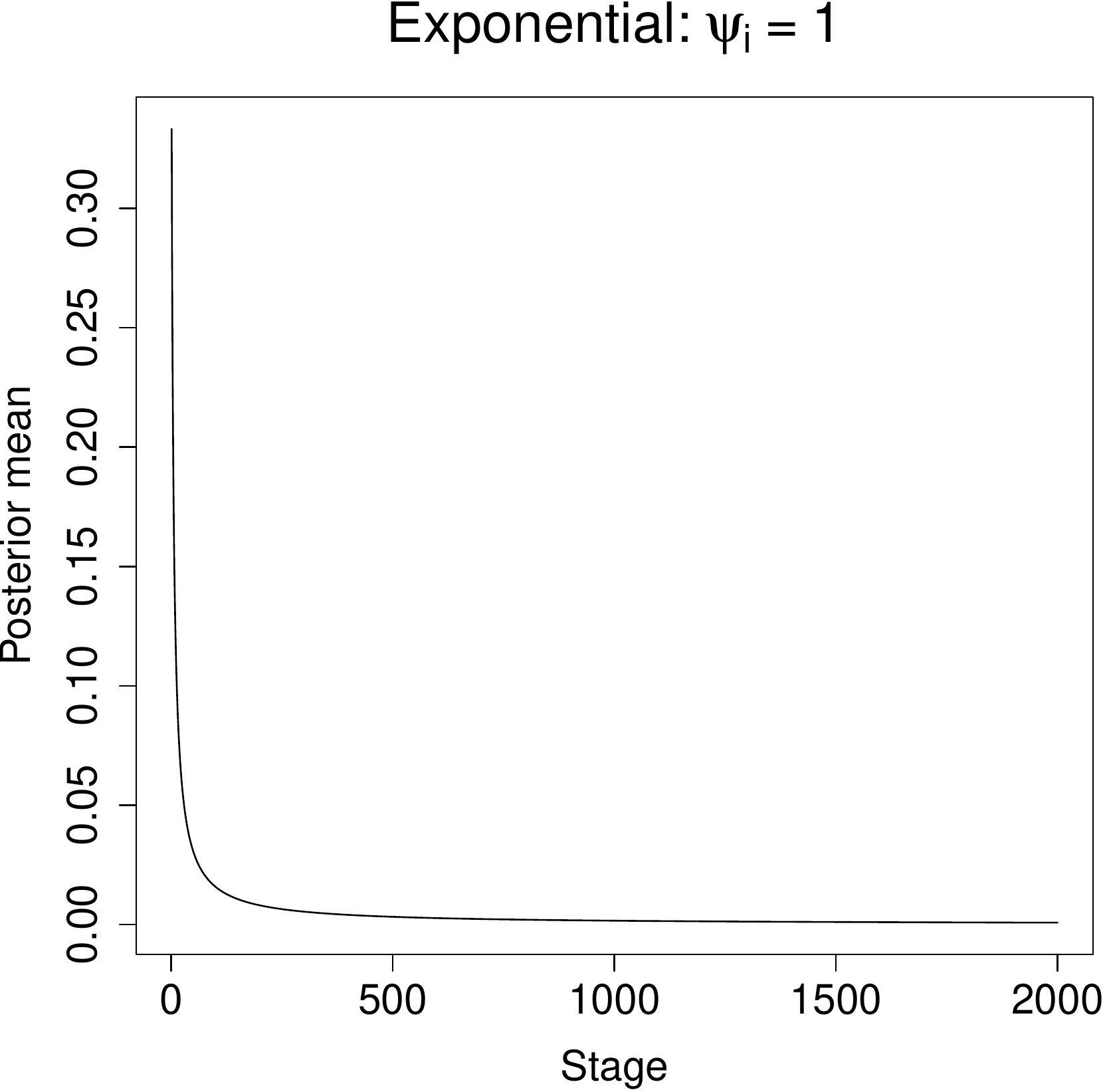}}
\caption{Example 1 revisited: Convergence and divergence for exponential series with nonparametric bound.}
\label{fig:example_exp2}
\end{figure}

\subsubsection{Example 2 revisited: Hierarchical normal distribution}
\label{subsubsec:bayesian_normal}
As in Section \ref{subsec:bayesian_normal}, we now let $X_i\sim N\left(\mu_i,\sigma^2_i\right)$, $\mu_i\sim N\left(0,\phi^2_i\right)$
and $\sigma^2_i\sim\mathcal E(\vartheta_i)$; $i\geq 1$. 
Here $\hat C_1=0.71$ turned out to be appropriate. Notice its close similarity with $\hat C_1=0.725$ for the exponential bound.

Figure \ref{fig:example_normal2} shows our results in this setup. In all the cases, correct results are convincingly obtained, even with such a small sample size.
The results are convincing in the sense that the underlying Bayesian procedure gets sufficiently close to $1$ and $0$ for all the convergent and divergent series,
respectively. Thus, compared to Figure \ref{fig:example_normal} corresponding to the parametric bound, we have a huge gain in efficiency and effectiveness.
However, it must be mentioned that for such small sample size, our method failed in the cases where $\phi_i=\vartheta_i=i^{-(1+a)}$, for $a\in (0.0,0.04)$. 
\begin{figure}
\centering
\subfigure [Divergence.]{ \label{fig:normal1_2}
\includegraphics[width=6cm,height=5cm]{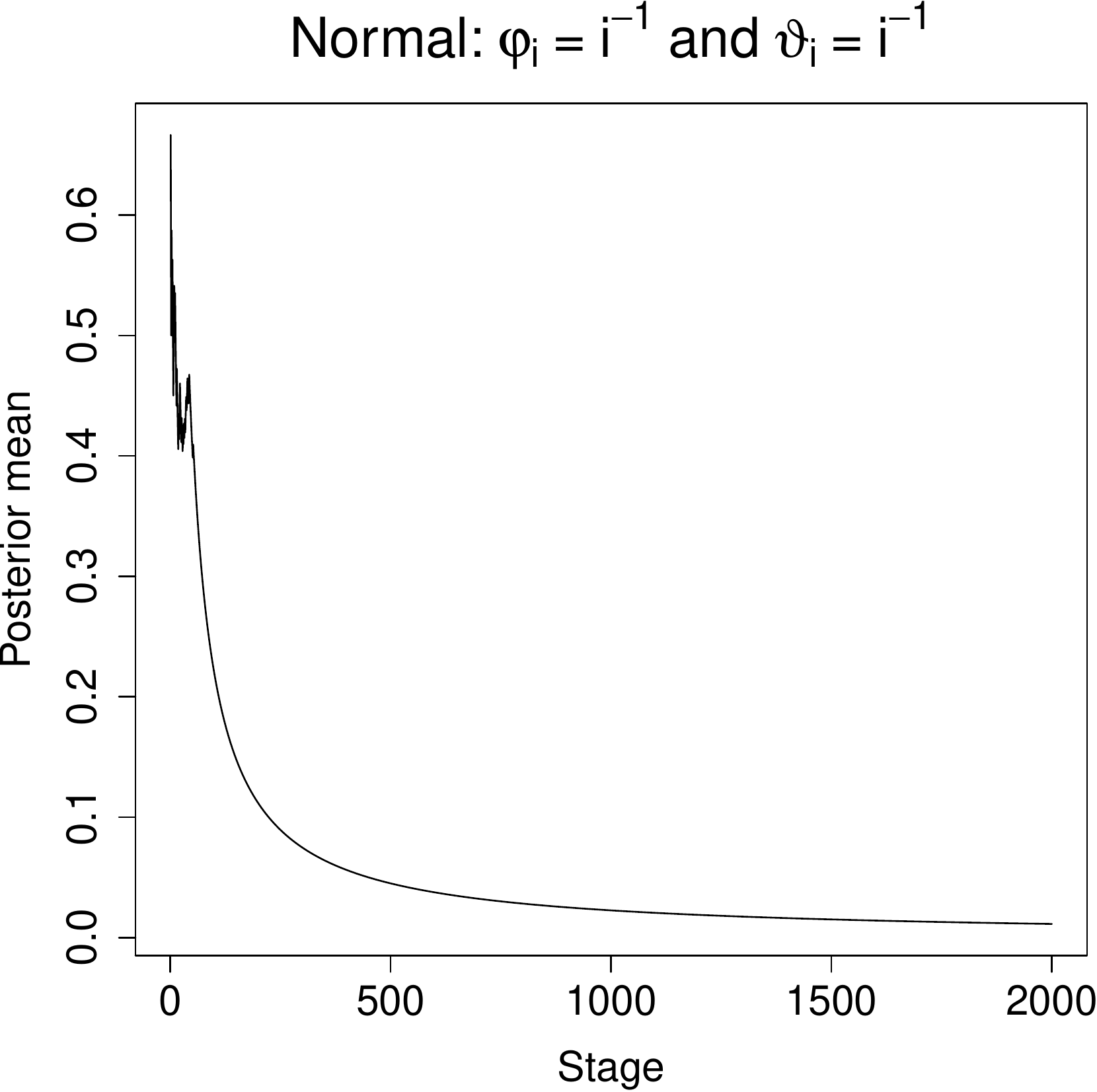}}
\hspace{2mm}
\subfigure [Convergence.]{ \label{fig:normal2_2}
\includegraphics[width=6cm,height=5cm]{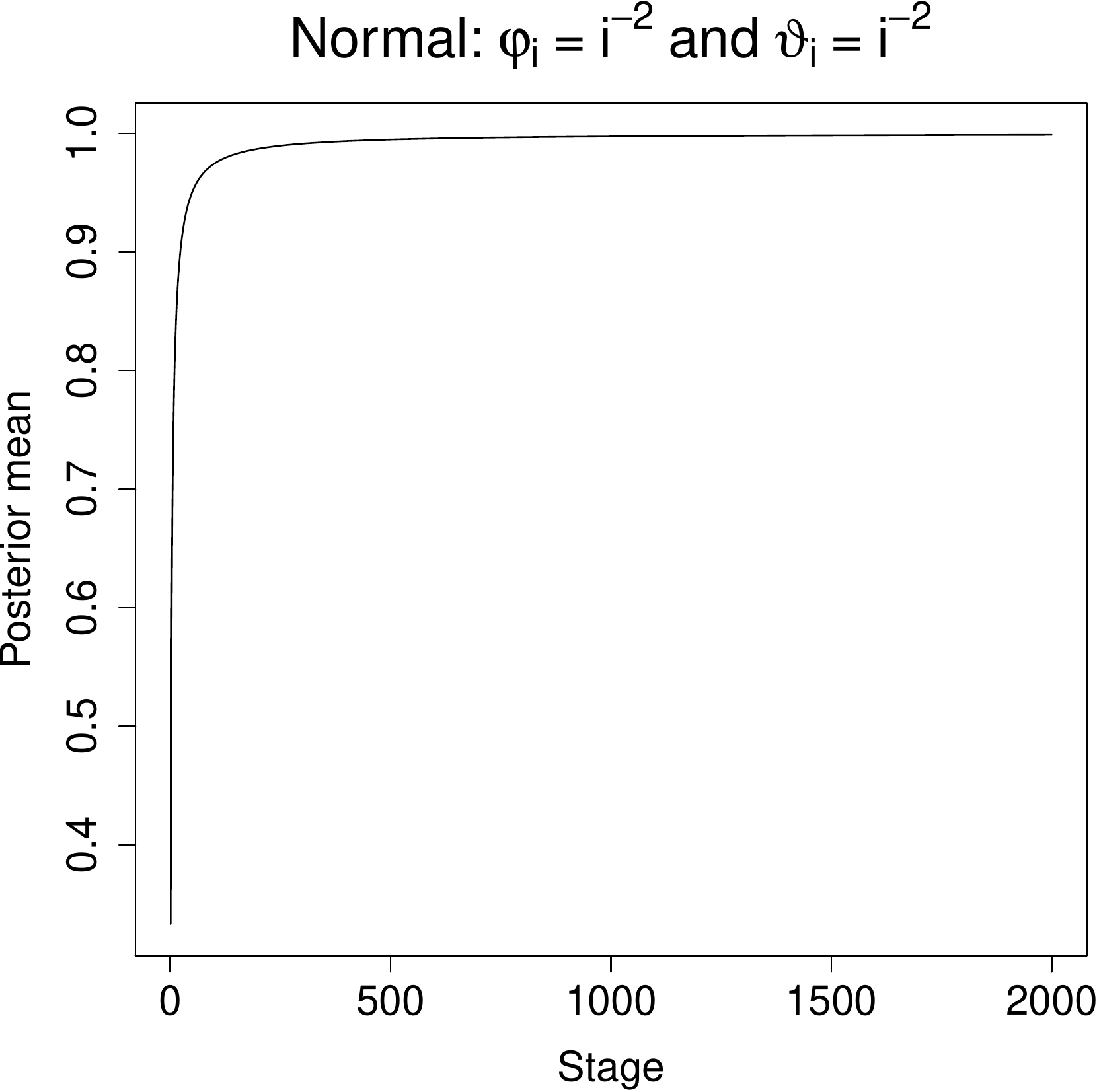}}\\
\subfigure [Convergence.]{ \label{fig:normal3_2}
\includegraphics[width=6cm,height=5cm]{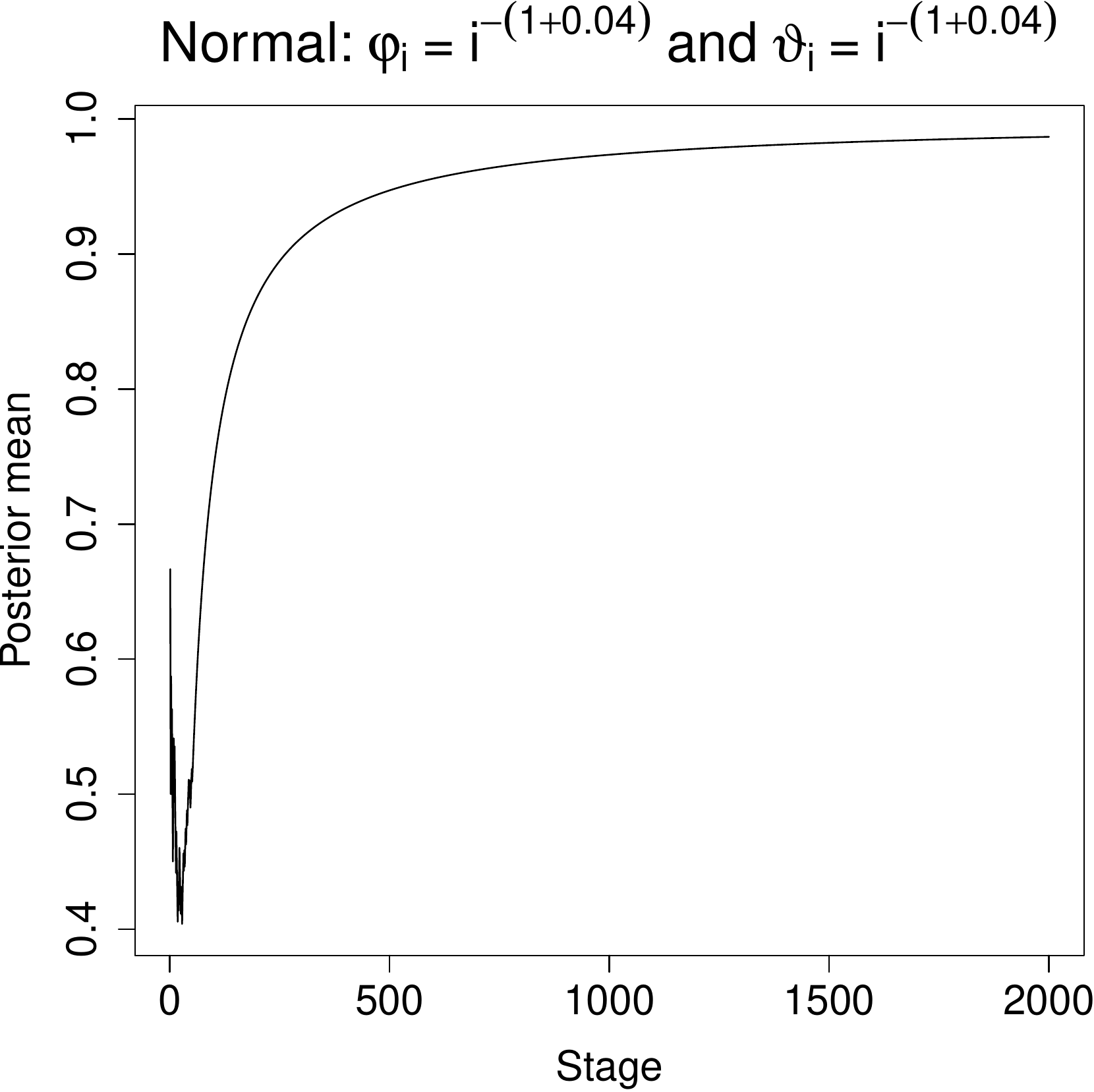}}
\hspace{2mm}
\subfigure [Divergence.]{ \label{fig:normal4_2}
\includegraphics[width=6cm,height=5cm]{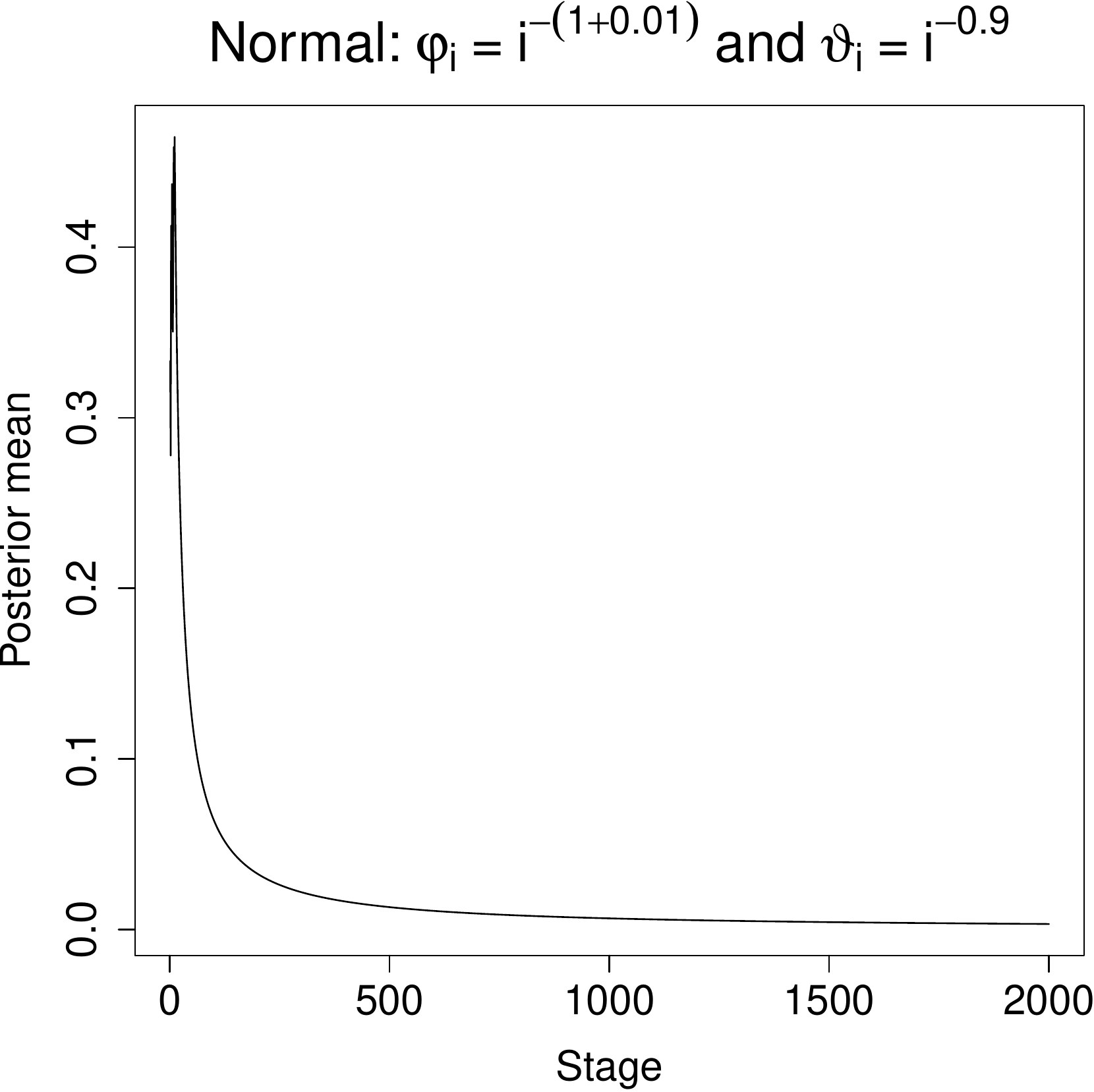}}\\
\subfigure [Divergence.]{ \label{fig:normal5_2}
\includegraphics[width=6cm,height=5cm]{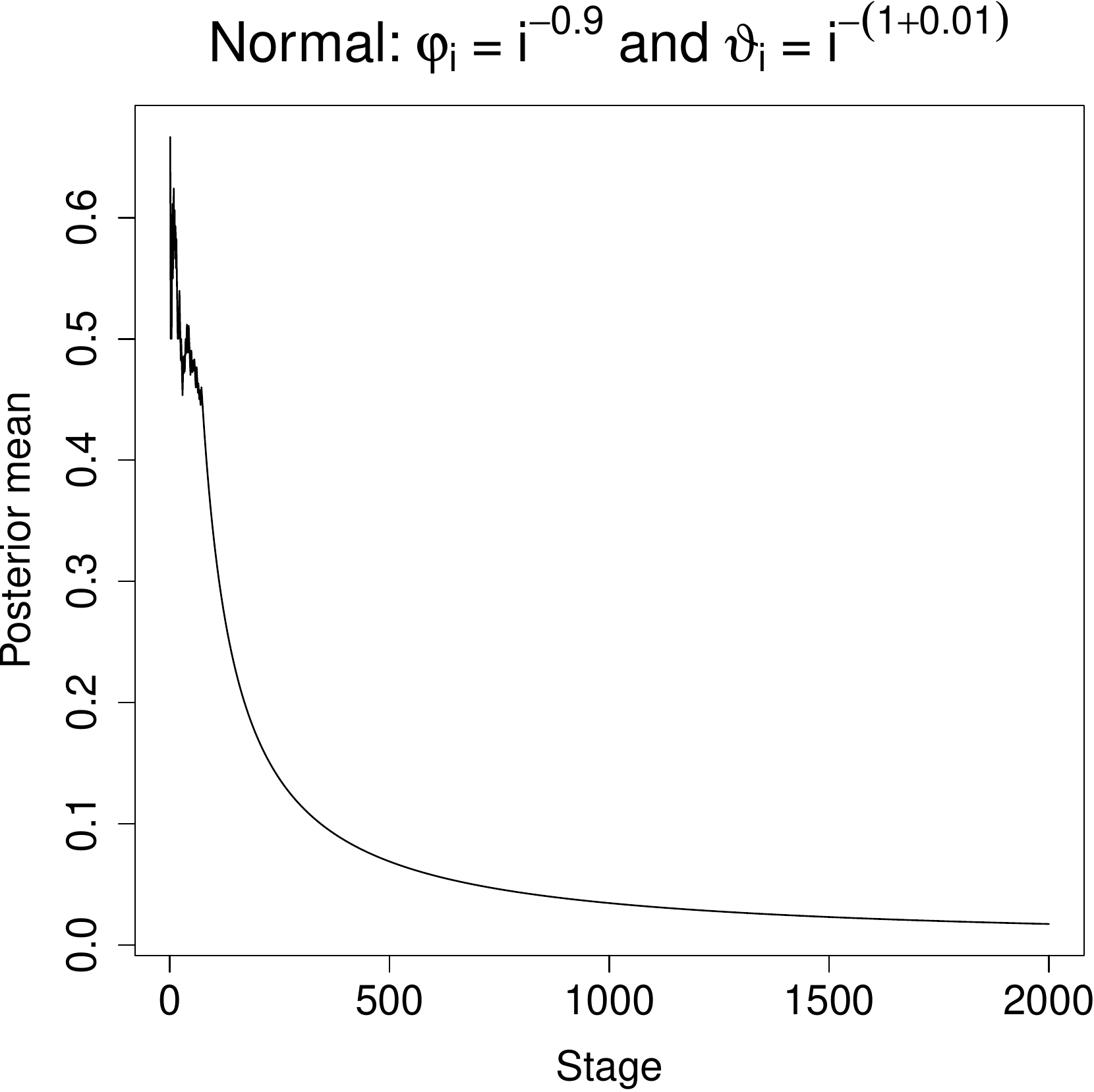}}
\hspace{2mm}
\subfigure [Divergence.]{ \label{fig:normal6_2}
\includegraphics[width=6cm,height=5cm]{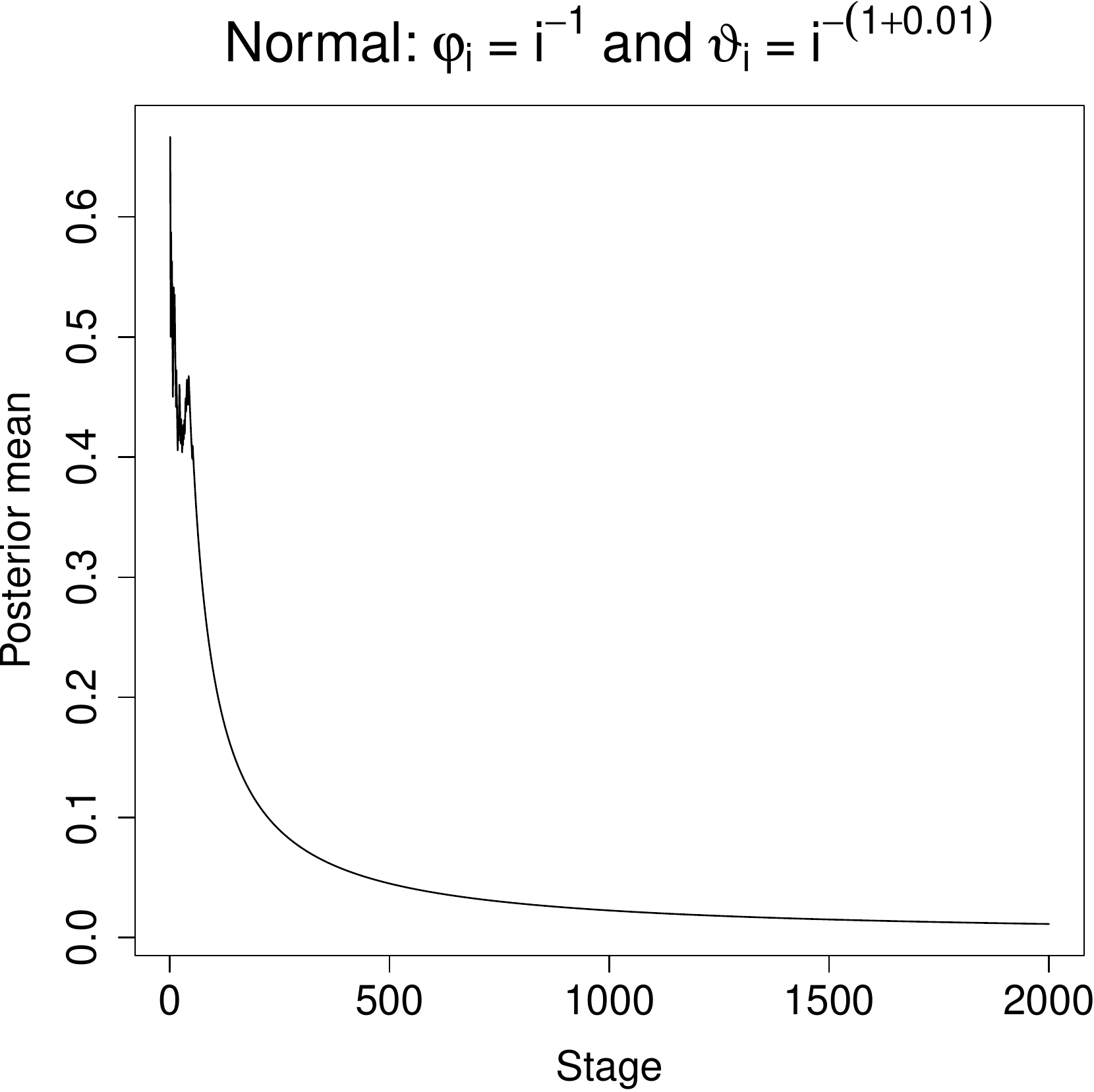}}\\
\subfigure [Divergence.]{ \label{fig:normal7_2}
\includegraphics[width=6cm,height=5cm]{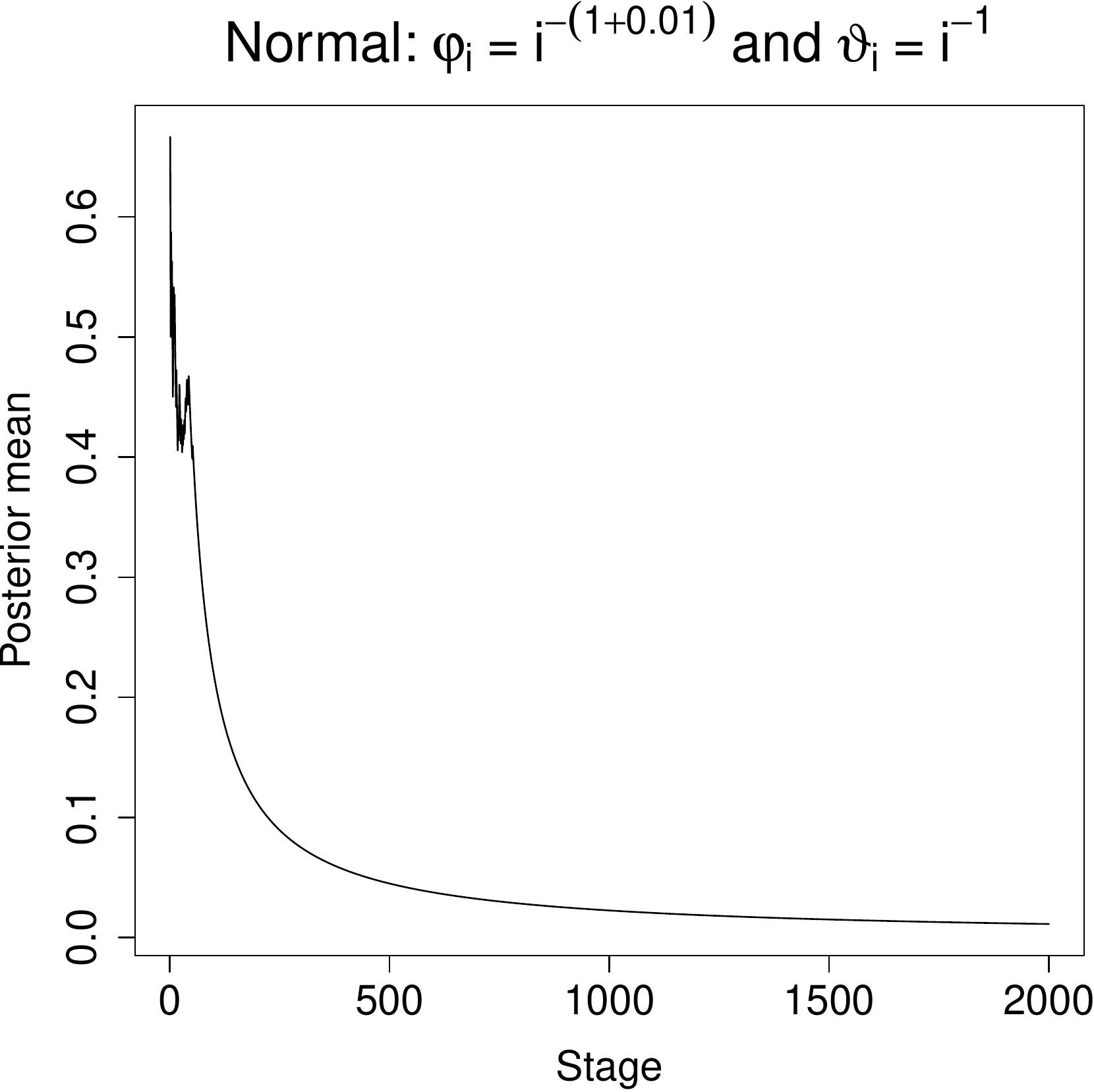}}
\caption{Example 2 revisited: Convergence and divergence for normal series with nonparametric bound.}
\label{fig:example_normal2}
\end{figure}

\subsubsection{Example 3 revisited: Dependent hierarchical normal distribution}
\label{subsubsec:bayesian_normal_dependent}
As in Section \ref{subsec:bayesian_normal_dependent} we again consider $[X_i|\xi]\sim N\left(\mu_i,\xi\sigma^2_i\right)$, independently, for $i\geq 1$, where $\xi\sim U(0,1)$,
$\mu_i\sim N\left(0,\phi^2_i\right)$ and $\sigma^2_i\sim\mathcal E(\vartheta_i)$, 
but now with the parametric bound for the partial sums replaced with the nonparametric form (\ref{eq:ar1_bound3}), with $\hat C_1=0.71$, the same initial constant
used for the nonparametric bound for the normal setup in Section \ref{subsubsec:bayesian_normal}. Figure \ref{fig:example_normal_dependent2} shows the relevant results
in this setup. The results are similar to the independent normal setup with nonparametric bound, and are very significant improvements 
to the results provided by the parametric bound displayed in Figure \ref{fig:example_normal_dependent}. Indeed, Figure \ref{fig:example_normal_dependent} 
shows that none of the convergence and divergence results for the parametric bound is convincing, even for such huge samples, and even after such long run-times.
In sharp contrast, the nonparametric bound results depicted by Figure \ref{fig:example_normal_dependent2} are highly persuasive, even with such small samples, requiring
run-times of less than a second on our ordinary dual core laptop.
\begin{figure}
\centering
\subfigure [Divergence.]{ \label{fig:normal_dep1_2}
\includegraphics[width=6cm,height=5cm]{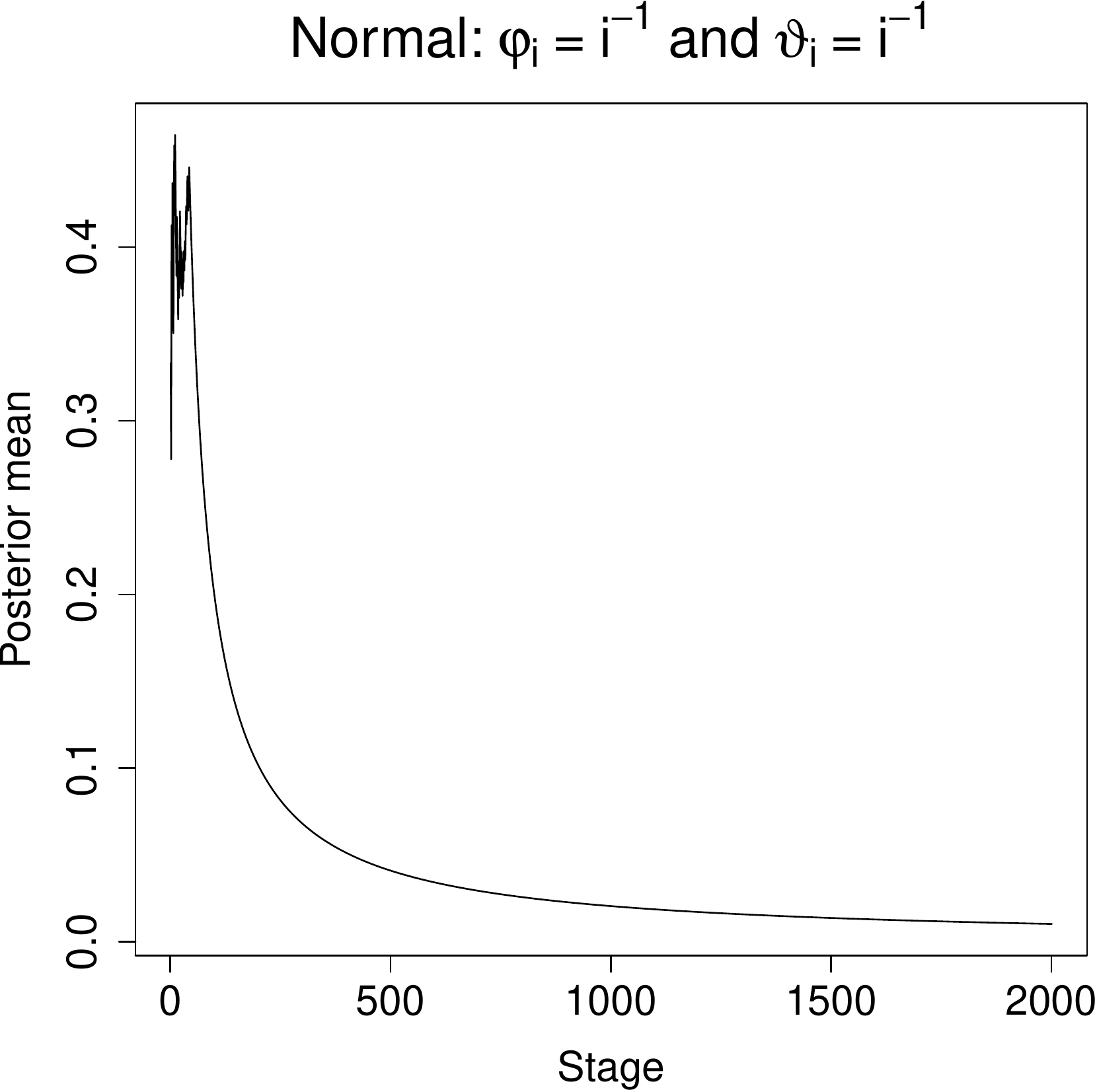}}
\hspace{2mm}
\subfigure [Convergence.]{ \label{fig:normal_dep2_2}
\includegraphics[width=6cm,height=5cm]{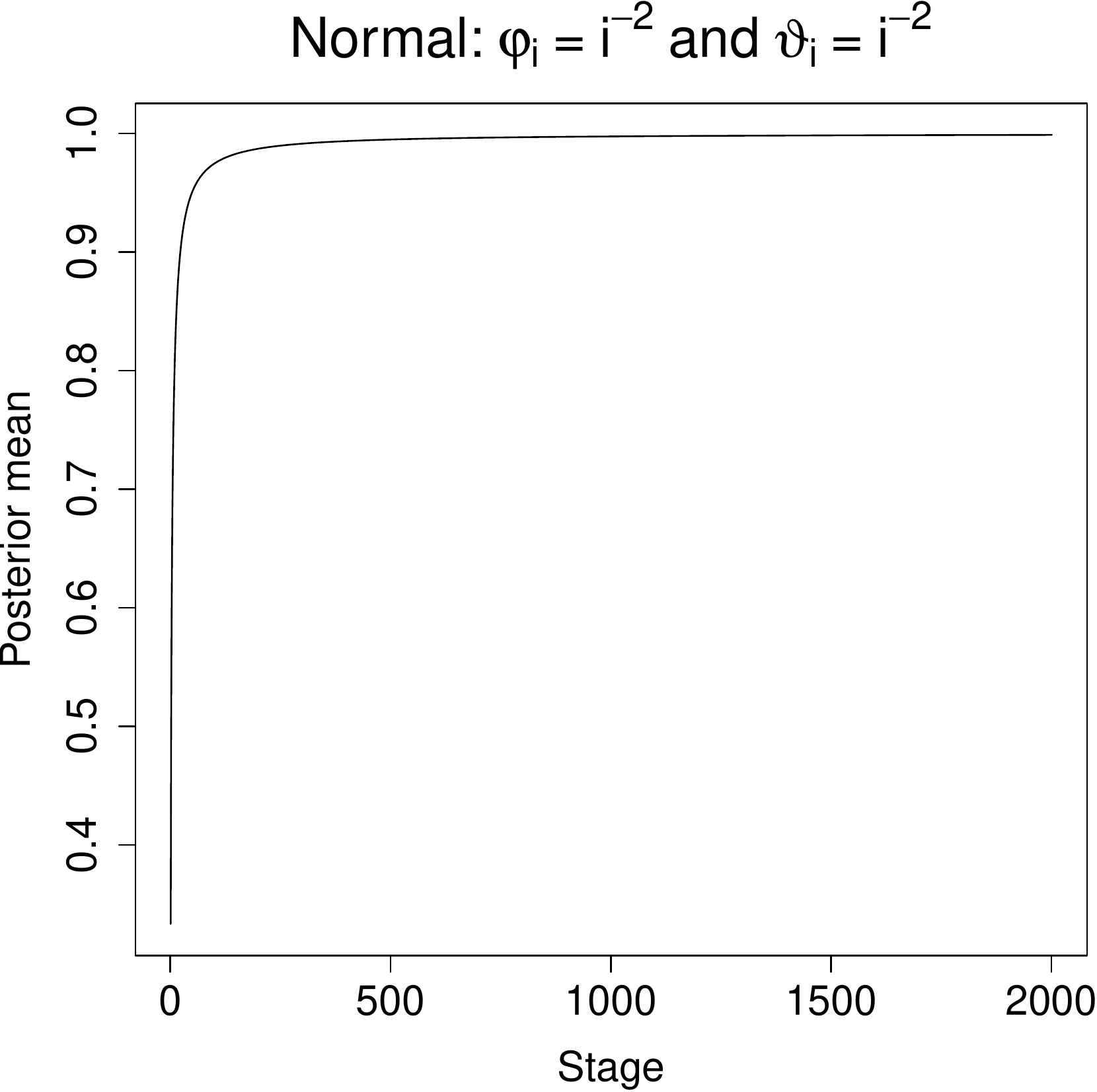}}\\
\subfigure [Convergence.]{ \label{fig:normal_dep3_2}
\includegraphics[width=6cm,height=5cm]{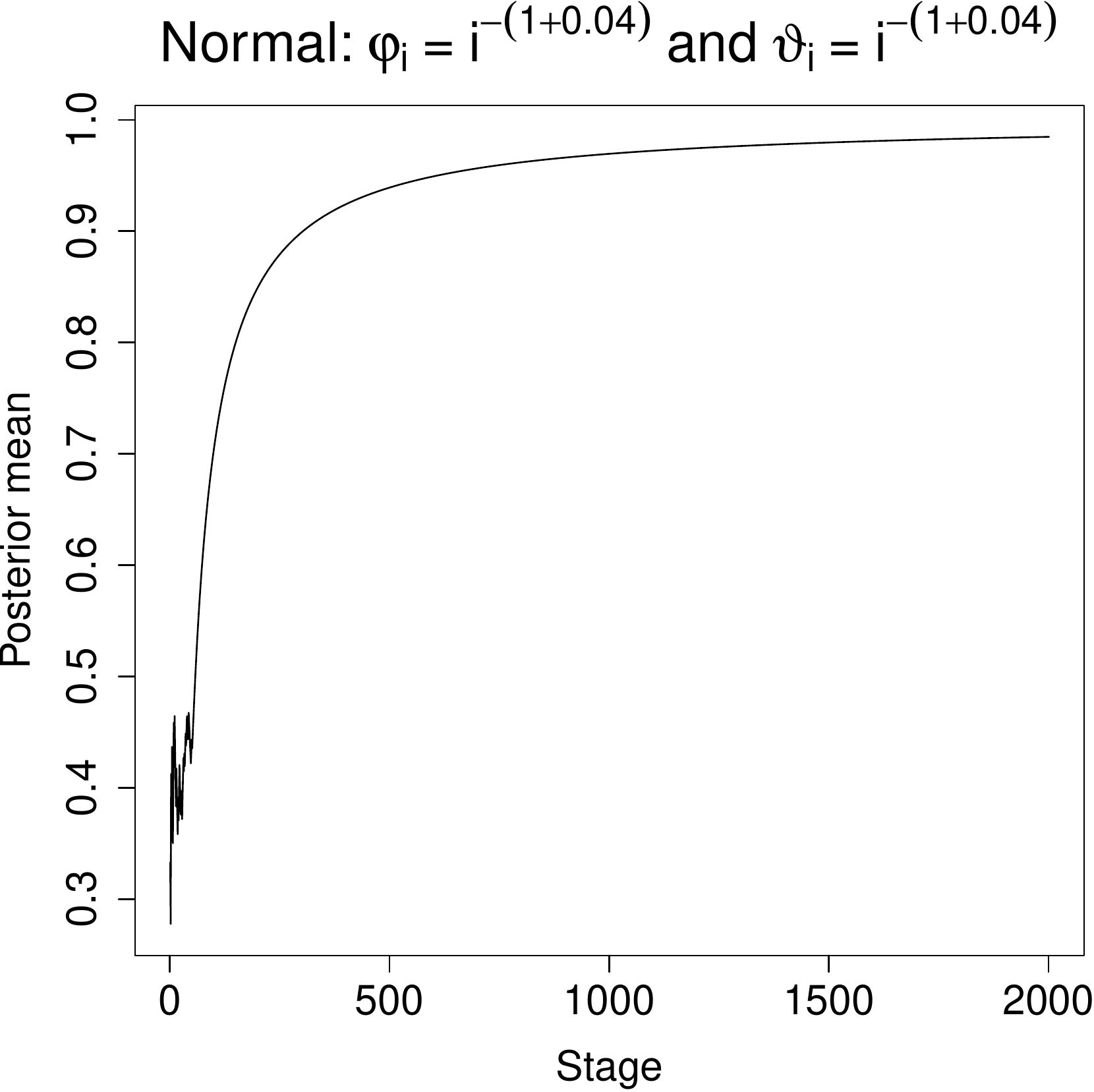}}
\hspace{2mm}
\subfigure [Divergence.]{ \label{fig:normal_dep4_2}
\includegraphics[width=6cm,height=5cm]{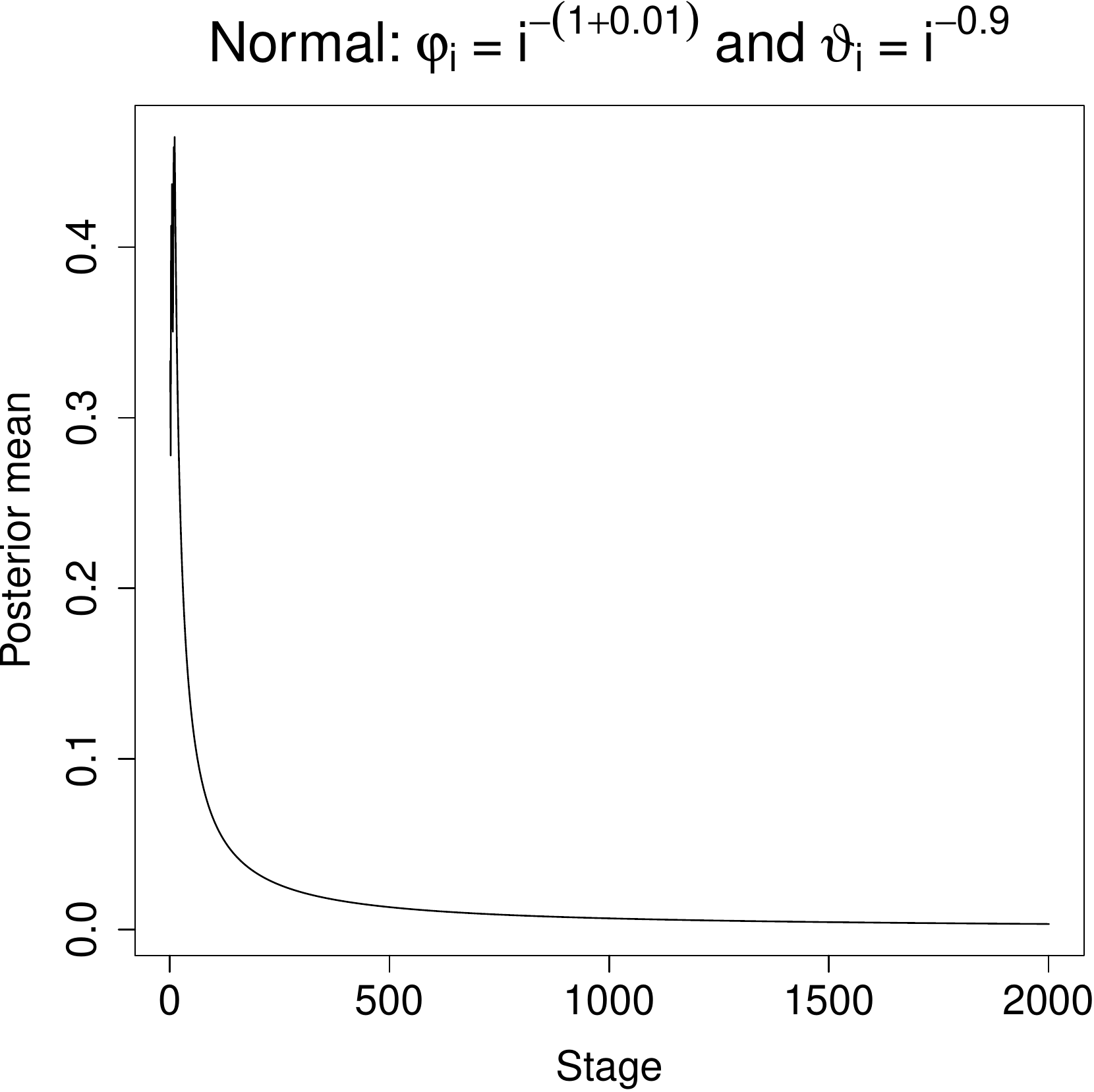}}\\
\subfigure [Divergence.]{ \label{fig:normal_dep5_2}
\includegraphics[width=6cm,height=5cm]{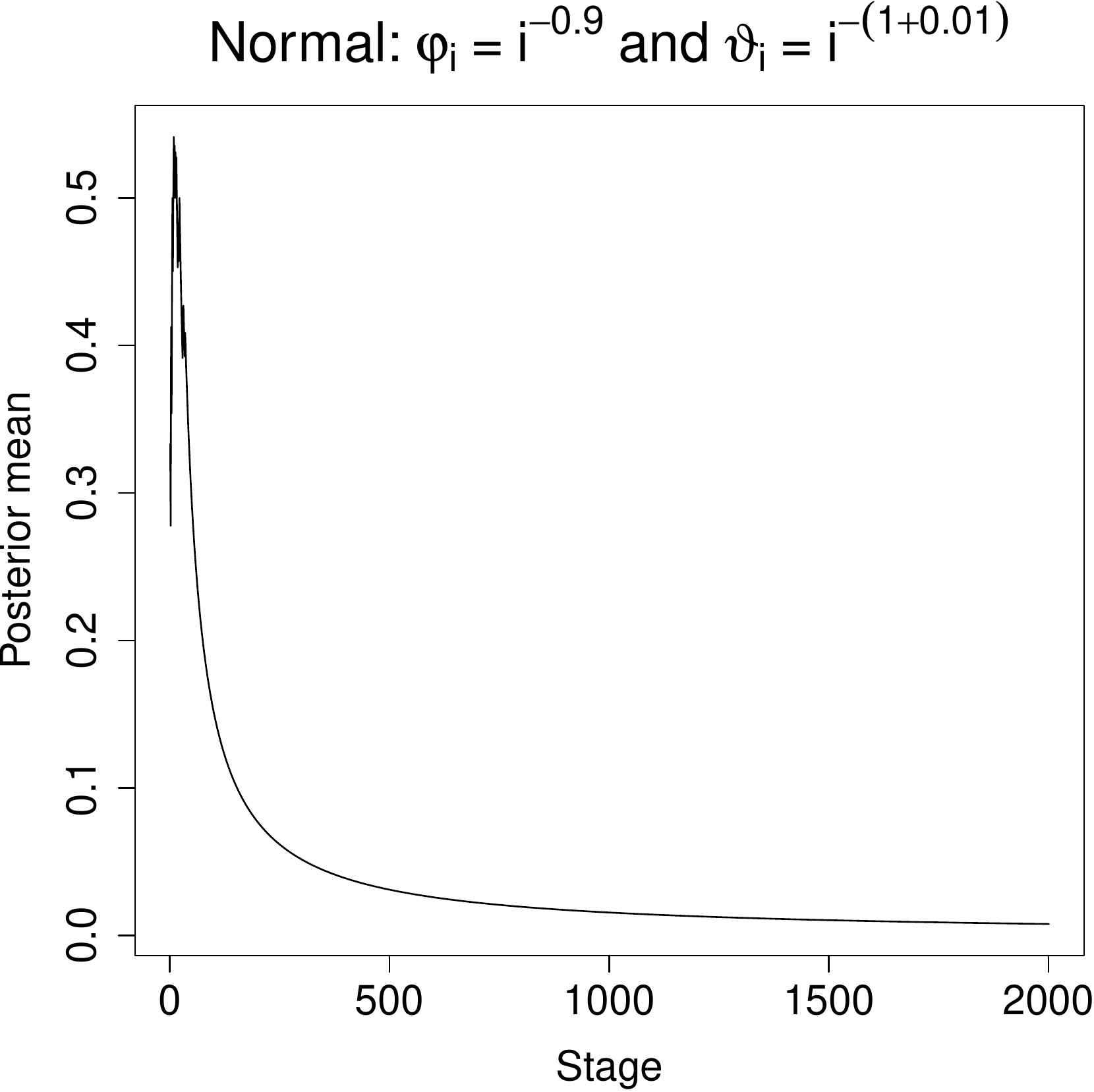}}
\hspace{2mm}
\subfigure [Divergence.]{ \label{fig:normal_dep6_2}
\includegraphics[width=6cm,height=5cm]{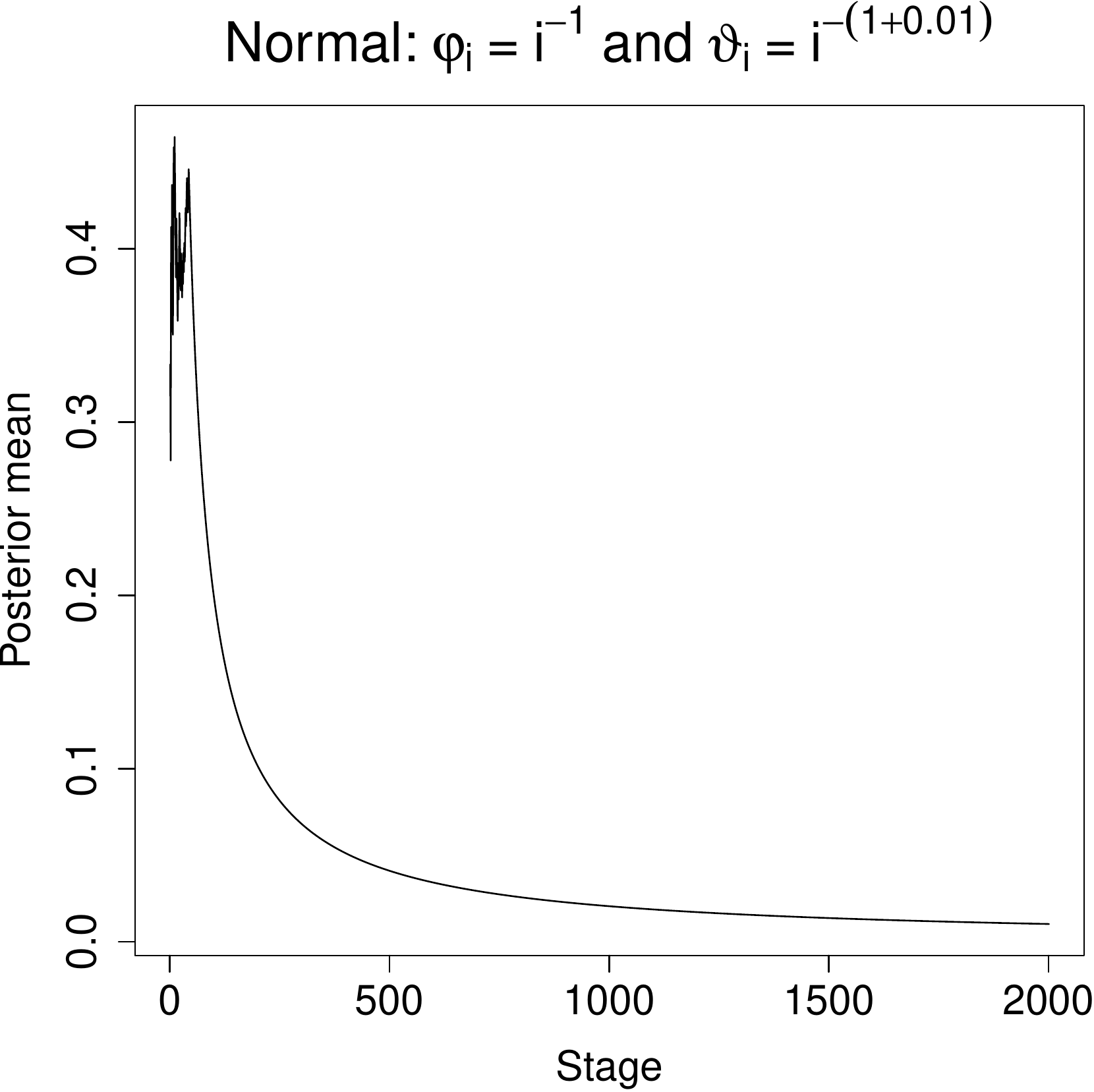}}\\
\subfigure [Divergence.]{ \label{fig:normal_dep7_2}
\includegraphics[width=6cm,height=5cm]{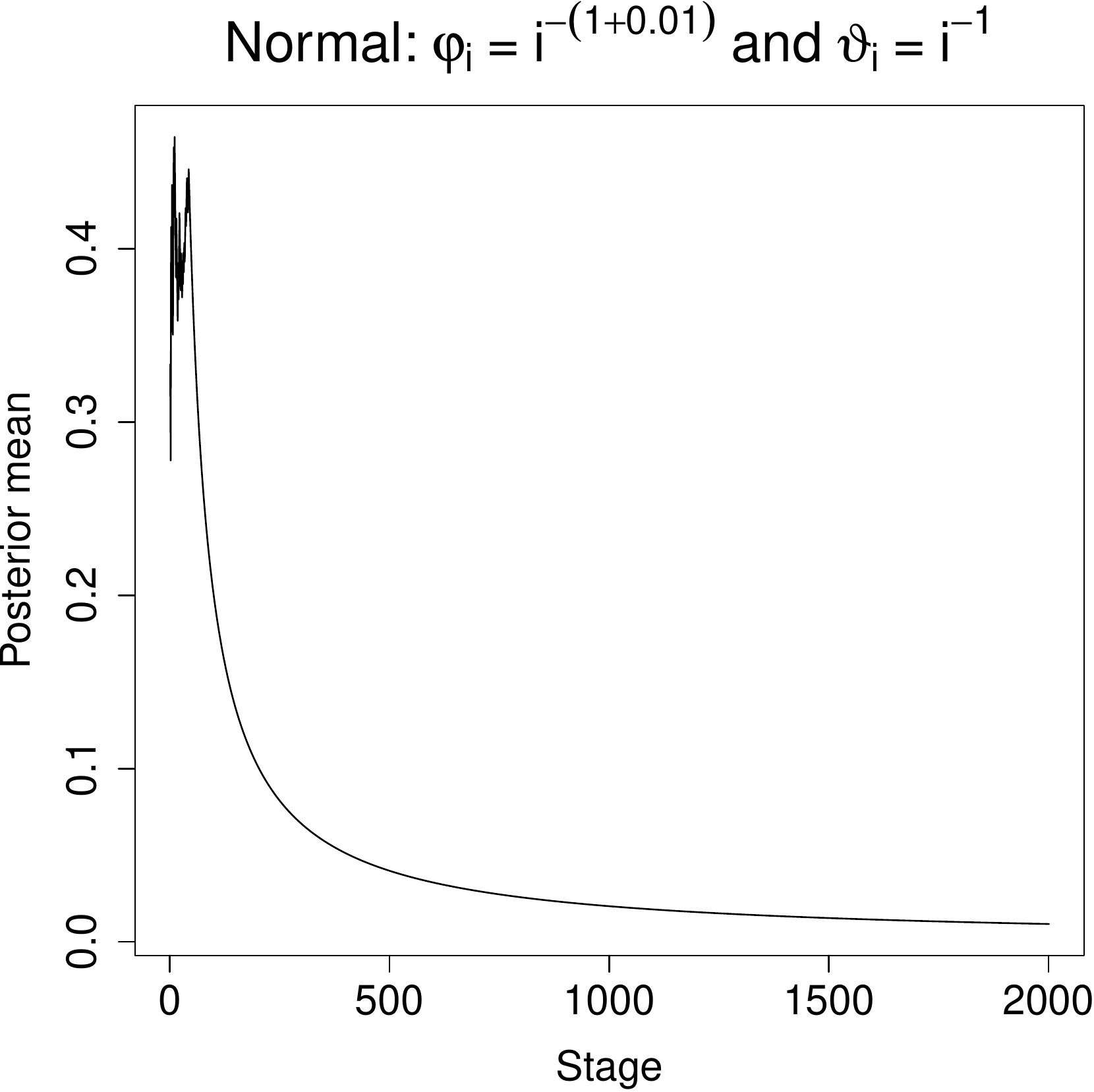}}
\caption{Example 3 revisited: Convergence and divergence for dependent normal series with nonparametric bound.}
\label{fig:example_normal_dependent2}
\end{figure}

\subsubsection{Example 4 revisited: Dependent state-space random series}
\label{subsubsec:bayesian_ss}
Following Section \ref{subsec:bayesian_ss} we consider random series of the form $\sum_{i=1}^{\infty}X_i\theta_i$ where 
for $i\geq 1$, $\theta_i\sim \mathcal E(\psi_i)$ independently, 
and $X_i$ has the state-space representation given by (\ref{eq:ss2}) and (\ref{eq:ss3}).
The rest of the model details remain the same as in Section \ref{subsec:bayesian_ss}.

Application of our new nonparametric bound to the partial sums, with $\hat C_1=0.725$, which is the same as that of the exponential series with the nonparametric bound,
we obtain correct results in all the cases, as displayed by Figure \ref{fig:example_ss3}.
In fact, the nonparametric bound not only matches the performance of the parametric bound method 
detailed in Section \ref{subsec:bayesian_ss}, it seems to outperform the latter for $\psi=i^{-(1+0.001)}$ in terms of faster convergence. 
\begin{figure}
\centering
\subfigure [Divergence.]{ \label{fig:ss1_3}
\includegraphics[width=6cm,height=5cm]{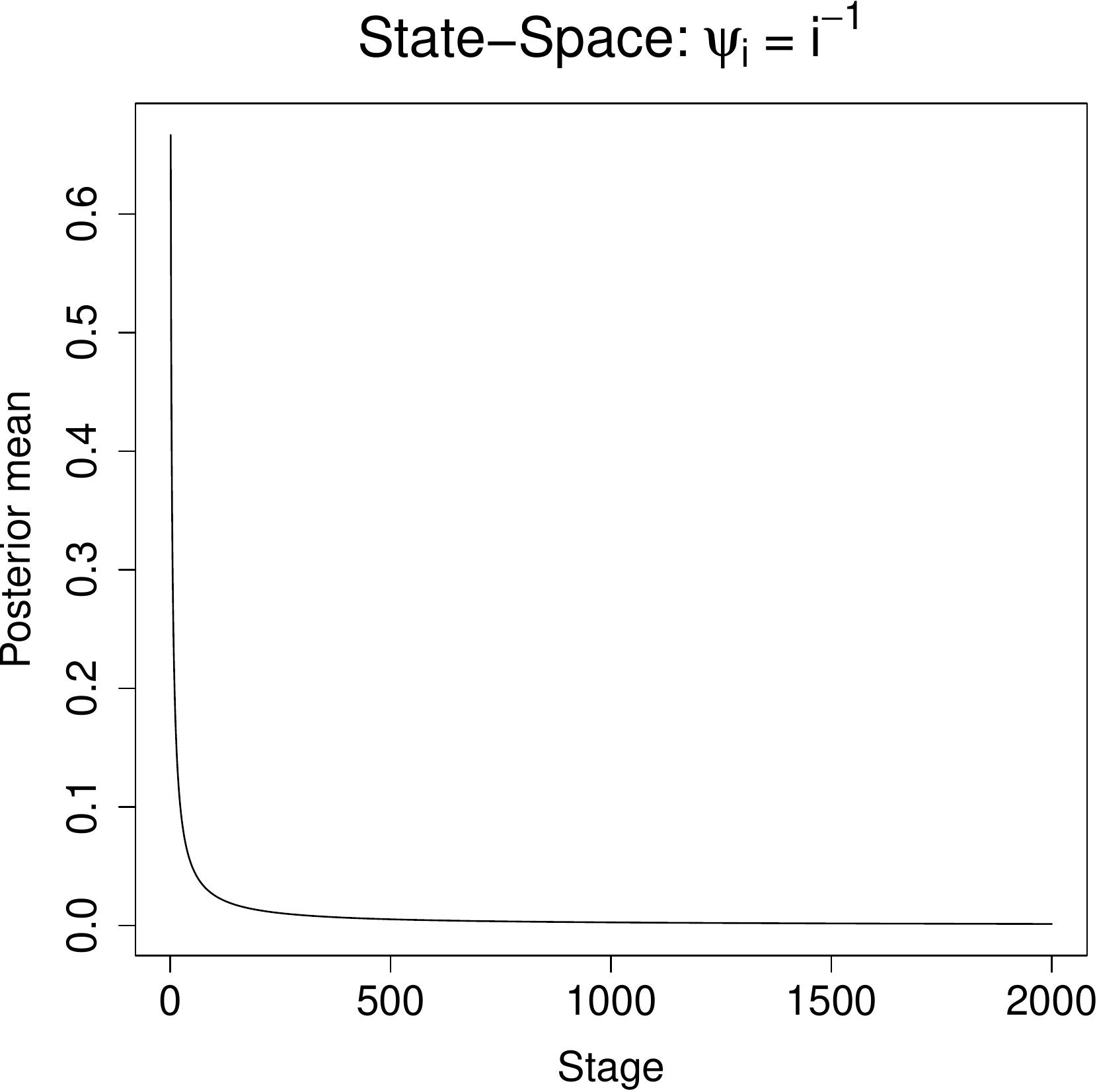}}
\hspace{2mm}
\subfigure [Convergence.]{ \label{fig:ss2_3}
\includegraphics[width=6cm,height=5cm]{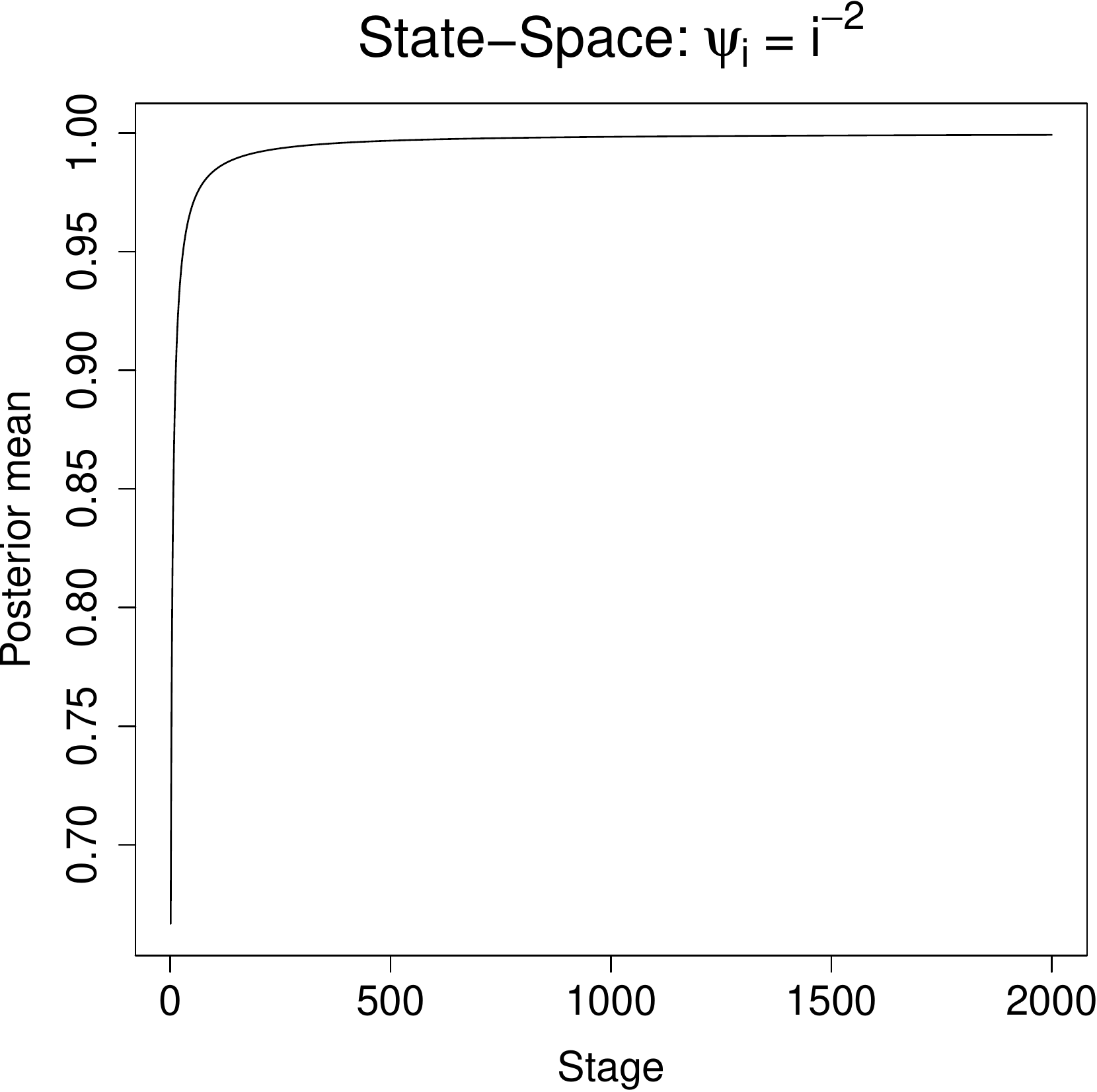}}\\
\subfigure [Convergence.]{ \label{fig:ss3_3}
\includegraphics[width=6cm,height=5cm]{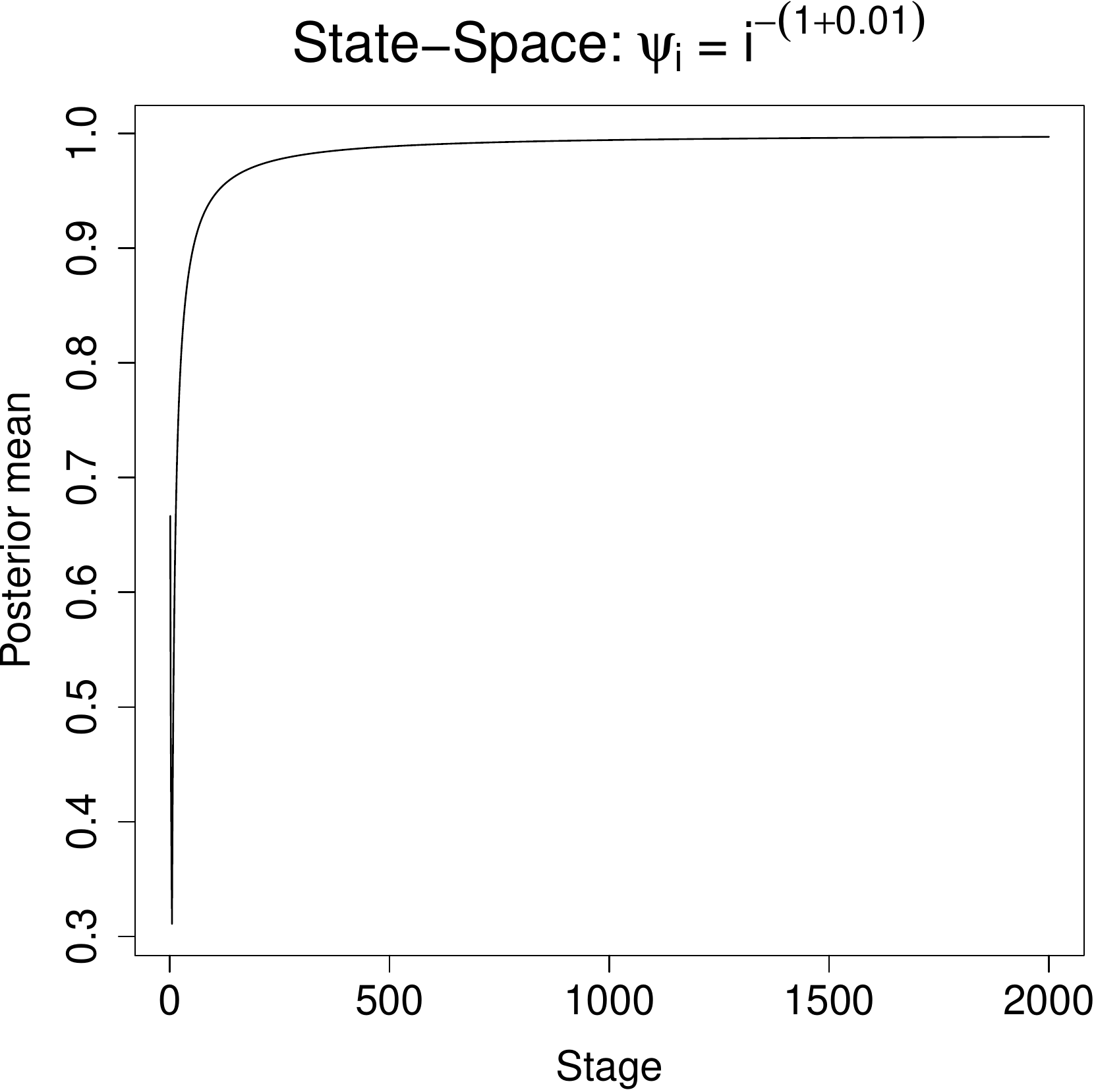}}
\hspace{2mm}
\subfigure [Convergence.]{ \label{fig:ss4_3}
\includegraphics[width=6cm,height=5cm]{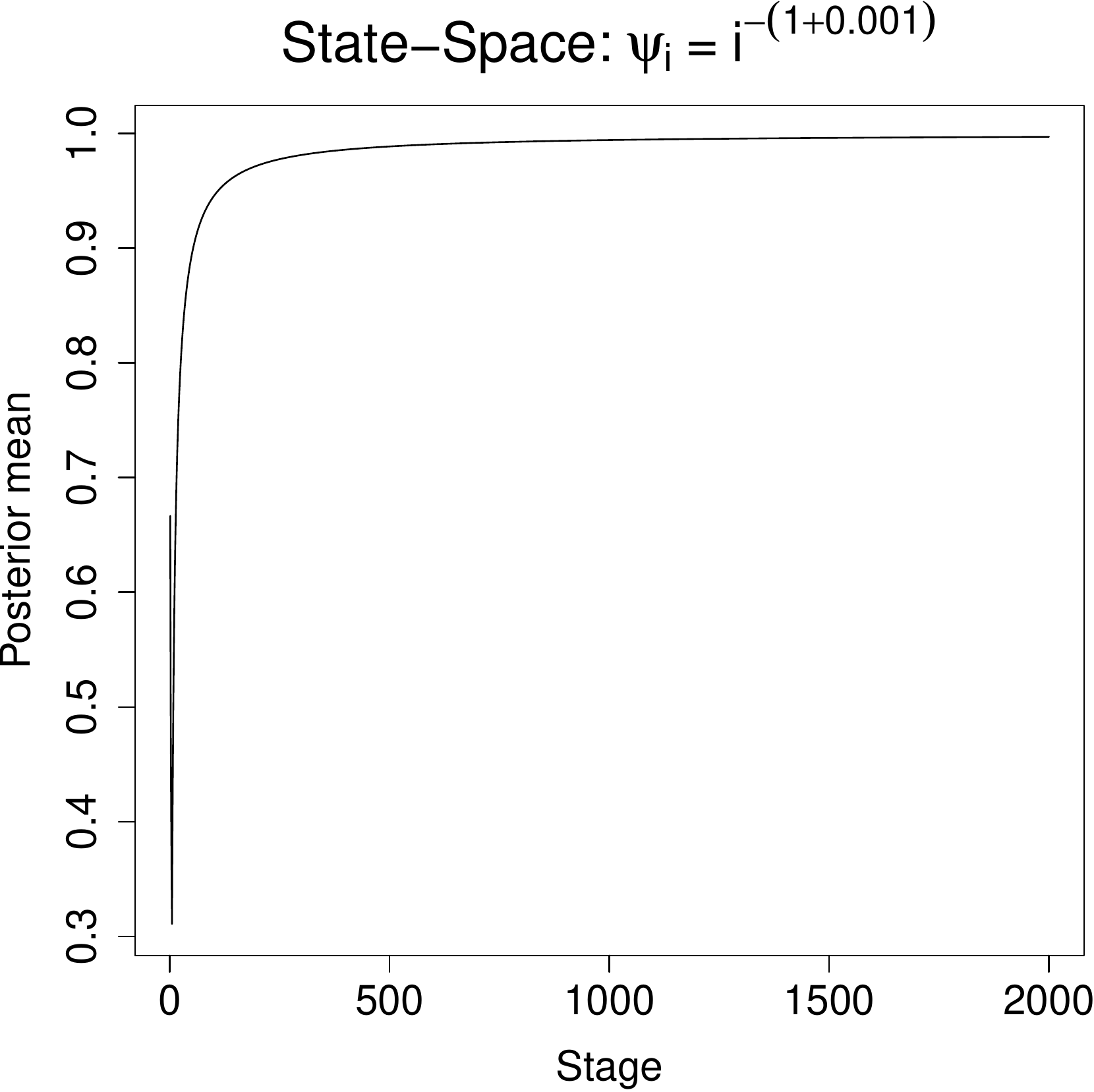}}\\
\subfigure [Convergence.]{ \label{fig:ss5_3}
\includegraphics[width=6cm,height=5cm]{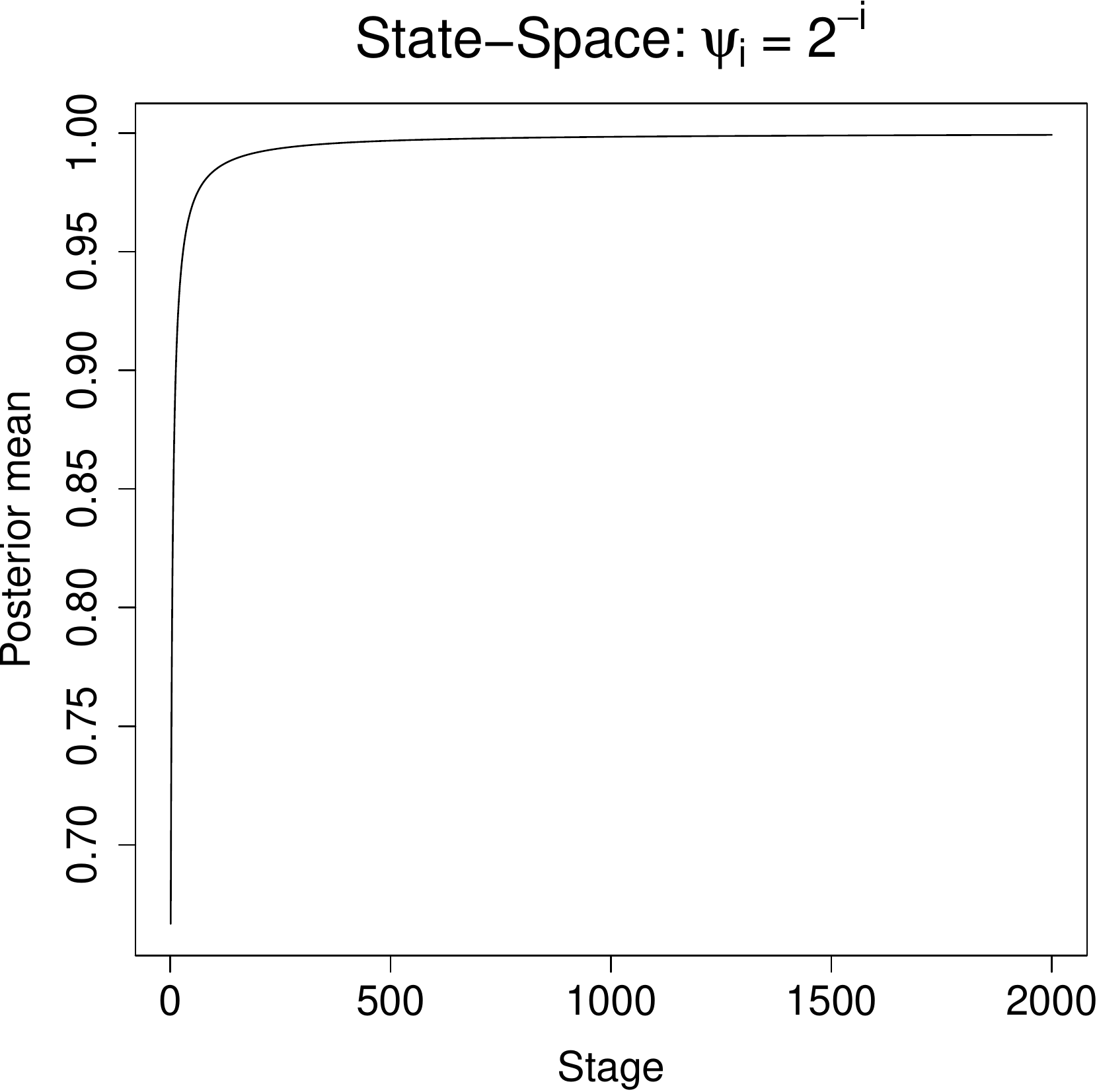}}
\hspace{2mm}
\subfigure [Convergence.]{ \label{fig:ss6_3}
\includegraphics[width=6cm,height=5cm]{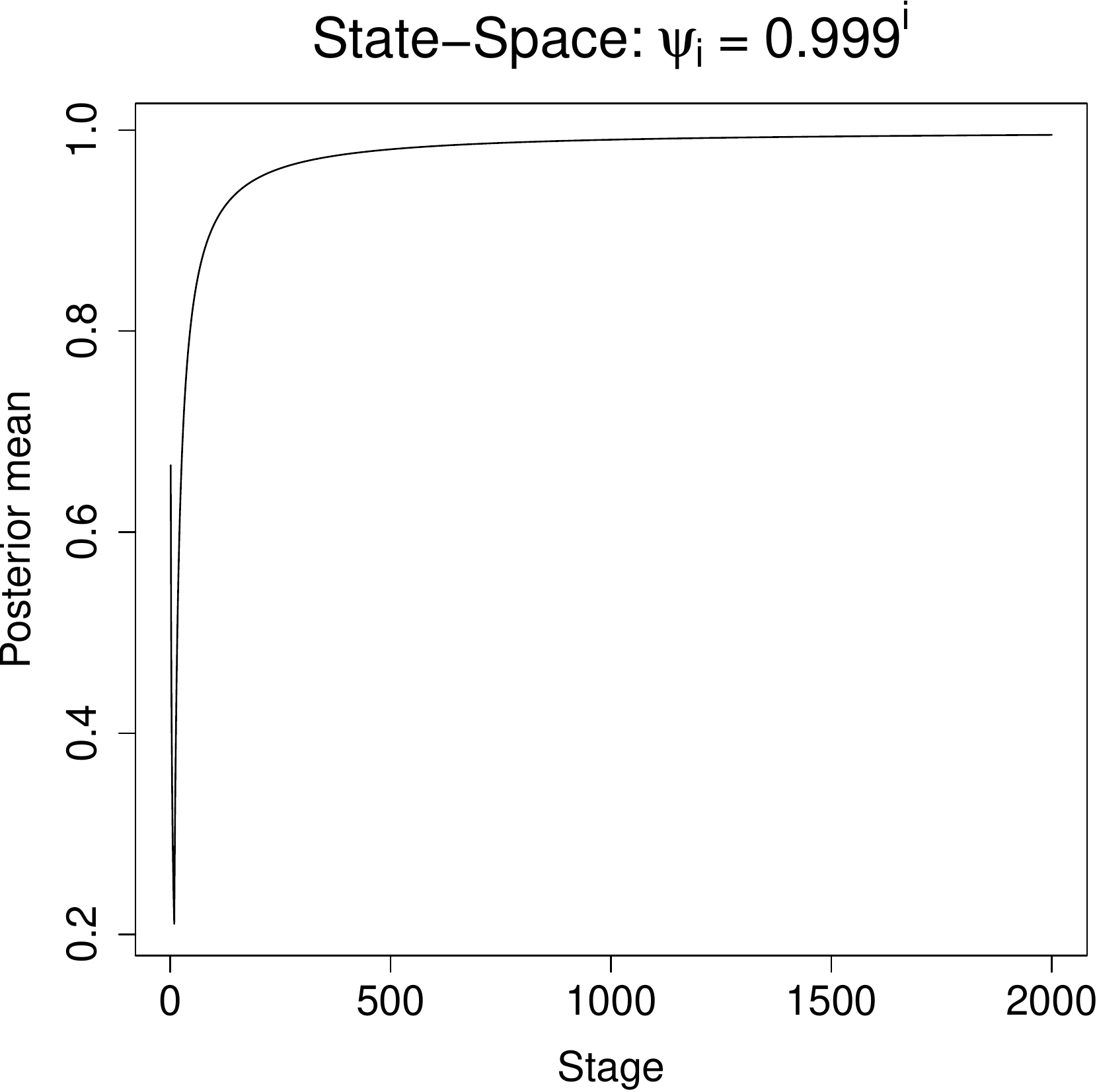}}\\
\subfigure [Divergence.]{ \label{fig:ss7_3}
\includegraphics[width=6cm,height=5cm]{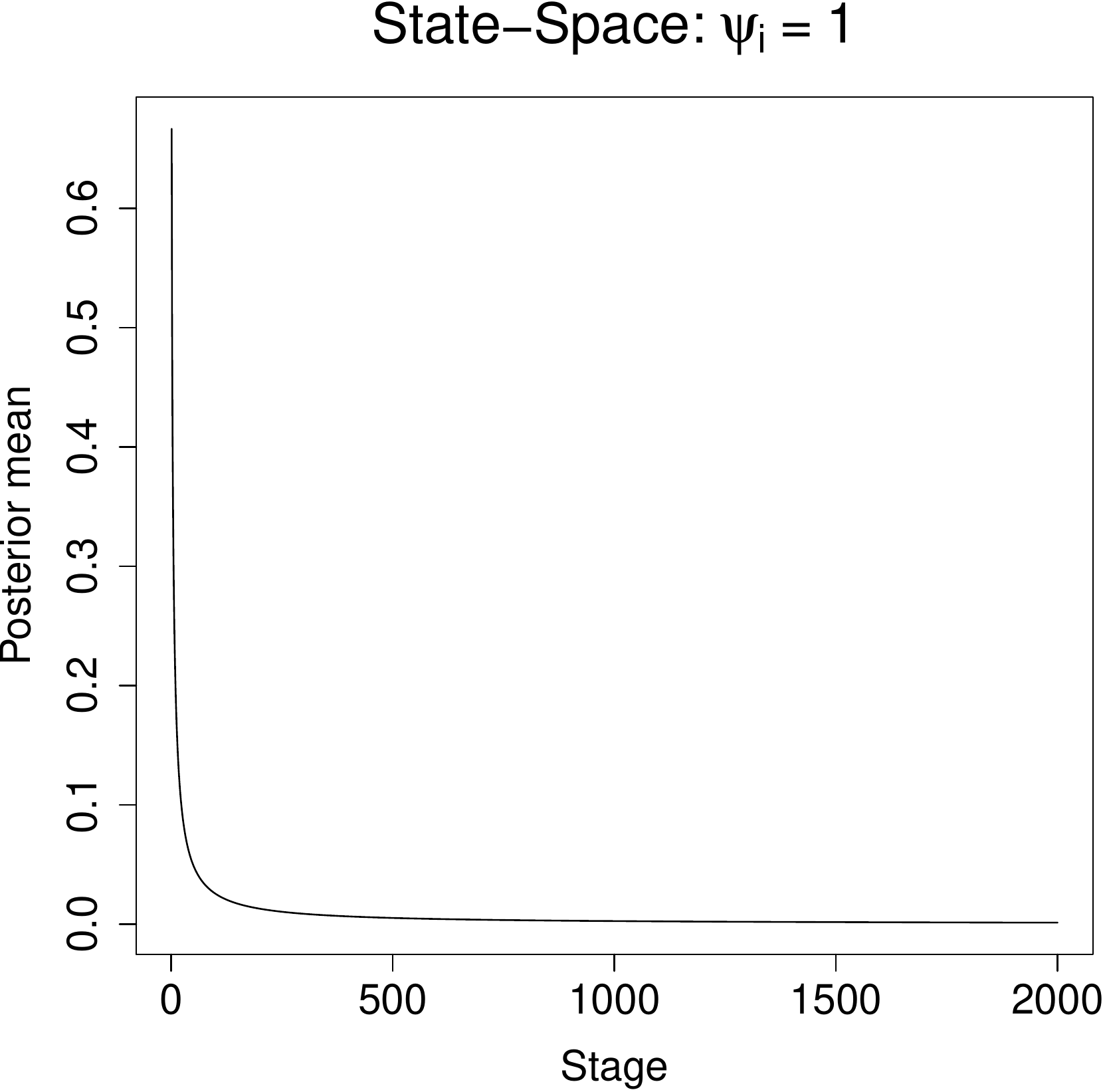}}
\caption{Example 4 revisited: Convergence and divergence for state-space series with nonparametric bound.}
\label{fig:example_ss3}
\end{figure}

\subsubsection{Example 5 revisited: Dependent state-space random series with hierarchical exponential distribution}
\label{subsubsec:bayesian_ss2}

In the state-space model with hierarchical exponential distribution considered in Section \ref{subsec:bayesian_ss2}, we now
apply the nonparametric bound with $\hat C_1=0.725$ to address convergence properties of $\sum_{i=1}^{\infty}X_i\theta_i$ using our Bayesian
methodology. The results 
displayed in Figure \ref{fig:example_ss4} again shows very accurate detection of convergence properties of the underlying infinite series
even with small samples sizes. However, it is to be noted that because of the hierarchy in the exponential distribution, a little subtlety has been sacrificed by our
method as it is unable to correctly diagnose divergence for $\psi=i^{-p}$ when $p\in(0.997,1]$. 
\begin{figure}
\centering
\subfigure [Divergence.]{ \label{fig:ss1_4}
\includegraphics[width=6cm,height=5cm]{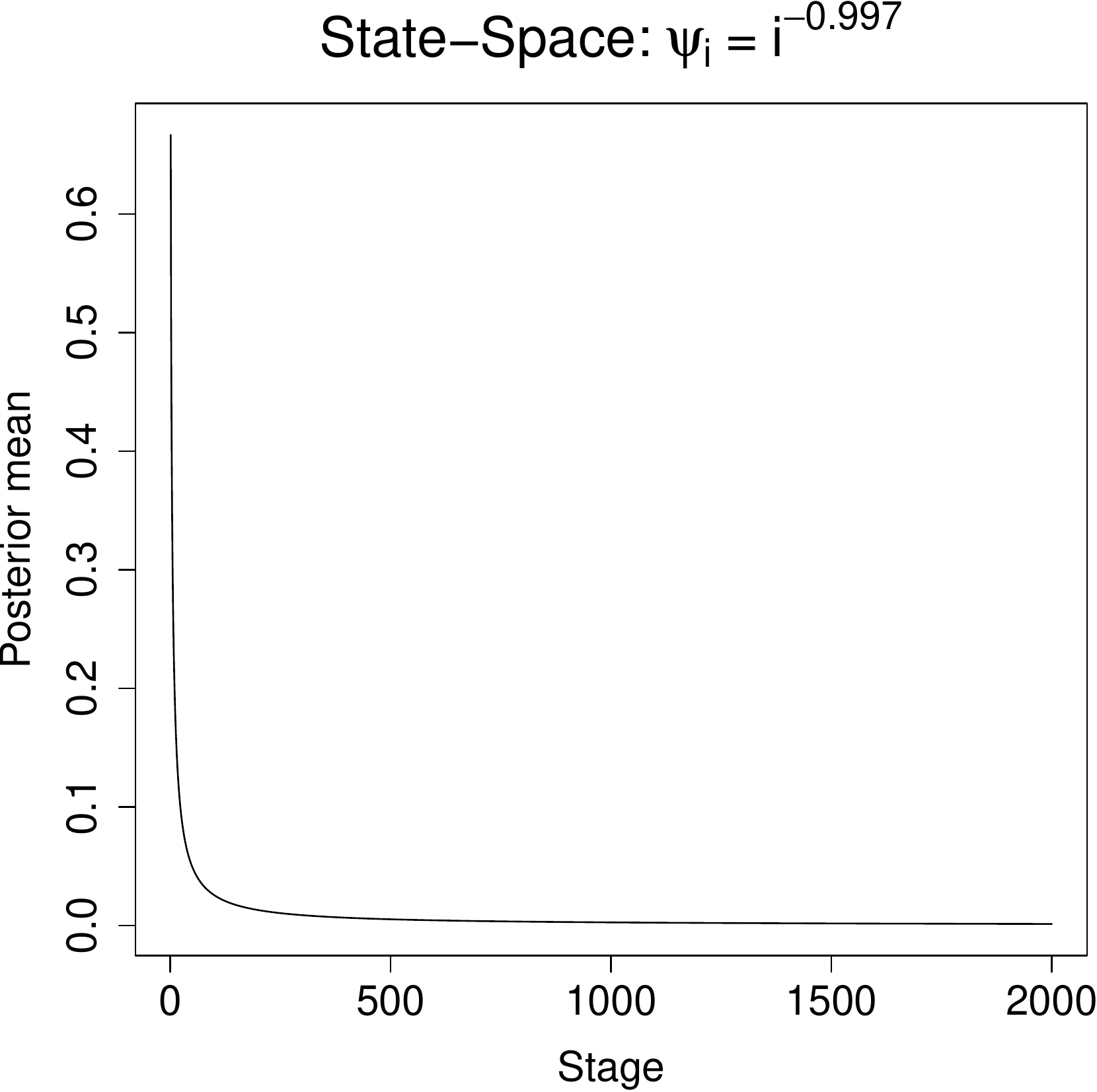}}
\hspace{2mm}
\subfigure [Convergence.]{ \label{fig:ss2_4}
\includegraphics[width=6cm,height=5cm]{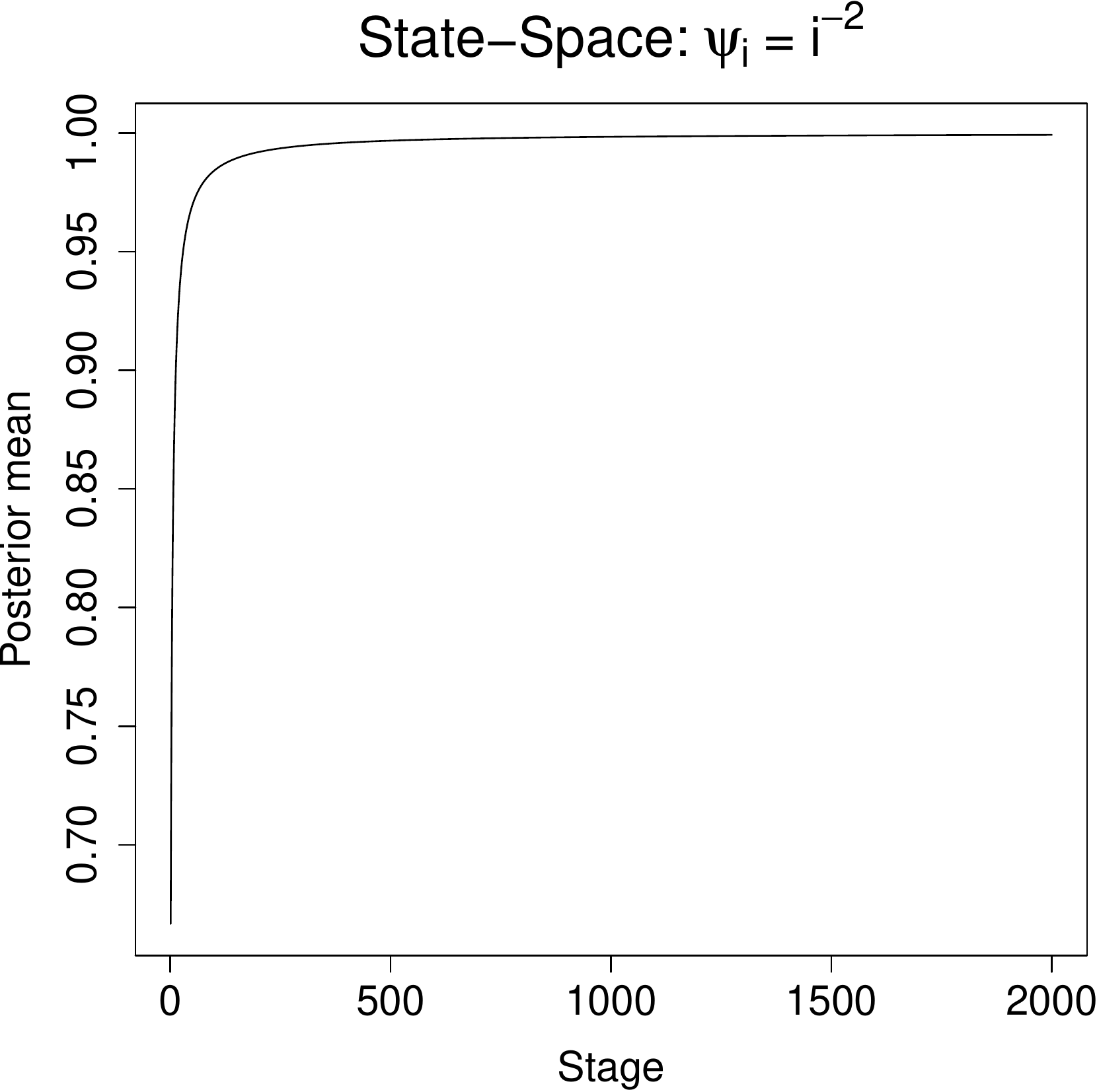}}\\
\subfigure [Convergence.]{ \label{fig:ss3_4}
\includegraphics[width=6cm,height=5cm]{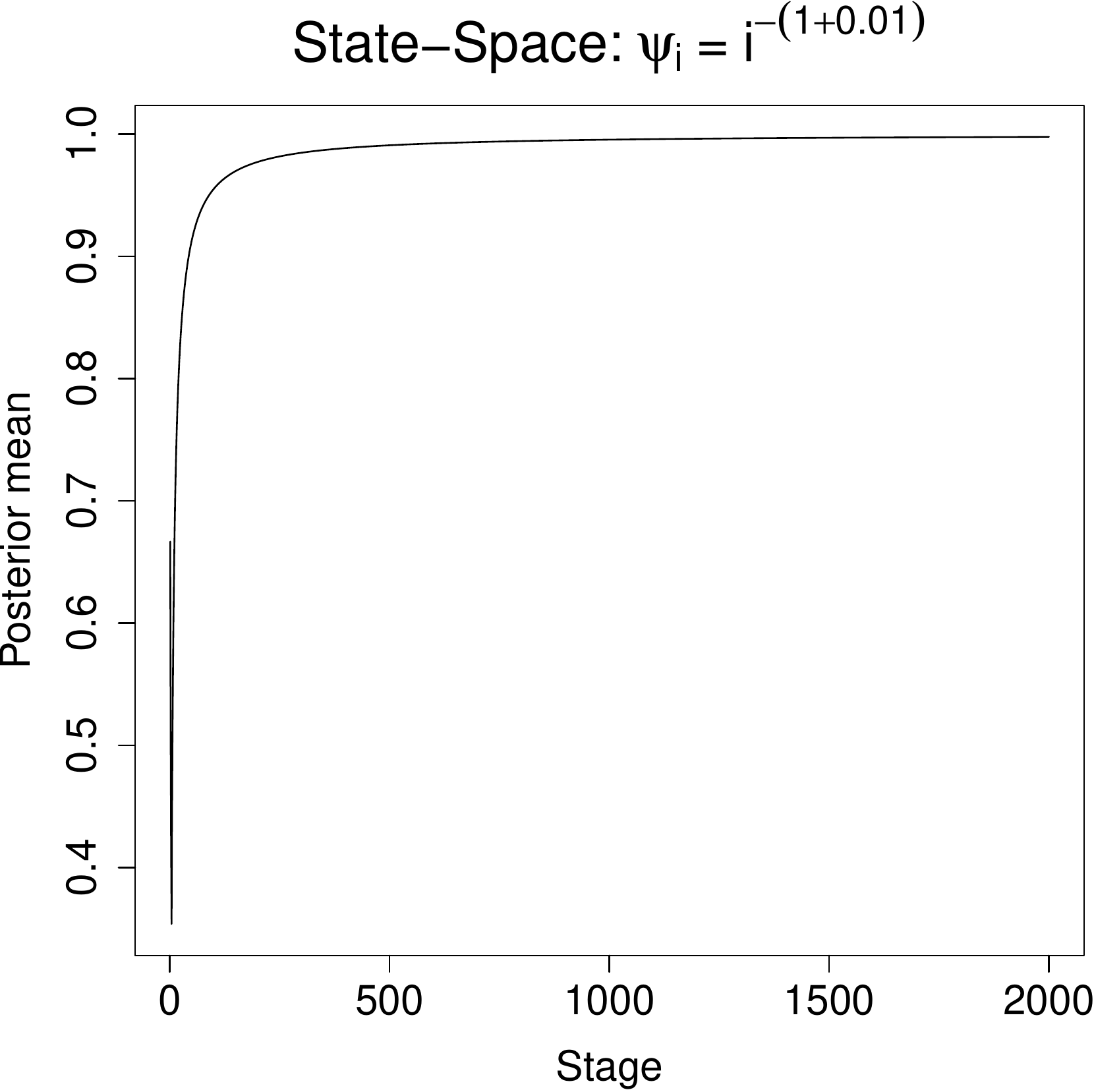}}
\hspace{2mm}
\subfigure [Convergence.]{ \label{fig:ss4_4}
\includegraphics[width=6cm,height=5cm]{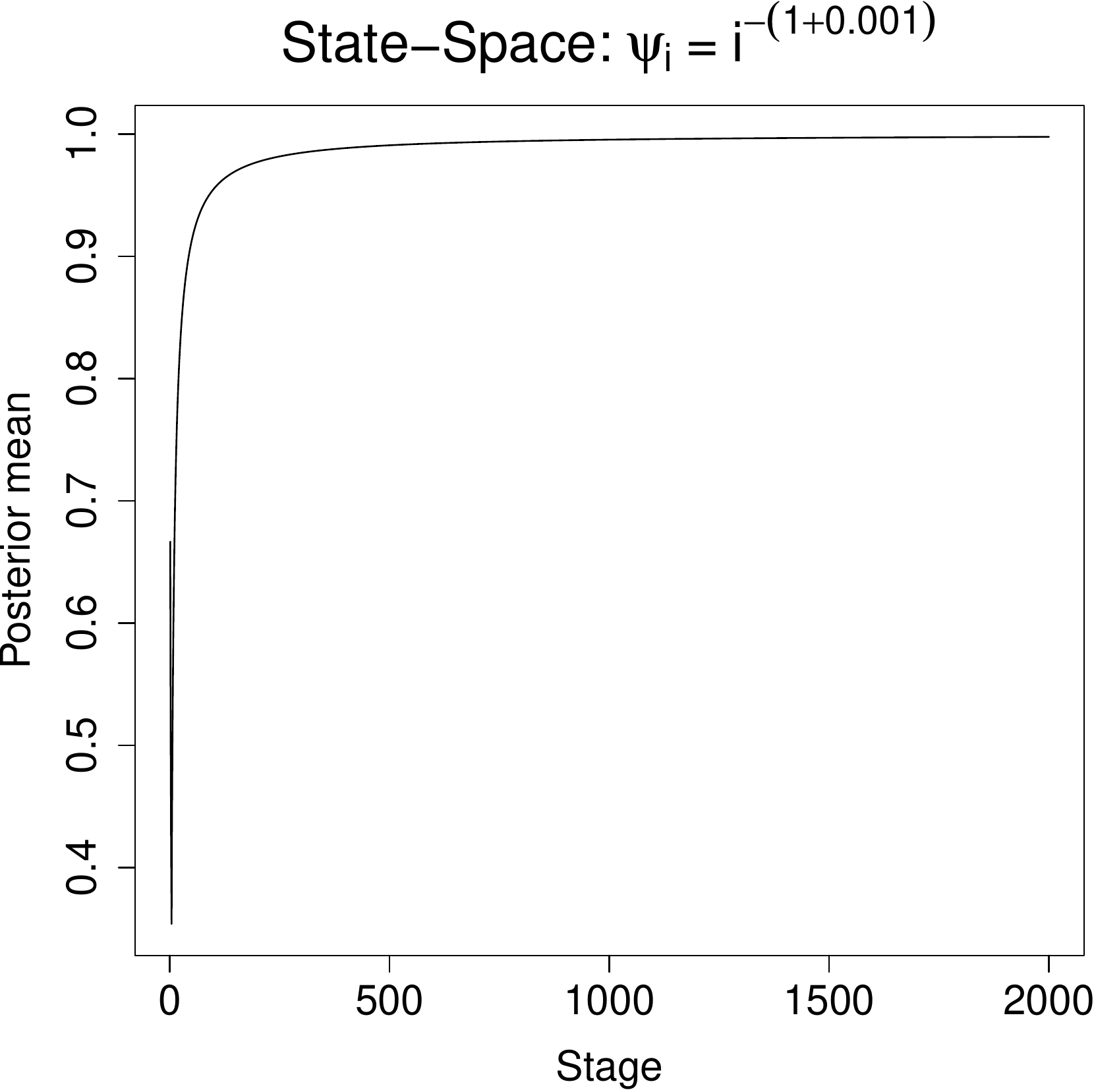}}\\
\subfigure [Convergence.]{ \label{fig:ss5_4}
\includegraphics[width=6cm,height=5cm]{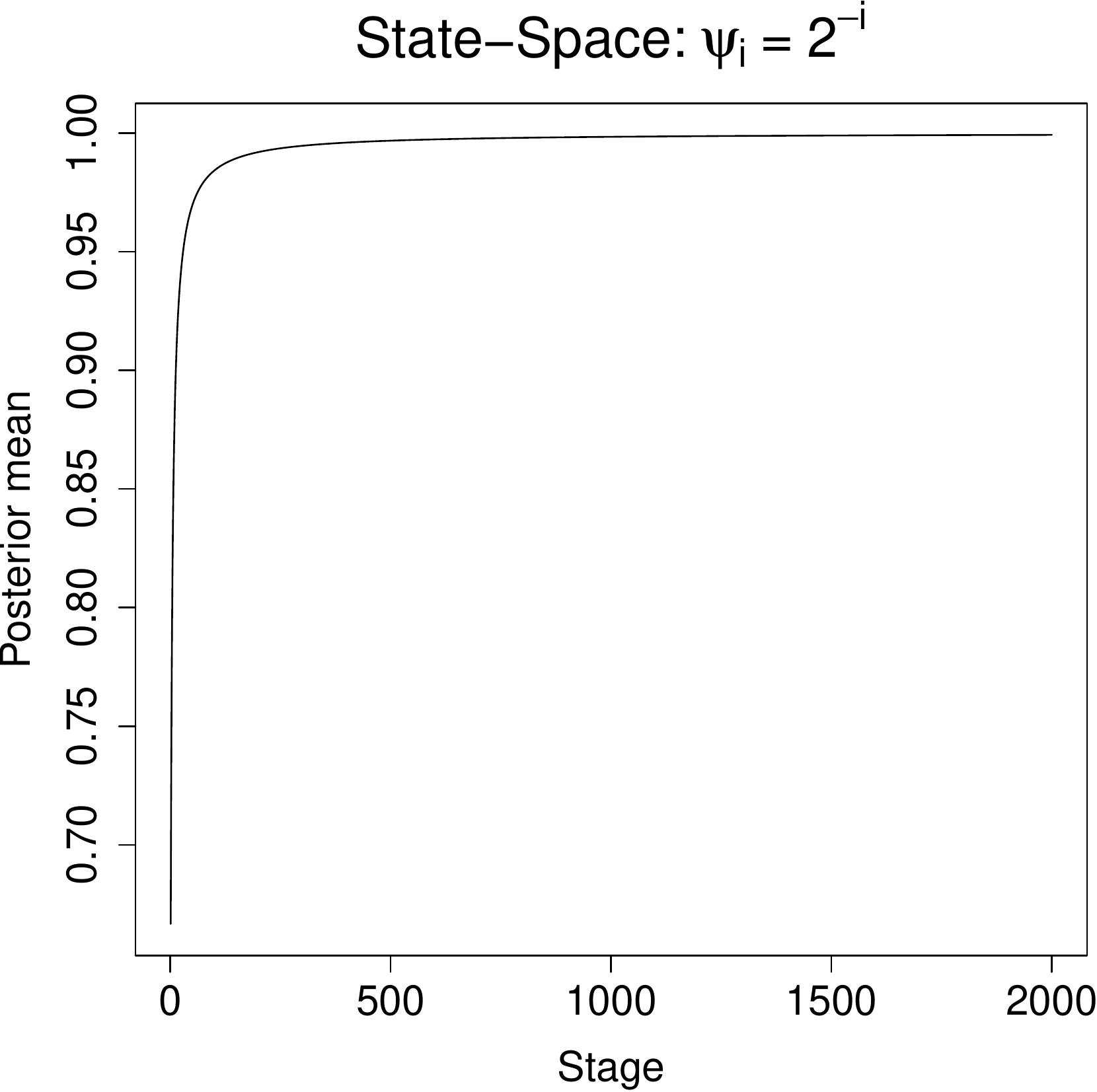}}
\hspace{2mm}
\subfigure [Convergence.]{ \label{fig:ss6_4}
\includegraphics[width=6cm,height=5cm]{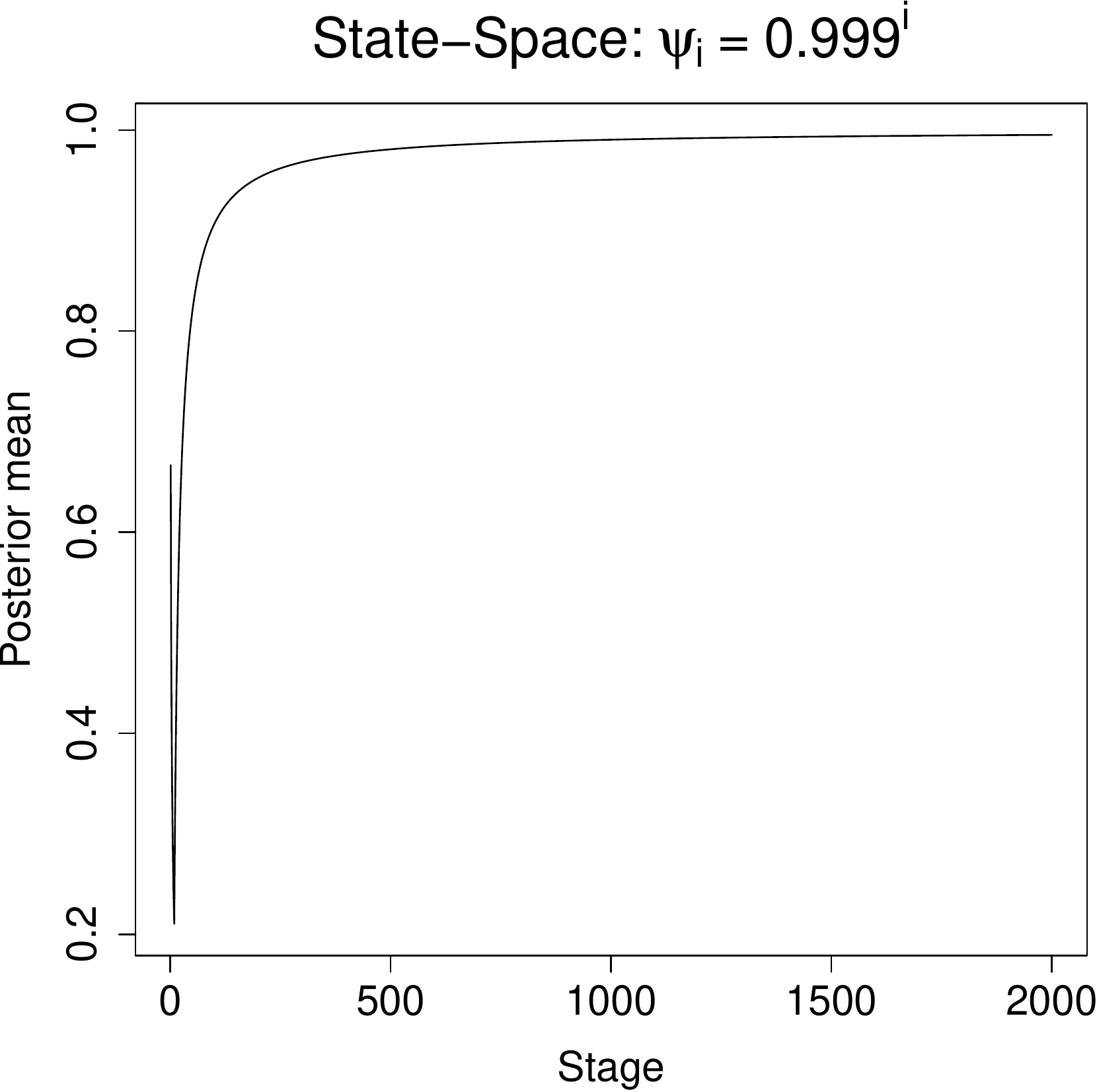}}\\
\subfigure [Divergence.]{ \label{fig:ss7_4}
\includegraphics[width=6cm,height=5cm]{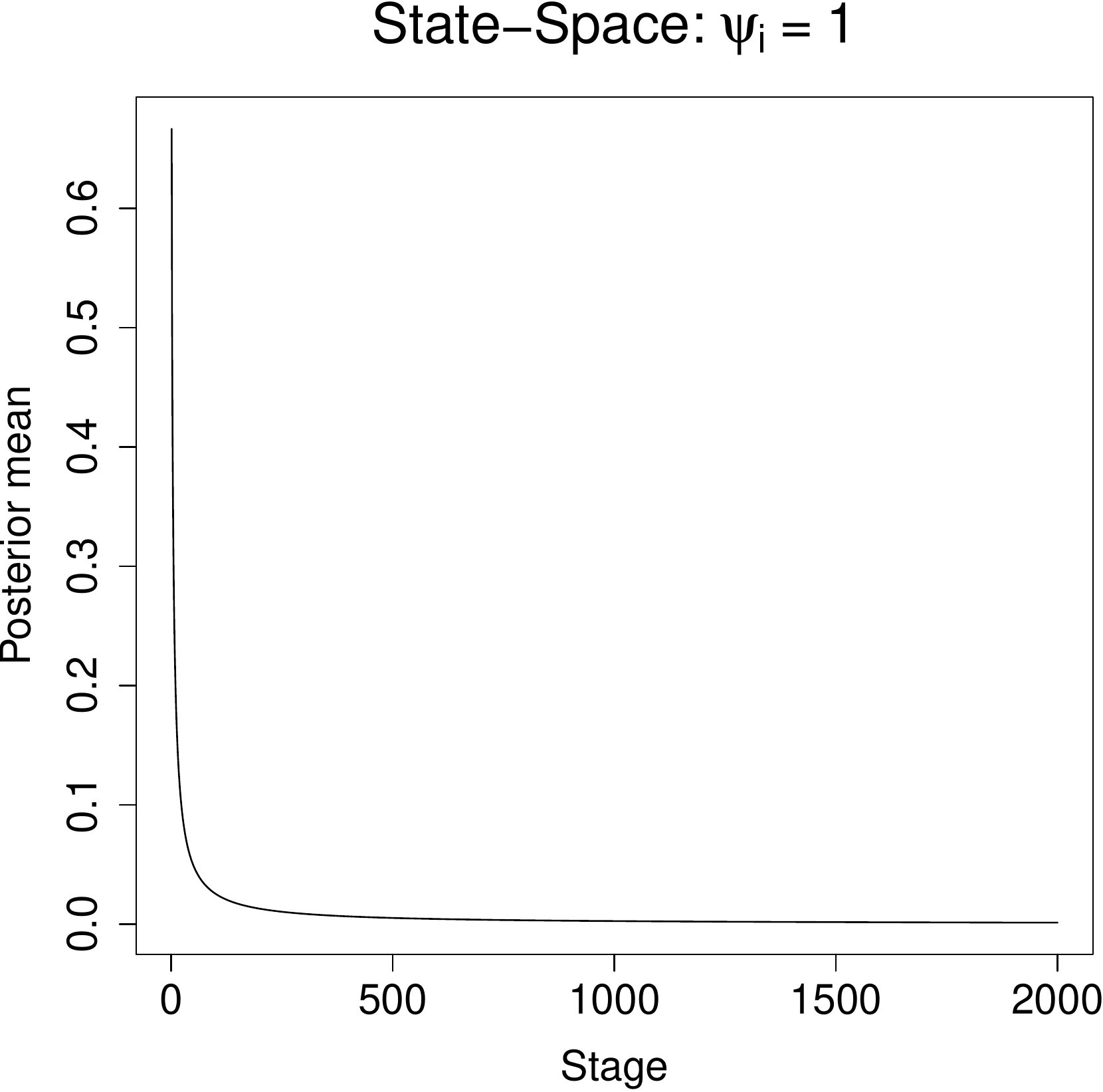}}
\caption{Example 5 revisited: Convergence and divergence for state-space series with hierarchical exponential distribution.}
\label{fig:example_ss4}
\end{figure}

\subsubsection{Example 6 revisited: Random Dirichlet series}
\label{subsubsec:bayesian_ds}
Again consider the RDS given by (\ref{eq:ds}).
Recall that this problem does not admit any theoretically valid upper bound since the summands take both positive and negative values with positive probabilities. 
Application of the general parametric upper bound (\ref{eq:S2}) to this problem in Section \ref{subsec:bayesian_ds} have led to wrong results in many cases of this problem.
Hence, we now employ our nonparametric bound to analyse convergence for the RDS.

As shown by Figure \ref{fig:ds}, application of our nonparametric bound to this problem for various values of $p$ 
revealed correct convergence analysis by our Bayesian method in all the cases.
To choose $\hat C_1$ appropriately in this problem, we first considered the deterministic series $\sum_{i=1}^{\infty}i^{-2p}$, whose convergence properties are known. 
For this series we selected
that value of $\hat C_1$ which led to correct convergence diagnosis of our Bayesian procedure with the nonparametric bound, for all (in practice, most) 
values of $p$. This led to $\hat C_1=0.44$, and this value turned out to be an excellent choice even for the RDS given by (\ref{eq:ds}). 

In other words, the nonparametric bound in this problem soundly beats the parametric bound.
\begin{figure}
\centering
\subfigure [Divergence.]{ \label{fig:ds1}
\includegraphics[width=6cm,height=5cm]{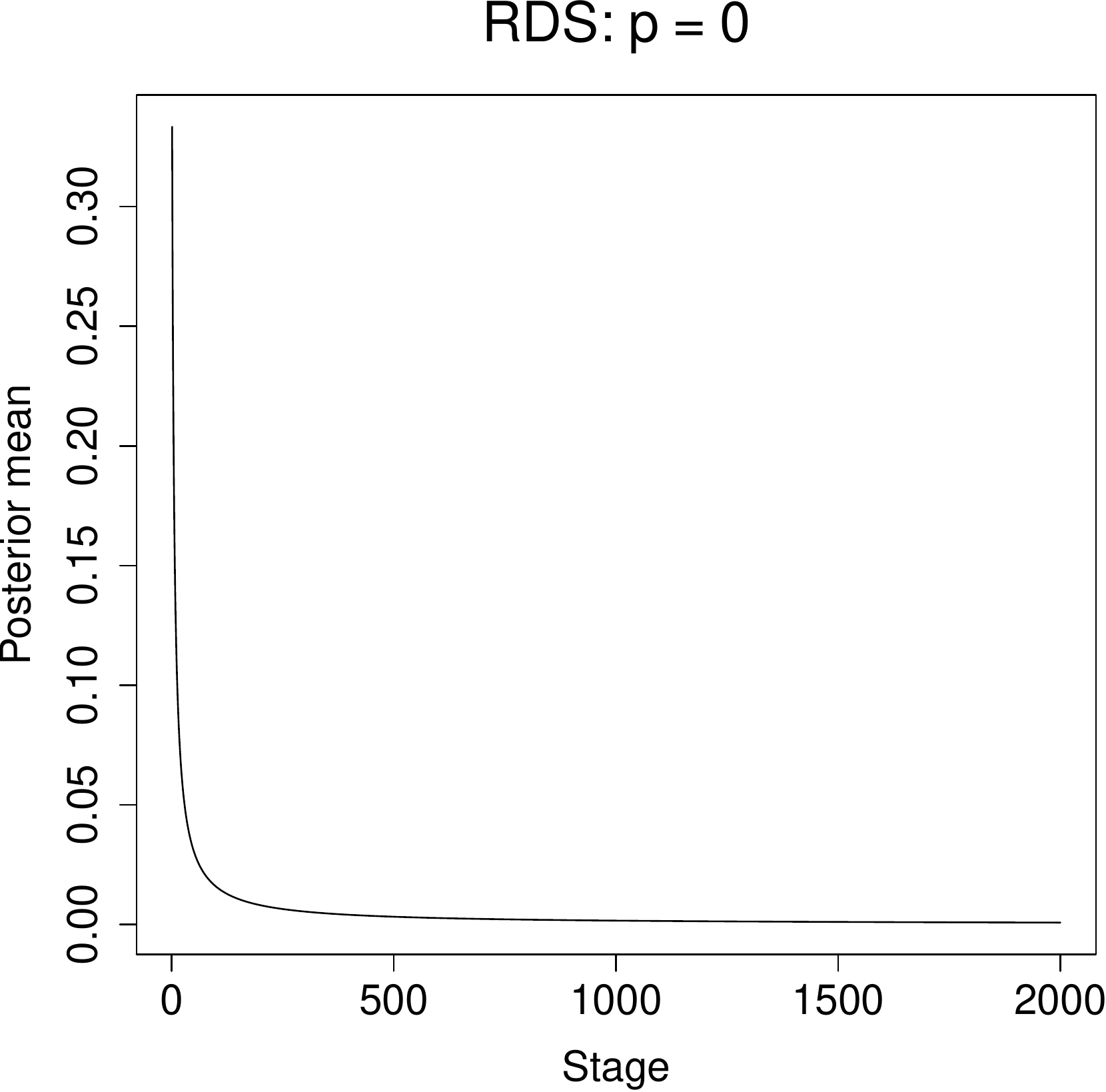}}
\hspace{2mm}
\subfigure [Divergence.]{ \label{fig:ds2}
\includegraphics[width=6cm,height=5cm]{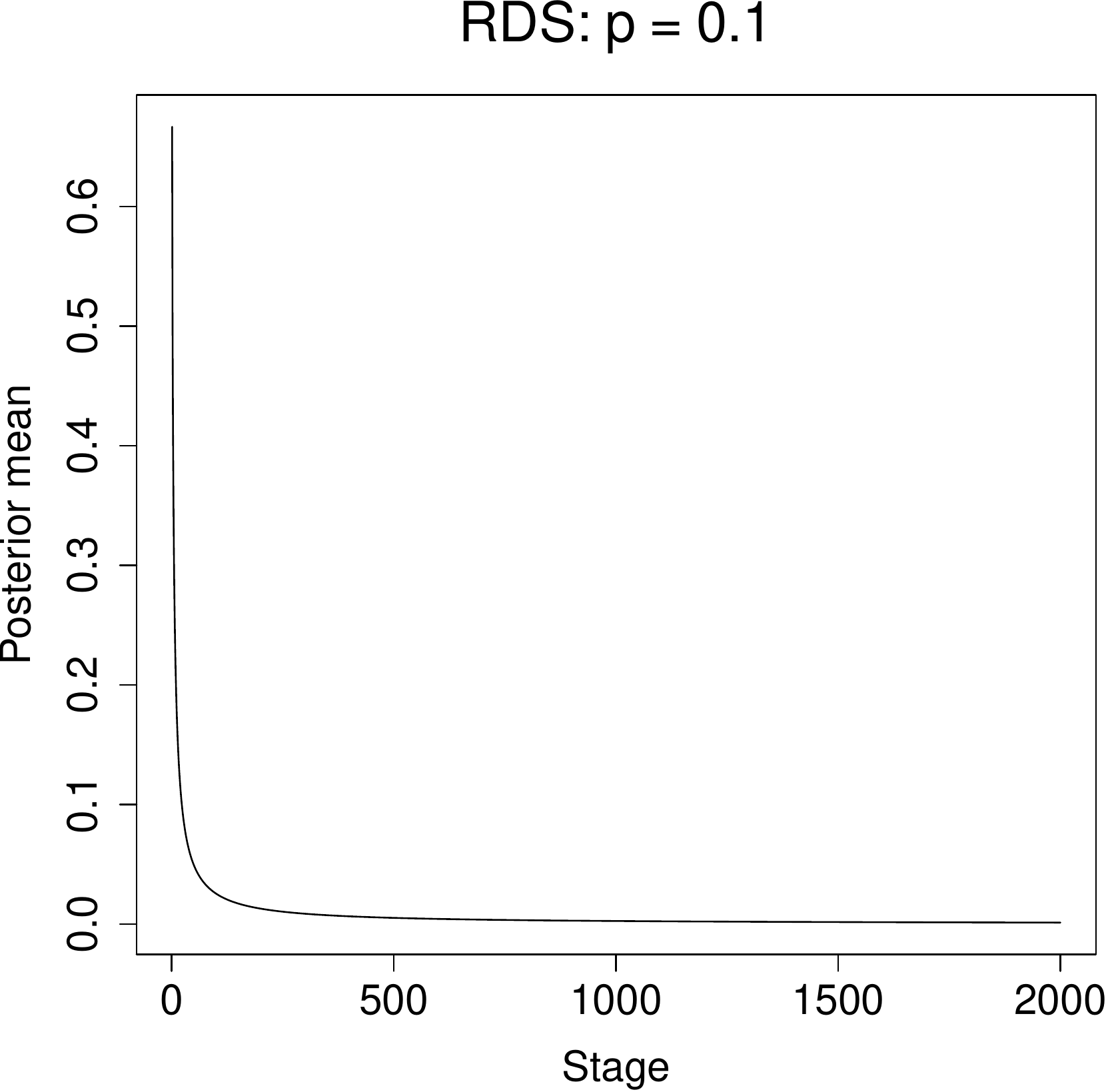}}\\
\subfigure [Divergence.]{ \label{fig:ds3}
\includegraphics[width=6cm,height=5cm]{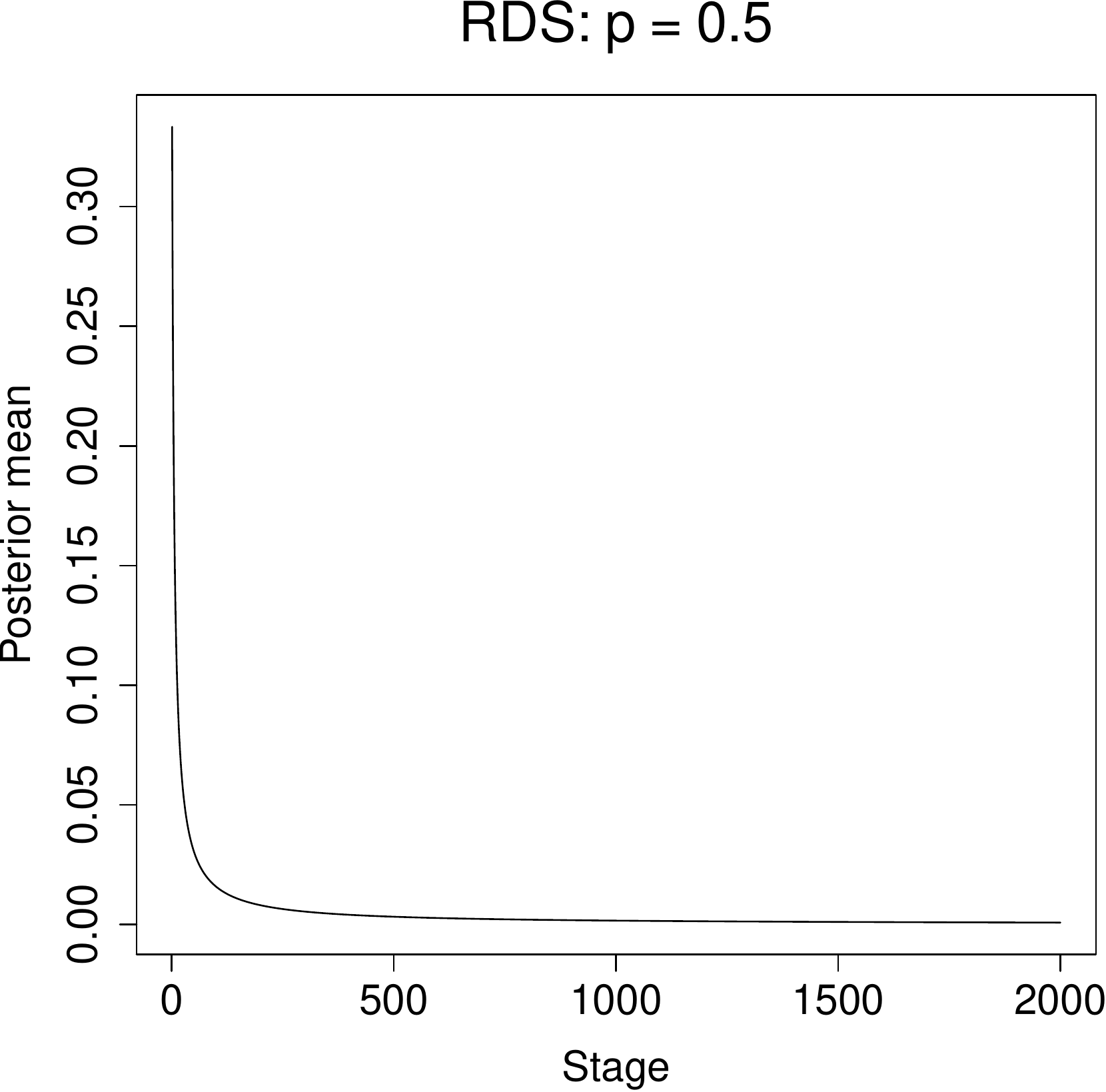}}
\hspace{2mm}
\subfigure [Divergence.]{ \label{fig:ds4}
\includegraphics[width=6cm,height=5cm]{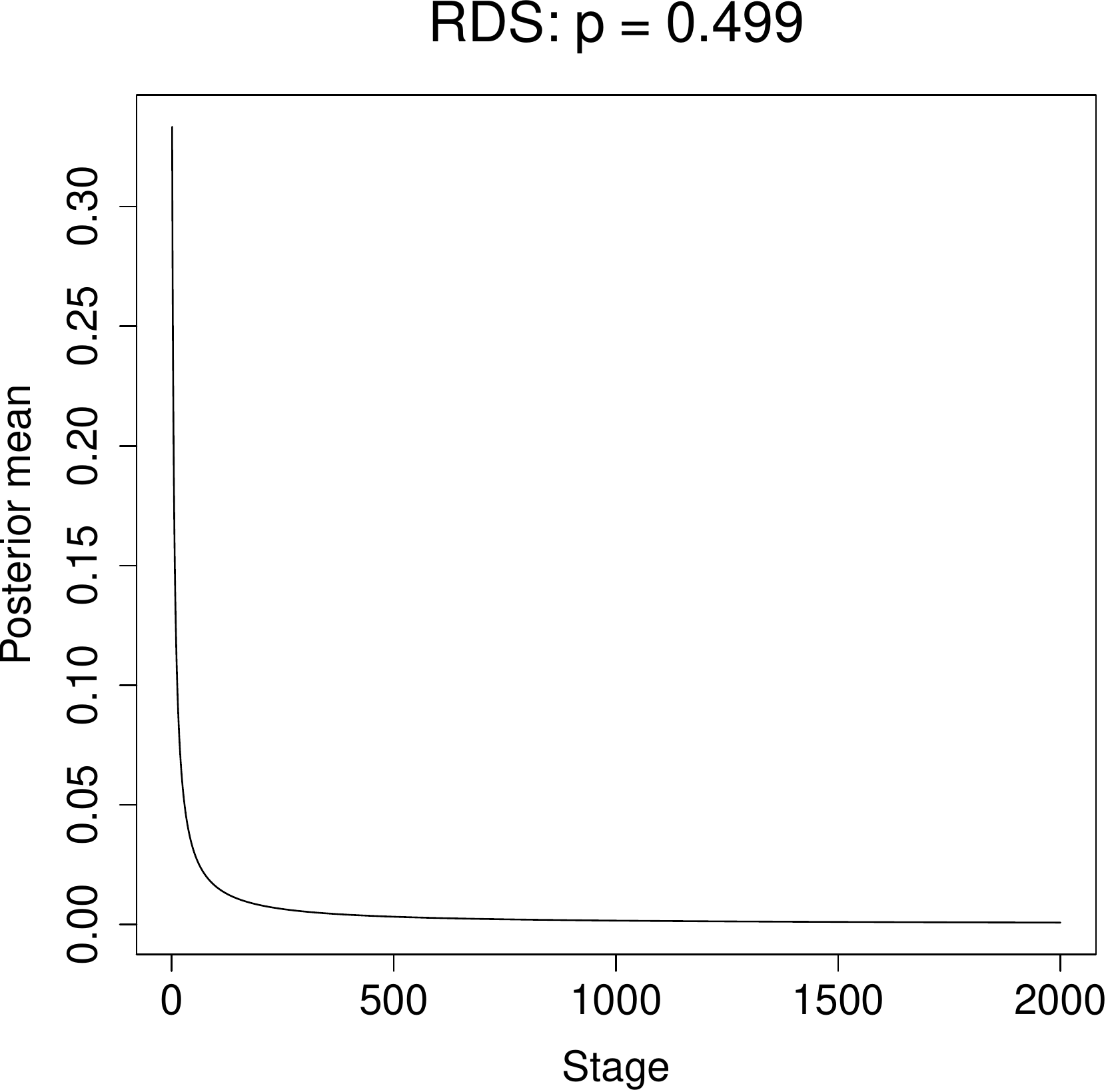}}\\
\subfigure [Convergence.]{ \label{fig:ds5}
\includegraphics[width=6cm,height=5cm]{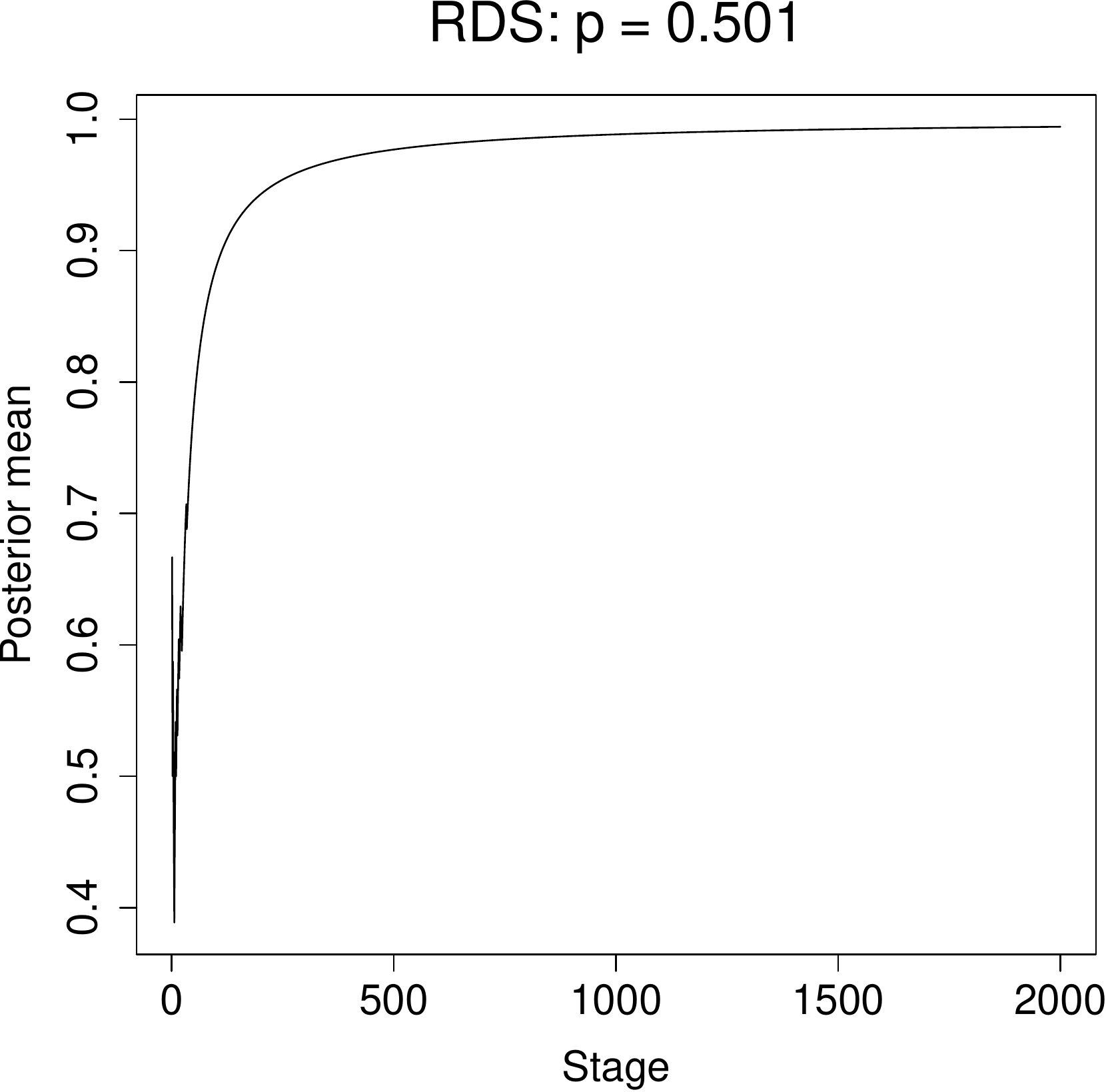}}
\hspace{2mm}
\subfigure [Convergence.]{ \label{fig:ds6}
\includegraphics[width=6cm,height=5cm]{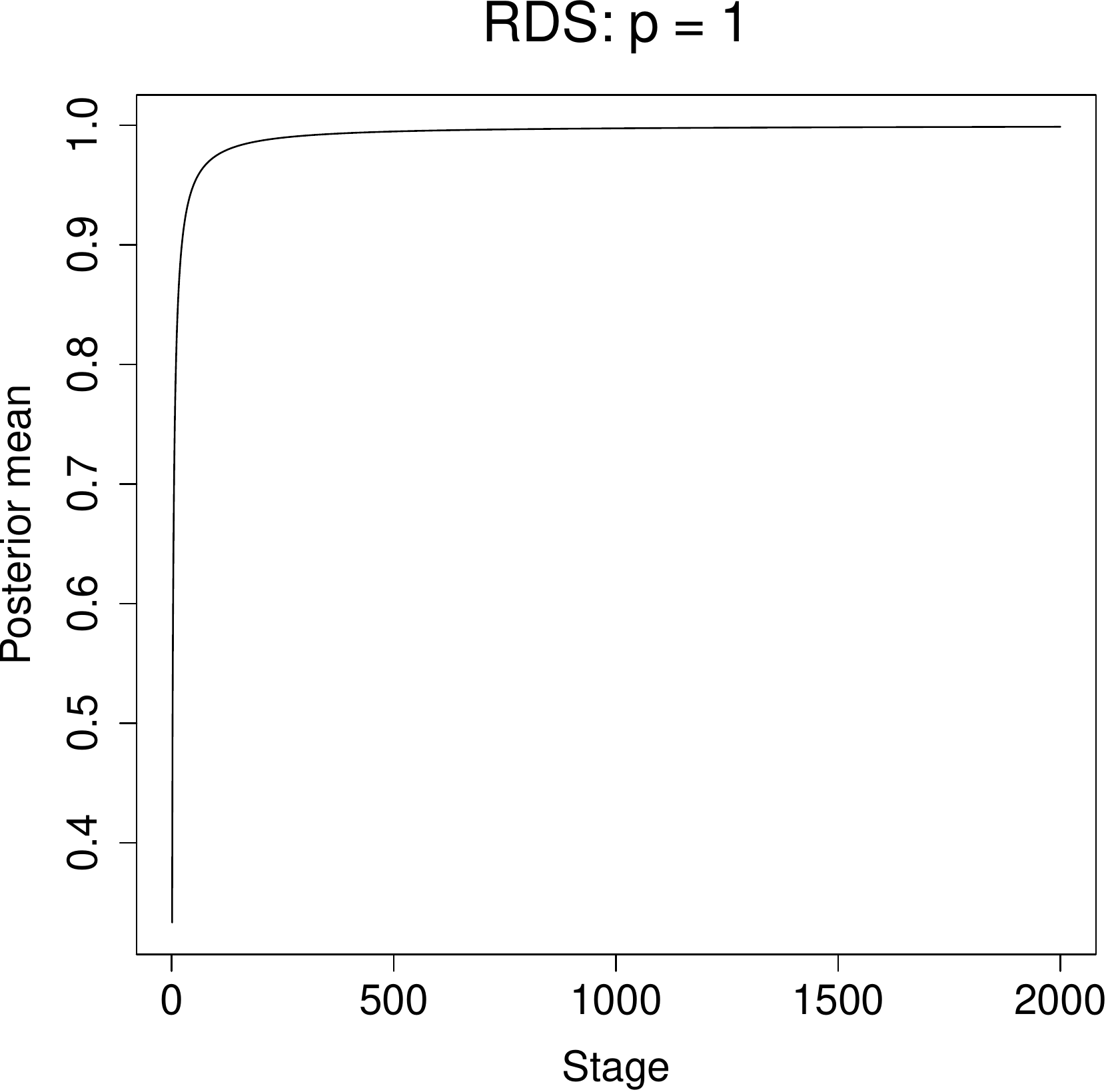}}\\
\subfigure [Divergence.]{ \label{fig:ds7}
\includegraphics[width=6cm,height=5cm]{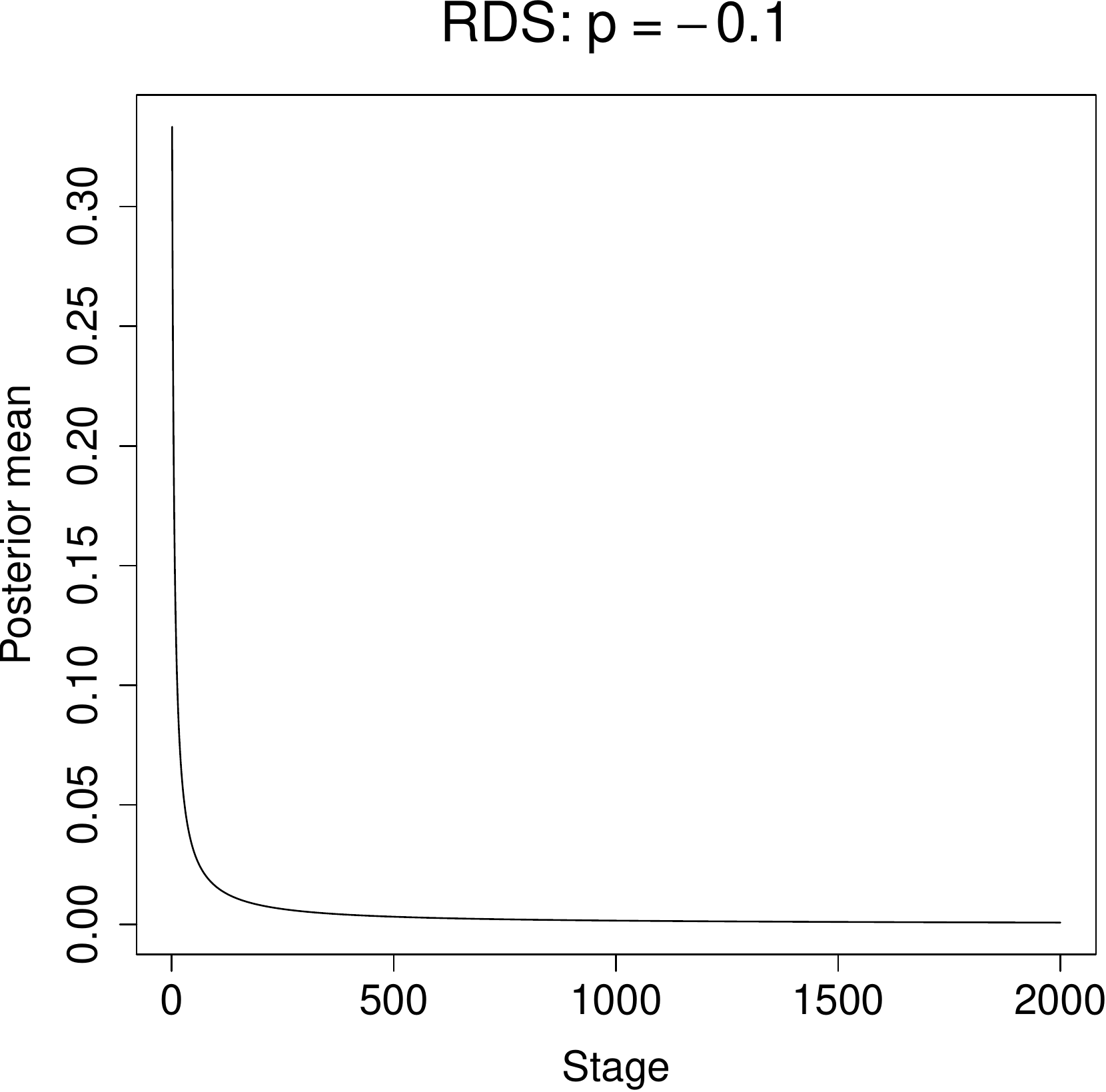}}
\caption{Example 6 revisited: Convergence and divergence for RDS.}
\label{fig:ds}
\end{figure}

\section{Application of random series convergence diagnostics to global climate change}
\label{sec:gw}

\subsection{Future global warming investigation}
\label{subsec:gw}
Global climate change, or gradual increase of the earth's average surface temperature, 
is arguably the most important issue plaguing the environmental scientists all over the world. 
Overwhelmingly strong evidence from various data sources have led the U.S. Global Change Research Program, the
National Academy of Sciences, and the Intergovernmental Panel on Climate Change (IPCC) to
declare that global warming in the recent decades is unquestionable.

Such a concern is supported by the HadCRUT4 observed near surface average global monthly temperature dataset during the years
$1850$ -- $2020$, available from the IPCC website; see \url{https://www.metoffice.gov.uk/
hadobs/hadcrut4/data/current/download.html}. But since the year $2020$ is still ongoing, data points for the last few years seem somewhat doubtful to us, and hence
we consider the monthly dataset in the range $1850 - 2016$ (see also \ctn{Chatterjee20} who analyzed the annual dataset). 
This dataset is only a record of temperature anomalies in degree celsius relative to the years $1961 - 1990$, while we prefer the actual temperatures. 
As in \ctn{Chatterjee20}, we convert this anomaly data to (approximate) actual temperature data by adding $14\degree$C to the anomalies, where
$14\degree$C is the most widely quoted value for the global average temperature for the $1961 - 1990$ period (see \ctn{Jones99} for the detailed development).
The IPCC website also provides $100$ replications of the monthly HadCRUT4 data. Since these replications have very little variation we amalgamate these with the best estimate
of the monthly global average temperature time series, to obtain a temperature time series for the $1850 - 2016$ period consisting of $167\times12\times100$ observations.
A plot of the data is provided in Figure \ref{fig:current}.
\begin{figure}
\centering
\includegraphics[width=10cm,height=8cm]{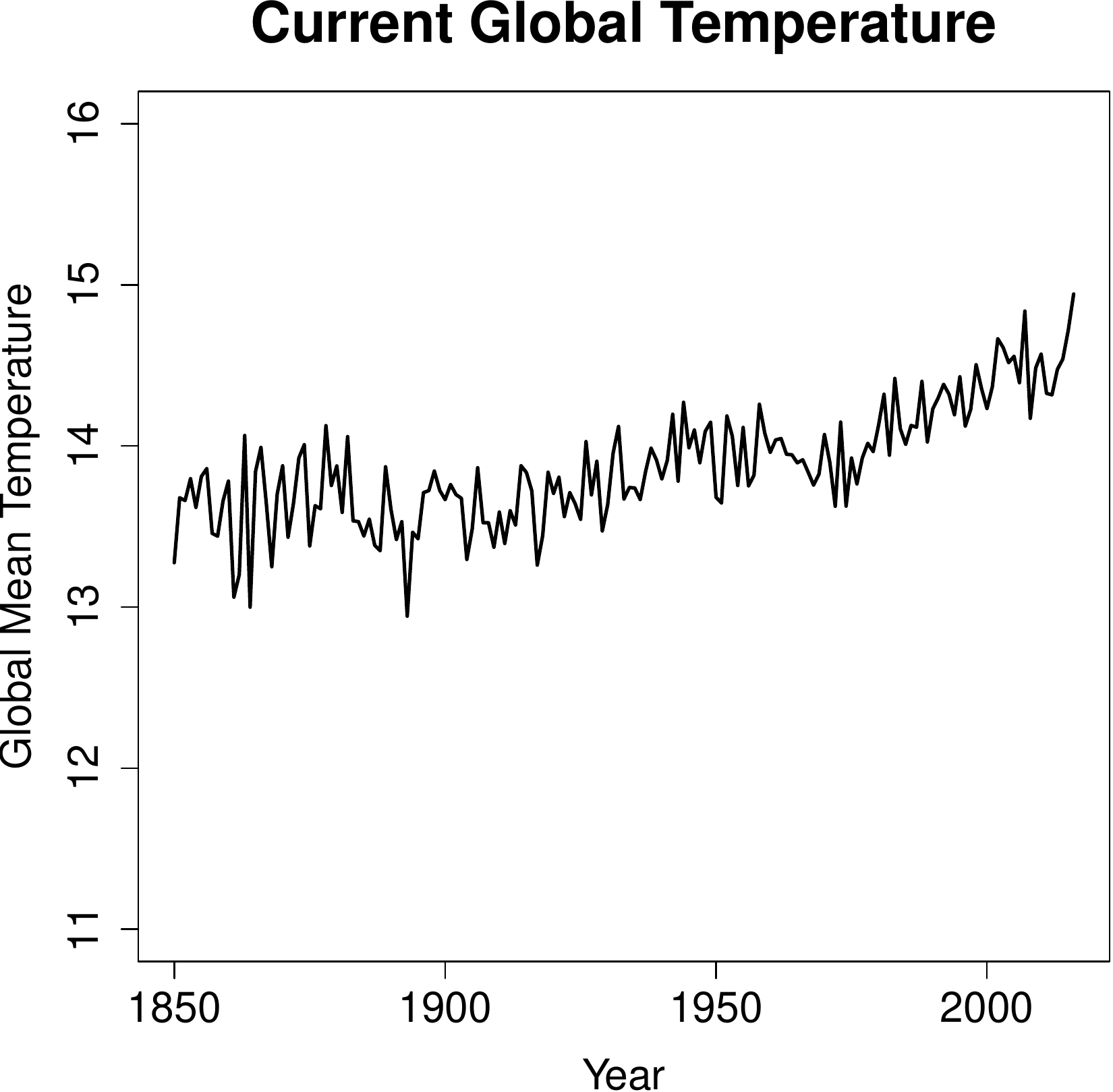}
\caption{Current, HadCRUT4 global mean temperature data.}
\label{fig:current}
\end{figure}

The dataset displayed in Figure \ref{fig:current} is not inconsistent with the IPCC records that compared to the pre-industrial baseline
$1850 - 1900$, the $2009 - 2015$ time period was warmer by about $0.87\degree$C, and that each decade is
getting warmer by about $0.2\degree$C. Such an alarming rate of increase is (arguably) unprecedented, 
and continuation of such global warming may threaten life on earth in the future.

Thus, it is important to investigate if global warming will continue even in the future or if the temperature can be expected to ``stabilize" in the near future around some 
value that does not threaten our existence on earth. Letting $X_t$ denote global monthly average temperature at time point $t$, and $\theta_0$ denote the temperature
around which $X_t$ is expected to concentrate for sufficiently large $t$, one may investigate convergence of the series $\sum_{t=1}^{\infty}Y_{\theta_0,t}$, 
where $Y_{\theta_0,t}=X_t-\theta_0$, or any other bijective transformation of $X_t$.
Convergence of the series would imply that $X_t\rightarrow\theta_0$, as $t\rightarrow\infty$. In contrast, if the series diverges, then either global warming will continue
or even if $X_t\rightarrow\theta_0$, as $t\rightarrow\infty$, the convergence would be much slower compared to the series convergence situation. 
Hence, in the case of divergence, stability can not be achieved in the near future. 

Now, mean global temperature can not be assumed to be an unbounded quantity: even though Figure \ref{fig:current} shows a clearly increasing trend in the recent decades,
it ceratainly must have an upper bound (say, $U$), and a lower bound (say, $L$) is even more obvious. 
Hence, if $\sum_{t=1}^{\infty}Y_{\theta_0,t}=\infty$ for all $\theta_0\in[L,U]$ then $X_t$ will not stabilize at any reasonable temperature value in the near future.
This would also imply that global average temperature will randomly oscillate around various temperature values in the near future, 
ranging from hot to cold, and neither global warming
or global cooling can dominate the climate dynamics in the near future.

For the HadCRUT4 data shown in Figure \ref{fig:current}, we set $L=11\degree$C and $U=16\degree$C, and consider the transformation 
$Y_{\theta_0,t}=\log(\log(X_t))-\log(\log(\theta_0))$. Hence, for all $\theta_0\in[L,U]$, $Y_{\theta_0,t}\in(-1,1)$. 
To implement our Bayesian procedure for random series convergence detection, we first note that there exists no standard model to represent the highly complex
global climate dynamics. Thus the nonparametric method of bounding the partial sums using (\ref{eq:ar1_bound3}) is the only option. For $\theta_0$, we divide
the interval $[11,16]$ into equidistant points with common gap $0.1$ between any two consecutive points. Then, for each $\theta_0$ in this grid of points, we
apply our Bayesian procedure with $n_j=1200$ for $j=1,\ldots,K=167$. In each case we obtain $\sum_{t=1}^{\infty}Y_{\theta_0,t}=\infty$, for $\hat C_1\in (0,10)$.
Setting $n_j$ and $K$ to different values did not change the inference in any of the instances. Following the discussion in the previous paragraph, this helps us strongly
conclude that in the near future the earth will not experience either global warming or global cooling. 
This conclusion is broadly consistent with the detailed future Bayesian nonparametric predictions of \ctn{Chatterjee20}.


\subsection{Investigation of past climate stability}
\label{subsec:past}
In Section \ref{subsec:gw} our Bayesian series convergence detection procedure helped us infer that future global warming or cooling is highly unlikely, and also that
stability of the future climate can not be expected. We now investigate if stability, gradual warming or cooling can be expected of climate in the past. If neither is likely,
then this would be consistent with our finding with the future climate dynamics, and would provide insight into general climate dynamics, both past and future.

To this end, we consider the Holocene global mean surface temperature reconstructions $12,000$ years before present by \ctn{Kaufman20}; here ``present" refers to the year 
$1950$. \ctn{Kaufman20} consider $5$ methods of Holocene climate reconstruction, namely, Composite Plus Scale (CPS), Dynamic Calibrated Composite (DCC),
General Additive Model (GAM), Pairwise Comparison (PAI) and Standard Calibrated Composite (SCC). We also consider the average of these $5$ reconstructions, which we
refer to as Average. The reconstructed Holocene temperatures by \ctn{Kaufman20} are available at \url{https://www.ncdc.noaa.gov/paleo-search/study/27330}. The reconstructions
are provided at $100$ years gap since $1950$ to the past $12,000$ years. 
We convert this to a monthly dataset by interpolation provided by the $R$ software function ``approx". Our datasets thus consist
of $144,000$ Holocene temperature reconstruction values. The $5$ reconstructions, along with their average, are displayed in Figure \ref{fig:past}.
\begin{figure}
\centering
\includegraphics[width=10cm,height=8cm]{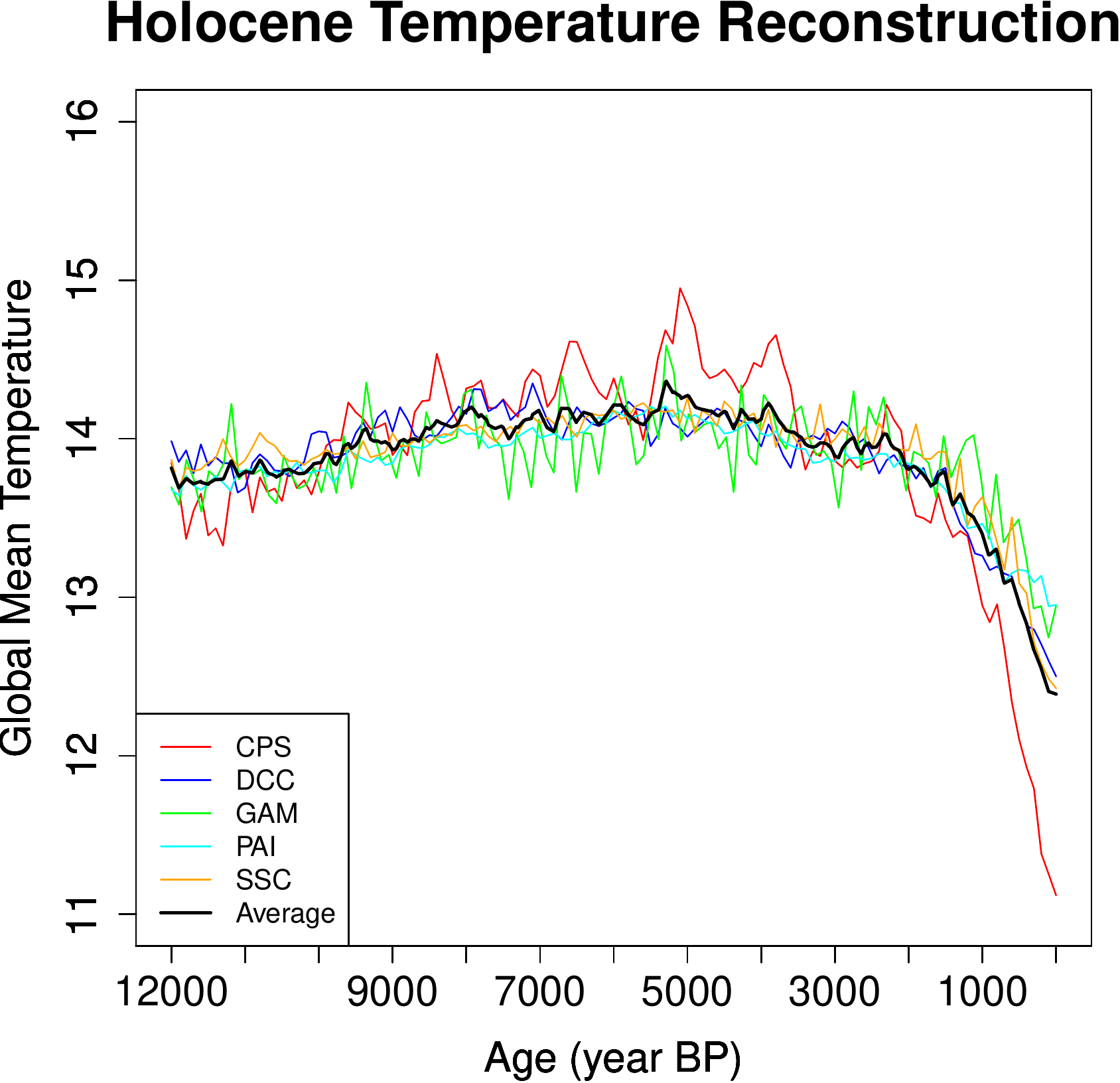}
\caption{Holocene global mean surface temperature reconstructions $12,000$ years before present.}
\label{fig:past}
\end{figure}

To apply our Bayesian method for assessment of convergence in these past climate contexts, we first read the datasets in the reverse order, that is, $\{X_1,X_2,\ldots\}$
now stand for the temperatures during progressively past time points.
Note that the reconstructions around the present (year $1950$) are not quite consistent with the HadCRUT4 temperature around the same year (see Figure \ref{fig:current}).
Hence, such reconstructions are perhaps not unquestionable. However, for investigation of the respective series convergence these are unimportant
since the first finite number of terms in the series do not influence convergence or divergence of the series.

As before, we set $L=11\degree$C and $U=16\degree$C, and consider the transformation $Y_{\theta_0,t}=\log(\log(X_t))-\log(\log(\theta_0))$, where
$\theta_0$ takes values in the grid of points obtained by dividing the interval $[11,16]$ into equidistant points with common gap $0.1$ between any two consecutive points. 
With $n_j=1000$ for $j=1,\ldots,K=144$, and their variations, we obtained $\sum_{t=1}^{\infty}Y_{\theta_0,t}=\infty$, for $\hat C_1\in (0,10)$, with respect to each of the 
$6$ time series shown in Figure \ref{fig:past}. Hence, again we strongly conclude that even Holocene global temperature did not exhibit either of stability, global warming
or global cooling, at least in relatively recent past. This is in keeping with our inference regarding future climate change, and hence allows us to conclude that
climate dynamics is subject to temporary variations, and long-term global warming or cooling is unlikely in the past as well as in the future. 

\section{Summary and discussion}
\label{sec:summary}
Fresh investigation of convergence properties of infinite series is an important undertaking in mathematical analysis, since the existing methods for detecting
convergence and divergence fail for most infinite series. This, along with the seemingly innocuous and informal question of the first author of this article 
regarding ability of the
Bayesian paradigm to address series convergence, stimulated \ctn{Roy17} to develop Bayesian characterization of infinite series that indeed attempts to 
answer such questions of convergence.
Their efforts further led them to valuable insights regarding the celebrated Riemann Hypothesis. 

The key idea of \ctn{Roy17} was to embed the deterministic series within a random, stochastic process framework, and hence their Bayesian characterization
is obviously and directly applicable to random infinite series. Interestingly, their Bayesian procedure is valid irrespective of any dependence
structure among the random elements of the series. In this regard, note that the famous Kolmogorov's three series theorem requires independence among the elements.   

In practice, success of the Bayesian procedure of \ctn{Roy17} depends upon creation of efficient upper bounds for the partial sums. For deterministic
infinite series the authors show how to achieve such bounds by judiciously exploiting the functional forms of the series elements. However, given any random infinite series,
the functional forms of the series elements are of course unknown. For theoretical sake, the marginal distributions of the elements may be assumed known. If the 
series elements are independent, then Kolmogorov's three series theorem is applicable in principle to directly assess convergence, but not in the case of dependence.
Our Bayesian characterization holds in either case, but practical implementation requires bound construction for the partial sums. As we demonstrated
in this article, even for known and simple standard distributions, construction of efficient parametric bounds is a highly non-trivial task. Although 
we could develop mathematically sound parametric upper bounds with non-negative distributional supports of the summands which also
performed very well in our simulation experiments, the method of construction of
valid parametric upper bounds in general setups still eluded us.
The proposed general upper bound (\ref{eq:S2}) 
can not be guaranteed to be a theoretically valid upper bound for arbitrary values of the tuning parameter $a$. Our properly tuned applications of (\ref{eq:S2}) 
to the normal and dependent normal setups indicate
correct results on convergence assessment in most cases, but with enormous sample sizes. Another concern is that in the normal based cases, 
even though the Bayesian algorithm shows 
eventual upward and downward trends for convergence and divergence respectively, it does not tend close enough to $1$ and $0$ even with such large sample sizes and run-times 
to persuasively demonstrate convergence and divergence with (\ref{eq:S2}).
Moreover, for the RDS, wrong convergence results are obtained with the general parametric upper bound in many cases.
A further criticism of the parametric upper bound construction methods is that, 
the forms of $\Psi^{(c)}_i$ and $\Psi^{(d)}_i$ employed 
are too restrictive. 

The aforementioned discussion points towards the requirement for constructing more effective and efficient bounds, reminding that parametric bounds 
can not be constructed in the first place if the underlying distributions are unknown. Indeed, given just the numerical values of the elements of the random series,
formation of parametric bounds for the partial sums seems to be infeasible. Borrowing ideas from \ctn{Roy20} we propose a nonparametric bound structure for 
partial sums of general random series, irrespective of known and unknown distributions. The performance of this nonparametric bound structure depends upon
the choice of the initial value $\hat C_1$ associated with the first iteration of the Bayesian algorithm. Experimentation demonstrates that $\hat C_1=0.71$ and $0.725$
are effective starting values for a wide range of random infinite series. These values are also not much different from those found effective by \ctn{Roy20}
in their wide variety of examples on stochastic processes. It is important to point out that if not much subtlety is required in practice in determination of convergence
properties (such as divergence for $p=1$ but convergence for $p=1+0.001$, many more values of $\hat C_1$ can also be good candidates for our randoms series setup, 
and therefore in practice the Bayesian procedure can exhibit considerable 
robustness with respect to choice of $\hat C_1$. 
To obtain $\hat C_1$ in the RDS context, we have demonstrated how the deterministic Dirichlet series can be exploited for our purpose. 

Our experiments in the random series context with the nonparametric bound structure persuasively 
demonstrate correct detection of convergence
properties with small sample sizes in all the setups, even in quite subtle situations. 
Indeed, our experiments reveal that performance of the nonparametric bound is very much comparable with the valid parametric bounds, whenever the latter are available. 
In the normal and dependent normal setups the nonparametric bound very significantly outperforms the parametric bound in terms of many times smaller sample size,
far greater accuracy and huge computational gains. In the RDS setup, the nonparametric bound gives correct and persuasive results for all the
cases even for small samples, while the parametric bound yields incorrect answers in many cases. Hence, overall the nonparametric bound quite emphatically 
outperforms the parametric bounds.


Although infinite series, both deterministic and random, have been topics of interest since ages, their applications in real data contexts are unheard of. This may be 
due to the reason that real data are always finite while here the topic of discussion is infinite series. However, if assessment of convergence properties is possible
even with finitely many series elements, then there is no reason to stay away from relevant real applications. This is what we attempt in this article. With our
Bayesian procedure, which assesses convergence of the underlying infinite series with only a finite number of series elements, we proceed to address past and future
climate change, a topic of great relevance and importance in the context of the current global warming scenario and climate change debate. 
The key issue that makes random infinite series
applicable to such analysis is that convergence makes the series elements tend to zero and at fast rate. Exploiting this concept and applying our Bayesian procedure
with our nonparametric upper bound for the partial sums on the current global temperature records and Holocene palaeoclimate temperature reconstructions, we obtain results
that help us make interesting inferences regarding general global climate dynamics. Specifically, there does not seem to have been instances of prolonged global warming
or cooling in the past, and nor such adverse climatic conditions are likely to prevail in the future. Indeed, global climate dynamics is subject to temporary variations only, and
the current global warming phenomenon is just an instance of such variation.

%


\normalsize
\bibliographystyle{natbib}
\bibliography{irmcmc}

\end{document}